\documentclass[10pt,a4paper,oneside]{amsart}
\usepackage{amssymb}
\usepackage{hyperref}
\usepackage{multirow}
\usepackage{amscd}
\usepackage{tikz-cd}

\usepackage{xcolor}

\usepackage[a4paper, total={5.4in, 8.48in}]{geometry}

\theoremstyle{plain}
\newtheorem{theorem}{Theorem}[section]
\newtheorem{proposition}[theorem]{Proposition}
\newtheorem{lemma}[theorem]{Lemma}
\newtheorem{corollary}[theorem]{Corollary}
\newtheorem{conjecture}[theorem]{Conjecture}

\theoremstyle{definition}
\newtheorem{definition}[theorem]{Definition}

\theoremstyle{remark}
\newtheorem{remark}[theorem]{Remark}
\newtheorem{example}[theorem]{Example}

 \hypersetup{
	colorlinks=true,
	linkcolor=blue,
	citecolor=blue,
	urlcolor=blue,
} 

\newcommand{\del}[1]{\frac{\partial}{\partial #1}}
\newcommand{\indel}[1]{\partial/\partial #1}

\newcommand{\PP}{\mathbf{P}}
\newcommand{\ZZ}{\mathbf{Z}}
\newcommand{\CC}{\mathbf{C}}
\newcommand{\QQ}{\mathbf{Q}}

\newcommand{\HH}{\mathbf{H}}
\newcommand{\FF}{\mathbf{F}}

\newcommand{\ii}{\mathrm{i}}
\newcommand{\dd}{\mathrm{d}}
\newcommand{\ee}{\mathrm{e}}

\numberwithin{equation}{section}

\title[Integration, models and symmetries of Lins Neto's  foliations]{Explicit integration, birational models and symmetries of Lins Neto's exceptional families of foliations}

\author{Adolfo Guillot}
\address{Instituto de Matem\'{a}ticas, Universidad Nacional Aut\'{o}noma de M\'{e}xico,  \'Area de la Investigación Científica, Ciudad Universitaria, Mexico City 04510, Mexico.}
\email{adolfo.guillot@im.unam.mx}
	
\author{Lu\'{\i}s Gustavo Mendes}
\address{Universidade Federal do Rio Grande do Sul, UFRGS, Brazil}
\email{gustavo.mendes@ufrgs.br}

\author{Wodson Mendson}
\address{ Universidade Federal Fluminense, Instituto de Matemática e
Estatística,  Rua Alexandre Moura~8, São Domingos, 4210-200
 Niterói RJ,
Brazil}
\email{oliveirawodson@gmail.com}

\author{Liliana Puchuri}
\address{Pontificia Universidad Cat\'olica del Per\'u, Per\'u}
\email{lpuchuri@pucp.pe}


\begin{document}

\begin{abstract} In 2002, Lins Neto introduced three remarkable one-parameter families of holomorphic foliations by curves on the complex projective plane, of degrees two, three, and four, whose properties established that the general form of the Poincar\'e Problem had no solution. These foliations are known to be birationally equivalent to certain quotients of linear foliations on abelian surfaces. We give explicit formulas for these birational equivalences, obtaining, in particular, parametrizations of the leaves of the foliations. For the families of degrees three and four, we determine explicit generators for the groups of birational transformations of the projective plane preserving them. For foliations in these families admitting a rational first integral, that is, for those whose parameter is an Eisenstein rational, we give a complete description of the nature and position of the singular points of a generic integral curve, and we present an algorithm that computes the rational first integral explicitly.  Lins Neto's foliations can also be defined over algebraically closed fields of positive characteristic, and, in this setting, we characterize those that are algebraically integrable.   Finally, we study the integrability of the reductions to fields of positive characteristic of some of the non-integrable foliations in the complex family.
\end{abstract}

\maketitle

\setcounter{tocdepth}{1} 
\tableofcontents 
 
\clearpage 
\section{Introduction} 
The problem of the algebraic integration of complex algebraic differential equations in the plane consists of deciding whether a complex polynomial form \(A\,\dd x+B\,\dd y\) in \(\CC^2\) is a rational multiple of the differential of a rational function, this is, if the foliation defined by its kernel has a rational first integral.  After the works of Darboux, Painlev\'e and Autonne on the subject, the question was studied by Poincar\'e near the end of the 19th century. A necessary condition for a foliation with simple singularities to be algebraically integrable is that there exists, around each singular point, a local first integral of the form \(u^s/v\), with~\(s\), the \emph{exponent}, being a rational number.  In \cite{poincare-palermo1}, under the assumption that, for a given foliation, all the negative exponents take the value \(-1\), Poincar\'e gave bounds for the degree of an eventual  first integral, and  expressed his faith on the possibility that the same principles could give results in less particular cases.  The problem of bounding the degree of a potential rational first integral of a foliation in the plane in terms of the exponents of the singular points became known as the \emph{Poincar\'e Problem}. It saw partial advances (including some by Poincar\'e himself \cite{poincare-palermo2}) until, in 2002,  Lins Neto gave three remarkable families of singular holomorphic foliations by curves of the complex projective plane \cite{linsneto}. In each one of these families, with finitely many exceptions, every foliation had non-degenerate singularities with local first integrals, whose exponents were independent of the foliation. Each family contained  foliations with first integrals of arbitrarily large degree, which showed that results like the ones established by Poincar\'e could not take place in full generality: in its most ambitious and general form, the Poincar\'e Problem, as stated above, had no solution. 

The interest in Lins Neto's foliations went beyond its role in the Poincar\'e problem: they gave early examples of foliations in the plane of Kodaira dimension zero \cite[Ex. IV.3.0]{mcquillan}; they are relevant examples of foliations in the plane admitting entire curves tangent to them \cite{brunella-entire}; those of degree two give exceptional fibers of the Baum-Bott map for foliations of degree two in \(\PP^2\) \cite[Cor.~3.3]{guillot-quadratic}, and are associated with an exceptional pair of commuting quadratic homogeneous vector fields~\cite{guillot-exemplesLN};  they are central to the notion of \emph{curvature} of pencils of foliations \cite{lins_neto-pencils}; they are related to the description of infinitely many \((-1)\)-curves in rational surfaces \cite{MP-negative}; \ldots

In this article, we study Lins Neto's families of foliations, and obtain explicit results about them concerning  their birational symmetries and the integration of some of their members.

Lins Neto's three families of foliations are formed by foliations on the complex projective plane \(\PP^2\) having, respectively,  degrees two, three, and four. The one of degree four, \(\mathcal{F}^4\), from which the other two may be constructed through rational quotients, is given by the foliations of degree four that are tangent to the lines of the \emph{dual Hesse configuration}, the one formed by the nine lines  
\begin{equation} \label{arr:hesse} (x^3-y^3)(y^3-z^3)(z^3-x^3)=0,\end{equation}
and their points of intersection. These foliations form a one-parameter family, parame\-trized by \(\alpha\in\PP^1\): the foliation \(\mathcal{F}_\alpha^4\)  on \(\PP^2\) is given  by the kernel of the form
\begin{equation}\label{eq:f4linsneto}
(y^3-z^3)(yz-\alpha x^2)\dd x+(z^3-x^3)(xz-\alpha y^2)\dd y+(x^3-y^3)(xy-\alpha z^2)\dd z.\end{equation} 
The configuration (\ref{arr:hesse}) contains twelve points. Each line of the configuration passes through four points, and there are three lines through each one of them. A generic foliation in \(\mathcal{F}^4\) has 21 singularities: twelve dicritical nodes at the points of the configuration, with local first integrals of the form \(u/v\), and nine saddles, with local first integrals of the form \(u^3v\), one on each line of the configuration (which are \emph{mobile}, as their position depends upon the foliation).

Lins Neto's exceptional foliations were very soon understood to be quotients of linear foliations on some abelian surfaces (\cite[Ex. IV.3.0]{mcquillan},  \cite{guillot-exemplesLN}; see also 
\cite[Ex.~8.4]{brunella-birational}, but take into account  Remark~\ref{rem:bru}).  For  \(\mathcal{F}^4\), this description goes as follows. Let \(\omega\) be a primitive cubic root of unity. Let \(\Lambda\subset \CC\) be a lattice invariant by multiplication by~\(\omega\), and let \(E=\CC/\Lambda\). On \(E\times E\), every \emph{linear} (or \emph{Kronecker})  foliation, one induced by a family of parallel lines on \(\CC^2\), is invariant under the diagonal action  of multiplication by~\(\omega\). \emph{The  quotient of \(E\times E\) under the diagonal action of \(\omega\) is  birationally equivalent to the projective plane \(\PP^2\), through a birational map that maps the images of linear  foliations to those in Lins Neto's family \(\mathcal{F}^4\).}  This explained in a conceptual way many of the phenomena described by Lins Neto, and allowed, for example, to establish straightforwardly that the foliation \(\mathcal{F}_\alpha^4\) of Eq.~(\ref{eq:f4linsneto}) has a rational first integral if and only if  \(\alpha\in\QQ(\omega)\) (the degrees of these first integrals were later established in \cite{puchuri}). Our first result gives the first explicit form of this birational model.

\begin{theorem} \label{thm:ln-desc} Let \(\Lambda\subset \CC\) be the period lattice of the Weierstrass elliptic function \(\wp\) such that \((\wp')^2=4\wp^3+1\) (\(\Lambda\) has the order-three symmetry \(\omega\Lambda=\Lambda\)). Let \(E=\CC/\Lambda\). The map \(\Phi\colon E\times E\dashrightarrow\PP^2\) given by
 \begin{multline*}
\Phi (u,v)  = (  4   \wp^2(u)   \wp^3(v)+6   \wp^2(v)   \wp'(u)+4   \wp^2(u)+6   \wp^2(v) :\\ :
  2\wp(u)[2   \wp(u)   \wp^3(v)-(1+2\omega)   \wp(u)   \wp'(v)-  \wp(u)+(1+2  \omega )   \wp(v)   \wp'(v)-3   \wp(v)]   : \\ :  \wp(u) [4   \wp(u)   \wp^3(v)+2(1+2  \omega )   \wp(u)   \wp'(v)-2   \wp(u)+(1+2  \omega)   \wp(v)   \wp'(u)   \wp'(v)+\\   + 3   \wp'(u)   \wp(v)+(1+2  \omega)   \wp(v)   \wp'(v)+3   \wp(v)] )	
\end{multline*}	
realizes birationally the quotient of \(E\times E\) under the diagonal action of \(\omega\), while mapping the foliation induced by \(\indel{u}+\lambda \indel{v}\) to Lins Neto's foliation   \(\mathcal{F}_{\alpha}^{4}\) for 
\begin{equation}\label{lambda_to_alpha}
	\alpha=(1+2\omega)\lambda+1.
\end{equation}
In particular, when substituting \((u,v)=(t,\lambda t+c)\)  into the above expression, we obtain parametrizations of all the integral curves of \(\mathcal{F}_{\alpha}^{4}\).
\end{theorem}

The parametrizations of the integral curves given by this result may pass through the singular points of the foliation, but, after lifting them to the common resolution of the singularities of the foliations, they realize their universal coverings (with the exception of the finitely many rational leaves of~\(\mathcal{F}^4_\alpha\)).

The explicitness of this result will allow us to investigate the \emph{birational symmetries of the family}~\(\mathcal{F}^4\), this is, the birational transformations of \(\PP^2\) such  that, for every \(\alpha\in\PP^1\), map \(\mathcal{F}^4_\alpha\) to another foliation in the same family. For \(E=\CC/\ZZ[\omega]\),
let \(\rho\colon E\times E\to E\times E\),
\[	\rho(u,v)=(-\omega u,-\omega v),\]
so that the diagonal action of \(\omega\) on \(E\times E\) is the action of \(\rho^2\). As we shall see in Proposition~\ref{id:groups}, the affine transformations of \(\CC^2\) that act on the quotient \((E\times E)/\langle\rho^2\rangle\) form a semidirect product \(\mathrm{GL}(2,\ZZ[\omega])\ltimes (\ZZ/3\ZZ)^2\), where the factor \((\ZZ/3\ZZ)^2\) is generated by the translations by \((0,\tau_0)\) and \((\tau_0,0)\), for the third-of-a-period \(\tau_0=\frac{1}{3}(1+2\omega)\). This group acts by birational transformations on every birational model of \((E\times E)/\langle\rho^2\rangle\)  while preserving the family of foliations induced by the linear ones. Proposition~\ref{prop:aff-bir} will establish that these are, in fact, the only birational symmetries of~\(\mathcal{F}^4\). Through  Theorem~\ref{thm:ln-desc}, and
in essentially the same way that produces explicit formulas for Latt\`es maps of \(\PP^1\), 
we will give explicit generators for the group of these birational transformations:

\begin{theorem}\label{thm:bir-f4} The group of birational transformations of \(\PP^2\) preserving the family of foliations \(\mathcal{F}^4\) is, under the birational identification of \(\PP^2\) with \((E\times E)/\langle \rho^2\rangle\) given by Theorem~\ref{thm:ln-desc}, induced by the  action of \(\mathrm{GL}(2,\ZZ[\omega])\ltimes (\ZZ/3\ZZ)^2\) on \(E\times E\). It  is generated by the birational transformations 
\[B_4(x:y:z)= (\omega^2 x+\omega y+z:\omega x+\omega^2 y+z:x+y+z), \]
\[S_4(x:y:z) =  ( y^2-\omega xz:x^2-\omega yz:z^2-\omega xy ),\] 
\begin{multline*}	U_4(x:y:z)  =  
			(  (x y+x   z+\omega  y   z)  (x+y+\omega^2 z)  (x+\omega^2  y+z): \\  
		 :( \omega  x   z+x  y+y   z)  ( \omega^2x+ y+   z)  (x+y+\omega^2   z): \\    
			:(  x   z+  yz+\omega x  y)  (x+\omega ^2  y+z)  ( \omega^2x+ y+   z)),\end{multline*} 
\[L_4(x:y:z)=  ( z:x:y).\]
The first three generate the action of \(\mathrm{GL}(2,\ZZ[\omega])\); the action of the normal factor \((\ZZ/3\ZZ)^2\) is generated by \(L_4\) and by its conjugate \(L_4'=B_4L_4B_4^{-1}\),	\begin{equation}\label{bir:e_conj}L_4'(x:y:z)=(x:\omega^2y:\omega z),\end{equation}
and consists entirely of linear transformations.
\end{theorem}

An alternative generating set of birational transformations for the action, involving only one non-linear transformation, will be given in Proposition~\ref{prop:altgen}. An analogous result for the  family of degree three will be the object of Theorem~\ref{bir:3fam}.

Lins Neto's foliations are central examples for the problem of the integration of algebraic foliations of the plane, and a natural problem is to integrate them, as explicitly as possible, in the many senses of the term. Theorem~\ref{thm:ln-desc} already gives parametrizations for the leaves of the foliations, although, as explicit as it may be, in the cases where there is a rational first integral, it is difficult to determine from it the position and nature of the singular points of a  generic integral curve.  Recently, in \cite{mendes-puchuri-effective}, for the cases where \(\alpha\in\ZZ[\omega]\), an algorithm that determines the  first integrals and the positions and multiplicities of the singularities of the generic integral curve of \(\mathcal{F}_\alpha^4\) was given. Our next result addresses this problem in general.

The twelve points \(\mathcal{P}\) of the dual Hesse configuration (\ref{arr:hesse})  may be grouped by triples: following the notation in \cite{mendes-puchuri-effective}, define 
\begin{align*}
	\mathcal{P}_{1} & : (1:1:1),   (1:\omega:\omega^2), (1:\omega^2:\omega),    \\ 
	\mathcal{P}_{\omega} & : (1:\omega:1),  (1:1:\omega),  (\omega:1:1),   \\
	\mathcal{P}_{\omega^2} &: (1:1:\omega^2), (\omega^2:1: 1), (1:\omega^2:1) ,   \\ 
	\mathcal{P}_\infty  & : (0:0:1),  (0:1:0), (1:0:0).  
\end{align*} 
Each line of the configuration contains one point of each of these blocks, and the nine lines of the configuration pass  by triples trough the three points of each block.

\begin{theorem}\label{thm:nodes} Let \(\alpha\in\QQ(\omega)\), \(\alpha\notin\{1,\omega,\omega^2\}\). For the \emph{slope} \(\lambda\) of 	\(\mathcal{F}_{\alpha}^{4}\), 
\begin{equation}\label{eq:slope}
	\lambda={\textstyle \frac{1}{3}}(1+2\omega)(1-\alpha),\end{equation}
if \(\lambda=p/q\), with \(p\) and \(q\) relatively prime elements of \(\ZZ[\omega]\), for the generic integral curve of  \(\mathcal{F}_{\alpha}^{4}\),   its singularities are  ordinary multiple  points of multiplicity
	\begin{itemize}
		\item \(N(q)\), at each one of the points in \(\mathcal{P}_{\infty}\),
		\item \(N(p)\), at each one of the points in  \(\mathcal{P}_1\),
		\item \(N(p+\omega^2q)\), at each one of the points in \(\mathcal{P}_{\omega}\), and
		\item \(N(p-\omega q)\), at each one of the points in \(\mathcal{P}_{\omega^2}\),
	\end{itemize}
where, for \(m+n\omega\in \ZZ[\omega]\), \(N(m+n\omega)\) denotes its norm, \(n^2-nm+m^2\).
\end{theorem}
Through this result we recover, under a slightly different formulation, the main result of \cite{puchuri}:
\begin{corollary}\label{coro:ln_degree} Let \(\alpha\in\QQ(\omega)\), \(\alpha\notin\{1,\omega,\omega^2\}\). With the notations of Theorem~\ref{thm:nodes}, the degree of the first integral of \(\mathcal{F}_{\alpha}^{4}\) is
	\[N(p)+N(q)+N(p-\omega q)+N(p+\omega^2q).\]
\end{corollary}

An analogue of these two results for Lins Neto's family \(\mathcal{F}^3\) of degree three will be given in Theorem~\ref{node:f3}.

Another natural problem about the integration of these foliations is to give explicit first integrals for \(\mathcal{F}_{\alpha}^{4}\) for \(\alpha\in\QQ(\omega)\). One approach for it  is to   determine the degree of the first integral, as was done in~\cite{puchuri}, for, once this is known, then, as mentioned by Poincar\'e, the problem of explicitly finding the first integral reduces to a purely algebraic one \cite{poincare-palermo1}.  Another approach consists in finding a birational transformation that, for a given \(\alpha\in\QQ(\omega)\),  maps the foliation \(\mathcal{F}_{\alpha}^{4}\)  to a reference integrable foliation, for then a first integral of the latter may be pulled-back by the birational transformation to obtain a first integral of the former. This was the approach recently followed in \cite{mendes-puchuri-effective}, where this was successfully carried out algorithmically for all \(\alpha\in\ZZ[\omega]\cup 1/\ZZ[\omega]\). We here continue in this order of ideas. The explicit nature of Theorem~\ref{thm:bir-f4} will allow us, in Section~\ref{sec:algo}, to give an algorithm that produces, for every \(\alpha\in\QQ(\omega)\), a rational first integral for \(\mathcal{F}_{\alpha}^{4}\), in the following way. As a consequence of Proposition~\ref{id:groups}, for \(\alpha\in\QQ(\omega)\), there exists a birational transformation of \(\PP^2\) that maps \(\mathcal{F}_\alpha^4\) into \(\mathcal{F}_\infty^4\) (Corollary~\ref{coro:trans_rational}). Furthermore, this transformation may be obtained algorithmically from~\(\alpha\). By pulling back through it a rational first integral of~\(\mathcal{F}_{\infty}^{4}\), we obtain a rational first integral for~\(\mathcal{F}_{\alpha}^{4}\).

Many of our results are, in essence, algebraic, and carry on to algebraically closed fields of positive characteristic.	 The dual Hesse configuration (\ref{arr:hesse}) makes sense for every field having a primitive cubic root of unity, and, on such a field, the most general foliation of degree four of the projective plane  that is tangent to this configuration is still the foliation \(\mathcal{F}_\alpha^4\)  given  by the kernel of the form (\ref{eq:f4linsneto}). As we will see in Theorem~\ref{sec:main}, the description of the foliations in \(\mathcal{F}^4\) of Theorem~\ref{thm:ln-desc} is still valid in positive characteristic. We  study the algebraic integrability  of the foliations in the family, which turns out to be governed by the congruence class modulo~\(3\) of the characteristic of the field:
	
\begin{theorem}\label{thm:charp}
Let \(k\) be an algebraically closed field of characteristic \(p>3\). For the foliation  \(\mathcal{F}^4_\alpha\)   on \(\PP^2_k\), defined, for \(\alpha\in k\cup\{\infty\}\), by  the kernel of the 	form (\ref{eq:f4linsneto}), we have that
\begin{itemize}
\item for \(p\equiv 2\pmod{3}\), \(\mathcal{F}^4_\alpha\) has a rational first integral for every \(\alpha\in k\cup\{\infty\}\).
\item for \(p\equiv 1\pmod{3}\), \(\mathcal{F}^4_\alpha\) has a rational first integral  if and only if \(\alpha\in\FF_p\cup\{\infty\}\).
\end{itemize}
\end{theorem}
	
Moreover, as we shall see in Theorem~\ref{thm:p-div}, in the second case, if \(\alpha\notin\FF_p\), every irreducible  invariant algebraic curve  of \(\mathcal{F}^4_\alpha\) is one of the nine lines of the dual Hesse configuration.

While most of our results concern the foliations in the families of degree three and four, they can be easily adapted to produce analogous results for the foliations in the family of degree two.

The authors thank Serge Cantat, Charles Favre and Frank Loray for enlightening conversations on the subject.

\section{The structure of the foliations} 

In this section we present the other two families of foliations   introduced by Lins Neto, \(\mathcal{F}^2\) and~\(\mathcal{F}^3\), and  describe  \(\mathcal{F}^4\) and \(\mathcal{F}^3\) by bringing into them and  further exploiting the ideas behind the study of \(\mathcal{F}^2\) carried out in~\cite{guillot-exemplesLN}. The idea is to lift the foliations to an explicit cover of \(\PP^2\) where they become true vector fields (and not just foliations), that can be integrated. Theorem~\ref{thm:ln-desc} will result from this analysis, and will be proved in Section~\ref{sec:deg4}. We refer the reader to the first chapters of \cite{brunella-birational} for background material on singular holomorphic foliations. We begin by recalling some facts about the Weierstrass elliptic functions that will be used (see~\cite[Ch.~7]{ahlfors} for details).

\subsection{The Weierstrass elliptic functions of equianharmonic lattices} \label{sec:weiers}

Let \(\Lambda\subset \CC\) be a lattice which is invariant under multiplication by the primitive third root of unity \(\omega=\ee^{2\ii\pi/3}\), this is, \(\Lambda=\theta \ZZ[\omega]\) for some \(\theta\in\CC^*\).   The Weierstrass  elliptic function \(\wp\) of \(\Lambda\) satisfies the differential equation 
\begin{equation}\label{eq:weies}
	(\wp')^2=4\wp^3-g_3,\end{equation}
for some \(g_3\in\CC^*\). Reciprocally, for every \(g_3\in \CC^*\), the solution to (\ref{eq:weies}) is an elliptic function whose lattice of periods  has this symmetry. We have 
\begin{align} 
	\wp(-u) & =\wp(u), & \wp'(- u) & =-\wp'(u), \label{sym:two} \\
\wp(\omega u) & =\omega \wp(u),  &  \wp'(\omega u) & =\wp'(u); \label{sym:three}
\end{align}
the addition formulas for \(\wp\) and \(\wp'\) are:
\begin{align} 
	\label{weies-sum}
	\wp(u+v) & =\frac{1}{4}\left(\frac{\wp'(u)-\wp'(v)}{\wp(u)-\wp(v)}\right)^2-\wp(u)-\wp(v),\\
	\label{weiesp-sum}
	\wp'(u+v) & =\frac{\wp'(u)\wp'(v)(\wp'(u)-\wp'(v))}{(\wp(u)-\wp(v))^3}-3\frac{\wp^2(u)\wp'(v)+\wp^2(v)\wp'(u)}{(\wp(u)-\wp(v))^2}.
\end{align}

The elliptic curve \(E=\CC/\Lambda\)  has the automorphism of order six induced by \(\rho_0(u)= -\omega u\). On~\(E\), the function \(\wp\) has a double pole at \(0\) and two simple zeros.  From (\ref{sym:two}), if \(\zeta_0\) is a zero of \(\wp\), so is \(-\zeta_0\). By (\ref{sym:three}), the zeros of \(\wp\) are permuted by the order-three transformation \(\rho_0^2\colon u\mapsto \omega^2 u\), and,  since there are only two of them, they must be fixed by \(\rho_0^2\). If   \(\Lambda=\theta\ZZ[\omega]\), the two fixed points of \(\rho_0^2\) other than \(0\) are the thirds-of-a-period \(\frac{1}{2}(1+2\omega)\theta\) and \(\frac{1}{2}(2+\omega)\theta\), and these   are thus the zeros of \(\wp\) on~\(E\).    If \(\wp(\zeta_0)=0\), \((\wp'(\zeta_0))^2=-g_3\), and
\begin{align*} 
\wp(u+\zeta_0) & =-\frac{g_3+\wp'(u)\wp'(\zeta_0)}{2\wp^2(u)},\\
\wp'(u+\zeta_0) & =\wp'(\zeta_0)+\frac{g_3(\wp'(u)-\wp'(\zeta_0))}{\wp^3(u)}.
\end{align*}

\begin{remark}	\label{rmk:Ealg}
Some of these facts may be expressed algebraically. Let \(E\) be the complete elliptic curve  associated to the affine elliptic curve  \(E_0\subset\CC^2\) of equation \(\zeta^2=4\xi^3-g_3\). The curve has a derivation~\(D_E\), for which \(D_E(\xi)=\zeta\) (and thus \(D_E(\zeta)=6\xi^2\)), and the order-six symmetry  \begin{equation}\label{sym1_alg}(\xi,\zeta)\mapsto(\omega \xi,-\zeta).\end{equation} 
All these objects are purely algebraic. The connection with the previous description of \(E\) is given by its uniformization via the Weierstrass elliptic functions   \(u\mapsto(\wp(u),\wp'(u))\), by means of which \(D_E\)  is the image of \(\indel{u}\). The symmetry (\ref{sym1_alg}) follows from the identities \(\wp(-\omega u)=\omega\wp(u)\) and \(\wp'(-\omega u)=-\wp'(u)\), resulting from (\ref{sym:two}) and~(\ref{sym:three}).
\end{remark}

\subsection{The family of degree four}\label{sec:deg4}
Consider the quartic homogeneous polynomial vector fields on~\(\CC^3\)
\begin{align*}
X_4 & = x(x^3-2y^3-2z^3)\frac{\partial}{\partial x} +  y(y^3-2x^3-2z^3)\frac{\partial}{\partial y} +    z(z^3-2x^3-2y^3)\frac{\partial}{\partial z}, \\
Y_4 & =y^2z^2\frac{\partial}{\partial x}+ x^2z^2\frac{\partial}{\partial y}+x^2y^2 \frac{\partial}{\partial z}.
\end{align*}
Through the projection of its integral curves, the homogeneous vector field \(X_4+3\alpha Y_4\)  
induces,  on~\(\PP^2\), Lins Neto's foliation \(\mathcal{F}_{\alpha}^{4}\), the one given by the kernel of~(\ref{eq:f4linsneto}). 

The   vector fields \(X_4\) and \(Y_4\) commute (their Lie bracket is identically zero), and  have a common first integral in the homogeneous polynomial of degree nine 
\[H=(x^3-y^3)(y^3-z^3)(z^3-x^3),\]
whose vanishing defines the lines of the dual Hesse configuration~(\ref{arr:hesse}). 
We will consider the following basis for the space of linear combinations of \(X_4\) and~\(Y_4\), 
\begin{align*}
	Z_0& ={\textstyle \frac{1}{3}}X_4+Y_4, &
Z_1 & =(1+2\omega)Y_4.
\end{align*}
 
Since
\begin{equation}\label{z0z1linind}
	Z_0\wedge Z_1\wedge\left(x\del{x}+y\del{y}+z\del{z}\right)=(1+2\omega)H\del{x}\wedge\del{y}\wedge\del{z},\end{equation}
these are  linearly independent in the complement of \(H=0\). 

Consider the homogeneous polynomials of degree six
\begin{align*} 
	p_0 & =(x^3-z^3)(y^3-z^3),\\
	p_1 & = -(y^2+yz+z^2)(x^2+xz+z^2)(x^2+xy+
y^2).\end{align*}
For these, for \(j=0,1\), 
\begin{equation}\label{p.1st.int}Z_{1-j}p_j=0.\end{equation}
Let \(q_j=Z_jp_j\), this is, 
\[q_0 =   - (x^3-z^3)(y^3-z^3)(x^3+y^3-2z^3),\]
\begin{multline*}q_1  =  -  (1+2\omega)(x^2+xy+y^2)(y^2+yz+z^2)(x^2+xz+z^2)\times\\ \times  
(x^2y+x^2z+xy^2+xz^2+y^2z+yz^2)
\end{multline*}
(these are homogeneous polynomials of degree nine). Observe that \(p_j\) divides \(q_j\). The commutativity of the vector fields and relation (\ref{p.1st.int})  imply that \(Z_{1-j}q_j=0\):
\[Z_{1-j}q_j=Z_{1-j}(Z_jp_j)=Z_j(Z_{1-j}p_j)=0.\] 
The first integrals of \(Z_{1-j}\) are bound by the relation
\[ q_j^2=4p_j^3+H^2.\]
Let \( \mu=\ee^{2\ii\pi/9}\)  (so that \(\mu^3=\omega\)), and consider the diagonal linear transformation of order nine \(\eta:\CC^3\to \CC^3\),
\begin{equation}\label{quot:ordernine} \eta(x,y,z) = ( \mu x,  \mu y, \mu z).\end{equation}
Let \(\Sigma_4'=H^{-1}(1)\).  It is invariant by \(\eta\),  and the quotient is realized by the  restriction of the projection \(\Pi\colon \CC^3\setminus \{0\}\to \PP^2\) to \(\Sigma_4'\).  The vector fields \(Z_0\) and \(Z_1\) are both multiplied by \(\omega\) under the action of \(\eta\), and the foliation induced by any one of their linear combinations is preserved under this action. The  foliation on \(\PP^2\) induced by the restriction of \(Z_0+ \lambda Z_1\) to \(\Sigma_4'\) is, for \(\alpha\) as in Eq.~(\ref{lambda_to_alpha}),  Lins Neto's foliation~\(\mathcal{F}^4_\alpha\).

Let \(\Sigma_4\) be the quotient of \(\Sigma'_4\) under the action of \(\eta^3\). Since the vector fields \(Z_i\) and  the functions \(p_i\) and \(q_i\) are all invariant under the action of \(\eta^3\), they induce well-defined objects on \(\Sigma_4\), which we will denote by the same symbols. We have an order-three transformation \(\overline{\eta}\colon \Sigma_4\to \Sigma_4\) induced by the action of \(\eta\) on~\(\Sigma'_4\). The quotient of \(\Sigma_4\) under the action of \(\overline{\eta}\) is that of \(\Sigma_4'\) under the action of~\(\eta\).

Consider the algebraic description of \(E\times E\) that follows from that of \(E\) (Remark~\ref{rmk:Ealg}): the projective surface associated to the affine surface on \(\CC^4\) defined, in the coordinates \((\xi_0,\zeta_0,\xi_1,\zeta_1)\), by \(\zeta_i^2=4\xi_i^3+1\) for \(i=0,1\). We consider the order-six symmetry   of \(E\times E\) given by the diagonal action of (\ref{sym1_alg}), 
\begin{equation}\label{rho-alg}
	\rho(\xi_0,\zeta_0,\xi_1,\zeta_1)=(\omega\xi_0,-\zeta_0,\omega\xi_1,-\zeta_1),
\end{equation}
as well as the derivations \(D_0\) and \(D_1\) on \(E\times E\) coming from considering the derivation \(D_E\)   on each factor (that is, \(D_i(\xi_i)=\zeta_i\), \(D_i(\xi_{1-i})=0\)). A birational model for \((E\times E)/\langle\rho^2\rangle\) may be given as follows.  Observe that \(\zeta\colon E\to\PP^1\)  realizes the quotient of the action of \(\rho_0^2\) on \(E\), as it is invariant under the latter and has degree three. In consequence, the map \(E\times E\to \PP^1\times\PP^1\) given by \((\xi_0,\zeta_0,\xi_1,\zeta_1)\mapsto (\zeta_0,\zeta_1)\) is invariant under the action of~\(\rho^2\) on \(E\times E\), and has degree nine. The nine points of its generic fiber form three orbits under the action of~\(\rho^2\), that can be distinguished by  the value that \(\xi_0/\xi_1\) takes on them. Thus, the map \(\phi\colon E\times E\to (\PP^1)^3\) given by \((\xi_0,\zeta_0,\xi_1,\zeta_1)\mapsto (\xi_0/\xi_1,\zeta_0,\zeta_1)\)  realizes the quotient of \(E\times E\) under the action of~\(\rho^2\).  Let \(M=\overline{\phi(E\times E)}\). 

Let \(F\colon\Sigma_4\to E\times E\) be given by 
\[ F(x,y,z)= (p_0,q_0,p_1,q_1).\]
It  is  equivariant with respect to the action of \(\ZZ/3\ZZ\) given by \(\overline{\eta}\) on \(\Sigma_4\) and by~\(\rho^2\)   on \(E\times E\).  Thus, it induces a map   \(\overline{F}\colon\PP^2\to (E\times E)/\langle\rho^2\rangle\)  between the respective quotients, and the following diagram commutes:
\[
\begin{CD}
	\Sigma_4 @>F>> E\times E\\
	@VV{/\overline{\eta}}V @VV{/\rho^2}V\\
	\PP^2 @>\overline{F}>>  (E\times E)/\langle\rho^2\rangle
\end{CD}. \]
The map \(F\) will be a birational isomorphism if and only if so is \(\overline{F}\). We will show that this is indeed the case; the map \(\Phi\) in  Theorem~\ref{thm:ln-desc} will appear as the inverse of~\(\overline{F}\).

In the birational model \(M\) of \((E\times E)/\langle\rho^2\rangle\),  \(\overline{F}\colon\PP^2\dashrightarrow M\) is given by
\begin{equation}\label{inv_of_phi}\overline{F}(x:y:z)=\left(\frac{p_0}{p_1},\frac{q_0}{H}, \frac{q_1}{H}\right).\end{equation}
The birationality of \(\overline{F}\) (or, equivalently, of \(F\)) amounts to the following result:
\begin{lemma}\label{lemma:ratfun} The field of rational functions of \(\PP^2\) is generated by 
\begin{align*}
	\frac{p_0}{p_1} & =-\frac{(x-z)(y-z)}{x^2+xy+
		y^2},\\
	\frac{q_0}{H} & =\frac{x^3+y^3-2z^3}{x^3-y^3}, \\ \frac{q_1}{H}& =-(1+2\omega)\frac{(x^2y+x^2z+xy^2+xz^2+y^
		2z+yz^2)}{(x-y)(y-z)(z-x)}.
\end{align*}
\end{lemma}
\begin{proof}  
	
If the point \(\alpha\) of \(\PP^2\) is not a base point of any one of these three pencils, there is a unique curve from each pencil passing through it.  The statement is equivalent to the fact that these three curves, that are determined by \(\alpha\), actually characterize it.
 
The base points   of the pencil of conics \(p_0:p_1\) are \((1:\omega:1)\), \((1:\omega^2:1)\), \((\omega:1:1)\) and \((\omega^2:1:1)\). These are also among the base points of each one of the pencils of cubics \(q_0:H\) and \(q_1:H\). Let \(j\in\{0,1\}\). If \(\alpha\in\PP^2\)   is not a base point of any of the pencils, there exists a unique \(\beta_j\in\PP^2\), generically different from \(\alpha\) and from the four base points above, such \(\alpha\) and \(\beta_j\) lie in the same element of the pencils \(p_0:p_1\) and \(p_j:H\). The claim amounts to show that, generically, \(\beta_0\neq \beta_1\). For instance,  for \(\alpha=(0:1:2)\), which is not a base point of any of the pencils, we have the different points	\(\beta_0 = (-6:-5:8)\)   and 
\(\beta_1  =(-5:-6:8)\).\end{proof}

Thus, by this lemma, \(F\colon\Sigma_4\to E\times E\) is a birational isomorphism.  It establishes a correspondence  between the functions \(p_0\), \(q_0\), \(p_1\) and \(q_1\)  on \(\Sigma_4\), and~\(\xi_0\), \(\zeta_0\), \(\xi_1\) and~\(\zeta_1\), respectively, on \(E\times E\). Since \(Z_jp_{1-j}=0\) and \(Z_jp_{j}=q_j\), this correspondence identifies \(Z_0\) with \(D_0\) and \(Z_1\) with \(D_1\). We have summarized this in Table~\ref{tab:corrbis}, where, in the third column, we have included the transcendental description in terms of the uniformization of \(E\times E\) via Weierstrass functions.  We conclude that \(\overline{F}\), which is also a birational isomorphism, identifies Lins Neto's foliation \(\mathcal{F}_\alpha^4\)  with the foliation on \((E\times E)/\langle\rho^2\rangle\) induced by the vector field \(D_0+\lambda D_1\) on \(E\times E\) for \(\lambda\) as in (\ref{eq:slope}). 

\begin{table}
	\begin{tabular}{cccc}
	& \multirow{2}{*}{On \(\Sigma_4\)}	&   On \(E\times E\) & On \(E\times E\)  \\ & & (algebraic) & (transcendental) \\
	\hline
	\multirow{4}{*}{functions} & \(p_0\)	&  \(\xi_0\) &\(\wp(u)\) \\
	
	&\(p_1\)	&  \(\xi_1\) &\(\wp(v)\) \\
	
	&\(q_0\)	& \(\zeta_0\) &  \(\wp'(u)\) \\
	&\(q_1\)		& \(\zeta_1\)	&  \(\wp'(v)\) \\ \hline
	
	\multirow{2}{*}{vector fields} & \(Z_0\) & \(D_0\) &  \(\indel{u}\) \\ 
	& \(Z_1\) &  \(D_1\) &\(\indel{v}\) \\     \hline
	order-three symmetries & \(\overline{\eta}\) & \(\rho^2\)  & \(\rho^2\)   \\  \hline 
\end{tabular}
	\caption{Correspondence of objects on \(\Sigma_4\) and  on \(E\times E\) induced by \(F_4\).}\label{tab:corrbis}
\end{table}

From Lemma~\ref{lemma:ratfun}, it is, in principle, possible to express \((x:y:z)\) rationally in terms of the components  of the expression of \(\overline{F}\) in Eq.~(\ref{inv_of_phi}). This can be done explicitly: 
\begin{multline}\label{bir_inverse_of_F}
	(x:y:z)  =  \left(\frac{p_0}{p_1} \left(\left(\frac{q_1}{H}\right)^2+3\right)+6 \frac{p_1}{p_0} \left(\frac{q_0}{H}+1\right) \right. : \\     : \frac{p_0}{p_1}\left( \left(\frac{q_1}{H}\right)^2 -2(1+2 \omega)\frac{q_1}{H}-3\right)+2(1+2 \omega ) \frac{q_1}{H}-6 : \\  :\left. \frac{p_0}{p_1}\left(\left(\frac{q_1}{H}\right)^2+2(1+2 \omega )\frac{q_1}{H}-3\right)+(1+2 \omega ) \frac{q_0}{H} \frac{q_1}{H}+3 \frac{q_0}{H}+(1+2\omega ) \frac{q_1}{H}+3\right).
\end{multline}

The expression \(\Phi\) of Theorem~\ref{thm:ln-desc} is a rewriting of this formula in terms of the Weierstrass functions, following the correspondence in Table~\ref{tab:corrbis}. This establishes Theorem~\ref{thm:ln-desc}. It may, of course, also be established by a direct calculation.

Theorem~\ref{thm:ln-desc} gives the setting for the following result, the first part of which originally appeared in~\cite{guillot-exemplesLN} and in~\cite[Thm.~1]{lins_neto-pencils}.	
	
\begin{corollary}\label{int:cond}The foliation \(\mathcal{F}_\alpha^4\) has a rational first integral if and only if \(\alpha\in \QQ(\omega)\). If \(\alpha\notin \QQ(\omega)\), every irreducible algebraic curve tangent to \(\mathcal{F}_4^\alpha\) is a line of the dual Hesse configuration~(\ref{arr:hesse}).
\end{corollary}

\begin{proof} For the first part, observe that the foliation on \(\PP^2\) will have a rational first integral if and only of the original foliation on \(E\times E\) does, and that this happens if and only if all of its leaves are closed. Since all the leaves are obtained from a single one by translation, the foliation will have a first integral if and only if the leaf  of \(\indel{u}+\lambda \indel{v}\) through \(0\) gives a closed leaf of  the foliation on  \(E\times E\). This happens if and only if the line \(L\) in  \(\CC^2\) which is the orbit of the vector field \( \partial/\partial u+\lambda\partial/\partial v\) through \(0\)  intersects \(\Lambda^2\) along a lattice. If \(L\) intersects \(\Lambda^2\) at a nonzero point \((u,v)\), we have that \(\lambda\in \QQ(\omega)\), and that \(L\)  also intersects \(\Lambda^2\) at the point \((\omega u,\omega v)\), thus intersecting \(\Lambda^2\) along a lattice. Reciprocally, if \(\lambda \in \QQ(\omega)\), and \(\lambda=p/q\) with \(p,q\in \ZZ[\omega]\), both \((q,p)\) and \((\omega q,\omega p)\) belong to the intersection of \(L\) and~\(\Lambda^2\). 
	
For the second, let \(C\subset \PP^2\) be an invariant algebraic curve for \(\mathcal{F}_\alpha^4\) that is not a union of lines of the Hesse arrangement. The preimage of \(C\cap \{H\neq 0\}\) in \(\Sigma_4\) is an algebraic curve, that maps under \(F\) to an algebraic curve in \(E\times E\) tangent to the  foliation generated by \(\indel{u}+\lambda v\indel{v}\) for some \(\lambda\in\CC\). But since this foliation has one closed curve, all of its curves are closed, and thus \(\lambda\)  (and hence ~\(\alpha\) as well) belongs to \(\QQ(\omega)\).
\end{proof}

\subsection{The family of degree three}\label{sec:deg3} Lins Neto's  family \(\mathcal{F}^3\) of foliations of degree three is formed by the foliations on \(\PP^2\) of degree three that are tangent to the configuration of two conics and three lines
\begin{equation}\label{conf_f3}(x^2-4yz)(x^2-yz)(x-y-z)(x-\omega y-\omega^2z)(x-\omega^2 y-\omega z)=0.\end{equation}
For \(\alpha\in\CC\), the foliation  \(\mathcal{F}_{\alpha}^{3}\)  is given by the kernel of $\Omega + \alpha \Theta$, for 
\[\Omega = y(2z^3-x^3+3xyz)\dd x+(x^4-xz^3+2y^2z^2-4yzx^2)\dd y  +y(x^2y-xz^2-2y^2z)\dd z,\] 
\[\Theta = z(x^3-3xyz-2y^3)\dd x+z(2yz^2-x^2z+xy^2)\dd y+(xy^3-2y^2z^2+4yzx^2-x^4)\dd z,\]
while \(\mathcal{F}_{\infty}^{3}\) is given by the kernel of~\(\Theta\).
The curve \(\mathcal{C}\) of Eq. (\ref{conf_f3})  has as irreducible components the three lines 
\begin{align}\label{triangle_f3}
	l_0  & =x-y-z, & 
	l_1  & = x-\omega^2 y-\omega z, &
	l_2 & = x-\omega y-\omega^2z,\end{align}
and the two conics
\begin{align}\label{conics_f3}
	C_1&=x^2-4yz,	 & C_2=x^2-yz.
\end{align}
The first conic is tangent to each one of the three lines, and the  second one passes through their three pairwise intersection points; the conics are   tangent at the points \((0:1:0)\) and \((0:0:1)\), which account thus for all of the points in their intersection;  see Figure~\ref{fig:curve_f3}. Having degree \(7\), the curve \(\mathcal{C}\)  also gives the locus of tangencies of any two foliations in the family; in particular, all the  singular points of every foliation in \(\mathcal{F}^3\) belong to \(\mathcal{C}\). The generic foliation in \(\mathcal{F}^3\) has \(13\) non-degenerate singularities: 
\begin{itemize}
	\item three radial ones (with ratio of eigenvalues \(1:1\)) at the vertices of the triangle \(l_0l_1l_2\);
	\item three at the tangencies of the \(l_i\) with \(C_1\), and two at the two points in \(C_1\cap C_2\), all of them with ratio of eigenvalues \(2:1\); and
	\item five mobile ones, one on each irreducible component of \(\mathcal{C}\), with ratio of eigenvalues \(-6:1\) for the point on \(C_1\), and  $-3:1$ for the four other ones.
\end{itemize}

\begin{figure}
\centering	
\includegraphics[width=0.45\textwidth]{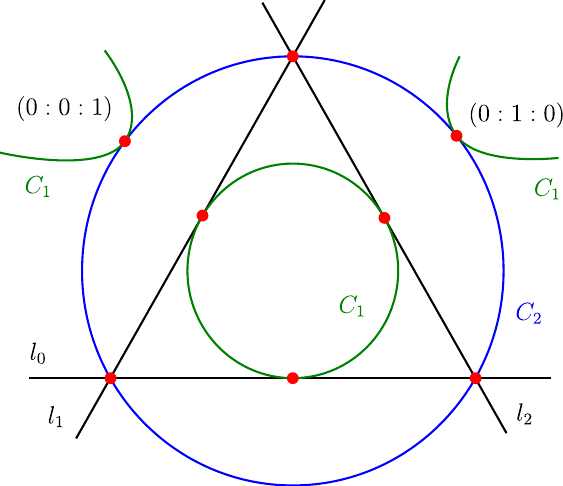}
	\caption{The configuration of curves \(\mathcal{C}\) for the family \(\mathcal{F}^3\), as sketched by Lins Neto. } \label{fig:curve_f3}
\end{figure}

Lins Neto's description of these foliations goes as follows \cite[Sect.~2.3]{linsneto}: the foliation \(\mathcal{F}_{\alpha}^{4}\) is invariant under the linear involution
\begin{equation}\label{sym2}
	J(x:y:z)= (y:x:z).
\end{equation} 
The   rational map \(\tau\)
\begin{equation}\label{quot_propj} 
	\tau(x:y:z)	= (z(x+y):xy:z^2)
\end{equation} 
realizes the quotient, and the foliation \(\mathcal{F}_{\alpha}^{3}\) is the image of \(\mathcal{F}_{\alpha}^{4}\) under this map. (We will be using the same homogeneous coordinates \((x:y:z)\)  in both in the domain and the range of~\(\tau\).)

A description of \(\mathcal{F}^3\) compatible with the one given in Theorem~\ref{thm:ln-desc} for  \(\mathcal{F}^4\)  is given by  following result:	 
\begin{proposition}\label{prop-rel34} 
The birational identification \(\Phi\colon(E\times E)/\langle\rho^2\rangle\dashrightarrow\PP^2\) of Theorem~\ref{thm:ln-desc} is equivariant with respect to the actions of \(\ZZ/2\ZZ\) given  by \(\rho\) on \((E\times E)/\langle\rho^2\rangle\) and by the linear involution \(J\) of Eq.~(\ref{sym2}) on \(\PP^2\). Hence, the map \((E\times E)/\langle\rho\rangle\dashrightarrow\PP^2\) obtained by composing the map \(F\) of Theorem~\ref{thm:ln-desc} with the map of Eq.~(\ref{quot_propj}) is a birational one. It maps the foliation \(\indel{u}+\lambda \indel{v}\) to Lins Neto's foliation \(\mathcal{F}_\alpha^3\) for \(\alpha\) as in Eq.~(\ref{lambda_to_alpha}).
\end{proposition}

\begin{proof} The linear involution of \(\CC^3\), 
	\begin{equation}\label{inv:4to3}
		\sigma(x,y,z)= (-y,-x,-z),
	\end{equation} 
induces the  involution \(J\) of Eq.~(\ref{sym2}) on \(\PP^2\). It	
commutes with \(\eta\), and preserves \(H\) together with its level surface \(\Sigma_4\), thus lifting   \(J\)   to \(\Sigma_4\). The involution \(\sigma\) acts upon the functions \(p_i\) and \(q_i\) by
\begin{align*}
\sigma^*p_0 & =p_0,  &  \sigma^*p_1 & =p_1, \\ \sigma^*q_0 & =-q_0,  & \sigma^*q_1 & =-q_1.
\end{align*}
From relation (\ref{sym:two}) and the correspondences of Table~\ref{tab:corrbis}, this implies that \(F\colon \Sigma_4\to E\times E\) is equivariant with respect to the actions of \(\ZZ/2\ZZ\) given by \(\sigma\) on \(\Sigma_4\), and by the involution of \(E\times E\) induced by \(\rho^3\colon  (u,v)\mapsto(-u,-v)\). These actions commute with the previously described actions of \(\ZZ/3\ZZ\), inducing a birational isomorphism between \(\PP^2/\langle \sigma\rangle\) and \((E\times E)/\langle \rho \rangle\), that maps \(\mathcal{F}_{\alpha}^{3}\) to the foliation induced by \( u\indel{u}+\lambda\indel{v}\) on  \((E\times E)/\langle \rho \rangle\), for \(\lambda\) as in~(\ref{eq:slope}). \end{proof}

We may also give a direct description of \(\mathcal{F}^3\) analogous to the one developed in Section~\ref{sec:deg4} for~\(\mathcal{F}^4\),  and which is, moreover, compatible with the previous description of the foliations in~\(\mathcal{F}^4\).  

Consider the linear involution of \(\CC^3\), \(\nu(x,y,z) =(y,x,z)\), whose action on \(\PP^2\) agrees with that of  \(\sigma\). It  preserves the vector fields \(X_4\) and \(Y_4\). The first integral \(H\) changes sign under the action of  \(\nu\), and, in particular, the image of \(\Sigma_4'\) under \(\nu\) is disjoint from~\(\Sigma_4'\).  The quotient of \(\CC^3\) under the action of \(\nu\) is  birationally  realized  by the  homogeneous rational map \(\Upsilon\colon \CC^3\dashrightarrow\CC^3\), 
\(\Upsilon(x,y,z)=(x+y,xy/z,z)\). The invariant first integral  \(-H^2\)  is mapped by \(\Upsilon\) to the homogeneous polynomial
\begin{equation}\label{vfdeg3-1stint} g=-z^6(x^2-4yz)(x^2-yz)^2(x-y-z)^2(x-\omega y-\omega^2z)^2(x-\omega^2 y-\omega z)^2,\end{equation}
of degree 18. The images of   \(X_4\) and \(Y_4\) are, respectively, the quartic homogeneous  vector fields
\begin{multline}	X_3   =  (6y^2z^2-6x^2yz+x^4-2xz^3)\del{x}+y(x^3-3xyz-5z^3)\del{y}+\\ +z(6xyz-2x^3+z^3)\del{z}, \label{vfdeg3-x} \end{multline}  
\begin{equation}	Y_3   =  (x^2 -2yz)z^2\del{x}+z(x^3-y^3-3xyz)\del{y}+y^2z^2\del{z}, \label{vfdeg3-y} \end{equation}
which commute as well.

The transformations \(\eta\) of Eq.~(\ref{sym2}) and \(\sigma\) of Eq.~(\ref{inv:4to3}) commute with \(\nu\), and induce well-defined transformations \(\widetilde{\eta}\) and \(\widetilde{\sigma}\) on the quotient. These are given by
\begin{align*}
\widetilde{\eta}(x,y,z) & =  ( \mu x,  \mu y, \mu z),	& \widetilde{\sigma}(x,y,z) &  = (-x,-y,-z),
\end{align*}
for they satisfy the relations \(\Upsilon\circ \eta=\widetilde{\eta}\circ\Upsilon\) and  \(\Upsilon\circ  \sigma=\widetilde{\sigma}\circ\Upsilon\).

For \(\Sigma_3'= g^{-1}(-1)\),  \(\Upsilon|_{\Sigma_4'}\colon \Sigma_4'\dashrightarrow  \Sigma_3'\) is a birational isomorphism that maps \(X_4\) and \(Y_4\) to \(X_3\) and \(Y_3\), respectively, and is equivariant with respect to the actions of \(\eta\) and \(\widetilde{\eta}\) and of \(\sigma\) and~\(\widetilde{\sigma}\). In particular, \(\Upsilon\) induces a birational equivalence between \(\Sigma_4\) and the surface  \(\Sigma_3\) obtained as the quotient of \(\Sigma'_3\) under the action of \(\widetilde{\eta}^3\). Completely analogous  facts to those established for \(\Sigma_4\) can be established for \(\Sigma_3\).  The natural projection \(\pi\colon \Sigma_3 \to \PP^2\)  realizes the quotient of the action of \(\ZZ/6\ZZ\) on \(\Sigma_3\) generated by \(\widetilde{\eta}\)  and~\(\widetilde{\sigma}\).  It maps the foliation generated by \(X_3+3\alpha Y_3\) to \(\mathcal{F}_\alpha^3\).

The family of foliations~\(\mathcal{F}^3\) may be studied independently of \(\mathcal{F}^4\) via the commuting homogeneous vector fields (\ref{vfdeg3-x}) and (\ref{vfdeg3-y}), and their common homogeneous first integral~(\ref{vfdeg3-1stint}).

\subsection{The family of degree two}\label{sec:deg2} It may  be described as the one formed by the foliations on \(\PP^2\) of degree two tangent  to the rational quartic 
\begin{equation}\label{tricuspid} Q:4z(y^2z-3xy^2-x^3)+(3x^2 + y^2)^2=0\end{equation}
and to its only bitangent \(z=0\).
The generic foliation of this family has seven non-degenerate singularities: one  at each one  of the three cusps of the quartic, with ratio of eigenvalues \(3:2\); one  at each one of the two points of contact of the quartic with its bitangent, with ratio of eigenvalues \(2:1\); and two mobile ones, one on each of the curves, with ratio of eigenvalues \(-1:6\) for the one in the quartic and \(-1:3\) for the one on its bitangent.

Lins Neto constructed this family of foliations as an alternative birational model for~\(\mathcal{F}^3\). His construction can be integrated into our discussion as follows. Consider, with the linear forms \(l_i\) of Eq.~(\ref{triangle_f3}), the cubic homogeneous polynomial map \(\Xi\colon \CC^3\to\CC^3\),
\begin{equation}\label{3to2}
	\Xi(x,y,z)=z(l_0(l_2+l_1),l_0(l_2-l_1),2(l_0 l_2+l_2 l_1+l_1 l_0)).
\end{equation}
The transformation it induces  on \(\PP^2\) is the composition of the  quadratic Cremona transformation with base points at the vertices of the triangle \(l_0l_1l_2=0\) with a linear map (after dividing by \(z\), the coordinates of \(\Xi\) are independent linear functions of \(l_0l_1\), \(l_1l_2\) and \(l_2l_0\)), and is thus a birational map. The image of \(\mathcal{F}_\alpha^3\) under this birational map is the foliation of degree two \(\mathcal{F}_\alpha^2\).

The map \(\Xi\) is \(3\)-to-\(1\), and realizes the quotient by \(\eta^3\). All the objects associated to \(X_3\) and \(Y_3\) that are invariant by \(\eta^3\) are defined in the quotient. In the image, the first integral \( g\) of \(X_3\) and \(Y_3\) becomes the homogeneous polynomial of degree six  
\[\frac{1}{1728}z^2(4z(y^2z-3xy^2-x^3)+(3x^2 + y^2)^2).\]
Thus, \(\Xi\) induces a birational equivalence between \(\Sigma_3\) and the affine surface where this first integral takes the value \(-1\). For the images \(X_2\) and \(Y_2\) of \(X_3\) and \(Y_3\) in the quotient, we have
\begin{align*} 
2X_2-6Y_2 &  =(y^2-3x^2+2xz)\del{x}+2y(2z-3x)\del{y}+2z(3x-z)\del{z},\\ 
2(1+2\omega)({\textstyle \frac{1}{3}}X_2+Y_2) & =2(z-x)y\del{x}+(3x^2-y^2)\del{y}+2yz\del{z}, 
\end{align*}
thus recovering the pair of commuting vector fields studied in \cite{guillot-exemplesLN}, from which Lins Neto's  family of foliations of degree two can be studied.  
 
\subsection{A unified picture}\label{tomados} The elliptic curve \(E=\CC/\ZZ[\omega]\) has the order-six transformation   \(\rho_0\colon  u\mapsto -\omega u\). In order to have a  comprehensive picture, we consider the quotients of \(E\) by all the nontrivial subgroups of \(\langle\rho_0\rangle\),  the subgroups \(\langle\rho_0\rangle\), \(\langle\rho_0^2\rangle\) and \(\langle\rho_0^3\rangle\), and the natural relations between them. The action of \(\rho_0\) on \(E\) has a fixed point at \(0\); it has  an orbit of order three, formed by the half-periods \(\frac{1}{2}\), \(\frac{1}{2}\omega\) and \(\frac{1}{2}(1+\omega)\), whose points are stabilized by \(\rho_0^3\); and an orbit of order two, formed by the thirds-of-a-period  \(\tau_0=\frac{1}{3}(1+2\omega)\) and \(-\tau_0\), whose points are stabilized by~\(\rho_0^2\). The stabilizer of all the other points is trivial. It will be useful to consider the quotients of \(E\) under the actions of the subgroups of \(\langle \rho_0\rangle\) as orbifolds. The quotient of \(E\) under the action of \(\rho_0^3\), the elliptic involution \(u\mapsto -u\), is an orbifold, modeled on \(\PP^1\), with four conical points of angle \(\pi\) in equianharmonic position (equivalent to that of the vertices of a regular tetrahedron on the round sphere), that we will denote by  \(O^*(2,2,2,2)\).  The quotient of \(E\) under the action of \(\rho_0^2\) is the orbifold modeled on~\(\PP^1\), with three conical points of angle \(2\pi/3\), that we will denote by \(O(3,3,3)\); and  the quotient of \(E\) under the action of \(\rho_0\) is the orbifold modeled on~\(\PP^1\), with three conical points, of angles \(1/2\), \(1/3\) and \(1/6\) times \(2\pi\), respectively,  that will be denoted by \(O(2,3,6)\). The latter can be seen as the quotient of either one of the former orbifolds under the corresponding automorphism induced by~\(\rho_0\):
\begin{itemize} 
\item The map \(\overline{\rho}_0\colon O^*(2,2,2,2)\to O(2,3,6)\) induced by \(\rho_0\) has degree three and ramifies triply (i) at the point of angle \(\pi\) coming from~\(0\), producing the conic  point of angle \(2\pi/6\), and (ii) at the regular point coming from \(\pm\tau_0\), producing the conic point of angle \(2\pi/3\); the other three conical points of \(O^*(2,2,2,2)\) form an orbit for the action of \(\rho_0\), and map to the point of angle \(\pi\) in the quotient. 
\item  The map \(\overline{\rho}_0\colon O(3,3,3)\to O(2,3,6)\) induced by \(\rho_0\) has degree two, and ramifies (i) at the conical point  coming from \(0\), producing the conical point with angle \(2\pi/6\),  and (ii) at the regular point coming from the orbit of the half-periods, producing the conical point with angle~\(\pi\);  the other two conical points of \(O(3,3,3)\) form an orbit for the action of \(\rho_0\) and map to the point of angle \(2\pi/3\).
\end{itemize}

On \(E\times E\), we have the order-six automorphism \(\rho\) given by the diagonal action of~\(\rho_0\), and the projection onto the first factor \(\pi\colon E\times E\to E\) is equivariant with respect to the actions of \(\rho\) and~\(\rho_0\). The quotients of \(E\times E\) under the actions of the subgroups of \(\langle\rho\rangle\) fiber over the previously described orbifolds, with the quotient of \(E\times E\) under the action of \(\rho^3\) being the Kummer surface \(\mathrm{Kum}(E\times E)\).  The ramified coverings among the orbifolds are associated to ramified coverings between  the associated fibered surfaces;  see Figure~\ref{fig:cube}.
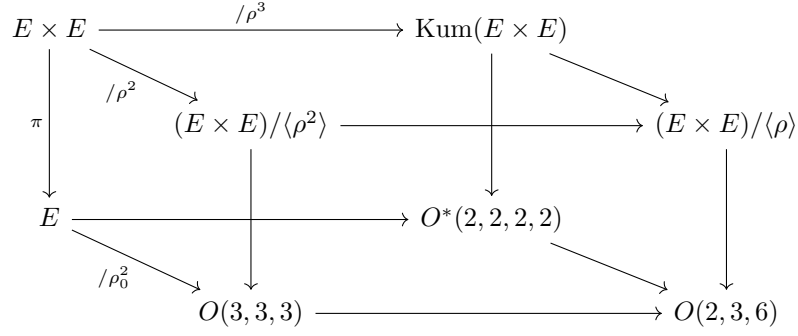
\begin{figure}
\centering	
\begin{tikzcd}%[row sep=1.5em, column sep = 1.5em]
E\times E \arrow[rr,"/\rho^3"] \arrow[dr,"/\rho^2",swap] \arrow[dd,"\pi",swap] &&
\mathrm{Kum}(E\times E) \arrow[dd] \arrow[dr] \\
& (E\times E)/\langle\rho^2\rangle \arrow[rr] \arrow[dd] &&
(E\times E)/\langle\rho\rangle \arrow[dd] \\
E \arrow[rr] \arrow[dr,"/\rho_0^2",swap] && O^*(2,2,2,2) \arrow[dr] \\
& O(3,3,3) \arrow[rr]&& O(2,3,6)
\end{tikzcd}
\caption{The elliptic fibrations over the orbifold bases.}\label{fig:cube}
\end{figure}
For \(i\in\{1,2,3\}\), the action of \(\langle\rho^i\rangle\) on \(E\times E\) has points with non-trivial stabilizers, lying above points of \(E\) with non-trivial stabilizers under the action of \(\rho_0\), which produce singular points in the quotient. After resolving these, the induced map \((E\times E)/\langle\rho^i\rangle \to E/\langle\rho_0^i\rangle\) is an elliptic fibration, with singular fibers above the conical points:  
\begin{itemize} 
	\item  fibers of Kodaira's type \(\mathrm{I}_0^*\) above those of angle \(2\pi/2\);
	\item (blown-up) fibers of Kodaira's type \(\mathrm{IV}\) above those of angle \(2\pi/3\); and 
	\item (blown-up) fibers of Kodaira's type \(\mathrm{II}\) above those of angle \(2\pi/6\)
\end{itemize} (See \cite[Ch.~4, Sect.~3]{brunella-birational}, \cite[Ch.~III, Sect.~10]{BPHV}). The combinatorics of these fibers  are presented in Figure~\ref{sing_fibers}.

\begin{figure}
	\centering	
\includegraphics[width=0.8\textwidth]{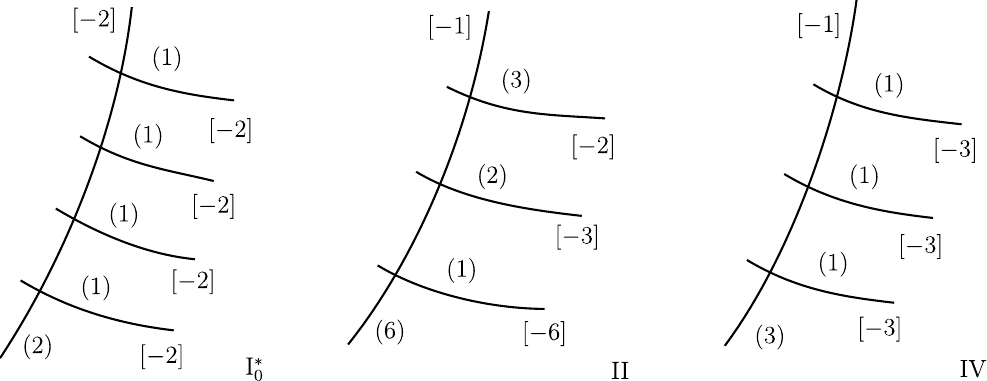}
	\caption{Kodaira's singular  fiber  of type  \(\mathrm{I}_0^*\) and blown-up fibers of type \(\mathrm{II}\) and~\(\mathrm{IV}\). All the irreducible components are rational; their self-intersections appear  in square brackets. The multiplicity with which each component appears in the fiber appears in parenthesis.}\label{sing_fibers}
\end{figure}

The linear foliations on \(E\times E\) induce foliations in all these quotients, with the fibration given by the projection onto the first factor being one of them. The other foliations in the quotients are \emph{turbulent} with respect to the fibration, in the sense that  the foliation is,  in the complement of some fibers (those above the orbifold points, in the present setting),  transverse to the fibration. The global dynamics of such foliations  are given by the dynamics of a group of biholomorphisms of \(E\) (the \emph{monodromy} group), and, for instance, turbulent foliations with infinite monodromy groups cannot have rational first integrals. See~\cite[Ch.~IV, Sect.~3]{brunella-birational} for a  discussion around turbulent foliations.

\section{Birational symmetries} 
As quotients of \(\CC^2\) under the action of a group of affine transformations, the quotients of  \(E\times E\) under consideration have biholomorphisms induced by affine transformations. These biholomorphisms preserve the  foliations induced by the linear ones on \(E\times E\). They  account for all the birational transformations that do so:
 
\begin{proposition}\label{prop:aff-bir} Let \(\Gamma\subset \mathrm{Aff}(\CC^2)\) be a group,  containing a lattice with finite index, that acts on \(\CC^2\) by preserving every foliation by parallel lines. Then, every birational automorphism of \(\Gamma\backslash\CC^2\) preserving the family of linear foliations is holomorphic, and is induced by  an element of \(\mathrm{Aff}(\CC^2)\).
\end{proposition}

\begin{proof} Let \(f\colon\Gamma\backslash\CC^2\dashrightarrow\Gamma\backslash\CC^2\) be such  a birational map. Let \(\Omega\subset \Gamma\backslash\CC^2\) be the Zariski-open subset where \(f\) is defined and is a local biholomorphism. Let \(U\subset \CC^2\) be the open subset where \(\Gamma\) acts freely. Let \(\pi\colon\CC^2\to \Gamma\backslash\CC^2\) be the quotient map. Let \(\widetilde{p}\in U\) such that \(p=\pi(\widetilde{p})\) is in  \(\Omega\). Let \(q=f(p)\) and let \(\widetilde{q}\in \pi^{-1}(q)\).  Let \(V_0\subset \CC^2\) be a neighborhood of \(\widetilde{p}\) and \(\widetilde{f}_0\colon V_0\to\CC^2\) a local lift of \(f\) mapping \(\widetilde{p}\) to \(\widetilde{q}\).

Since \(\widetilde{f}_0\) preserves the foliations by parallel lines, \emph{it is the restriction to \(V_0\) of an affine map \(F\in\mathrm{Aff}(\CC^2)\).}    In fact, if \(W\) is a connected open subset of \(\CC^2\), and \(\phi\colon W\to\CC^2\) is a biholomorphism onto its image that maps parallel lines to parallel lines, it is an affine map: up to pre-and post-composition by affine maps, we may suppose that it fixes \(0\), and that it preserves the foliations given by the level sets of \(x\), \(y\) and \(x-y\); since it preserves the first two, it has the form 
\(\phi(x,y)=(\phi_1(x),\phi_2(y))\); and since it preserves the third, the image  of  line \(t\mapsto (t,t+c)\), given by \((\phi_1(t),\phi_2(t+c))\), must still be a line with  slope~\(1\), hence \(\phi_1'(t_0)=\phi_2'(c+t_0)\), and we conclude that \(\phi_1'\) and \(\phi_2'\) are constant and equal, that \(\phi\) is a homothety (in particular, an affine map).

Let \(g\in\Gamma\), and let \(\gamma\colon [0,1]\to U\cap \pi^{-1}(\Omega)\) be a path joining \(\widetilde{p}\) to \(g(\widetilde{p})\), so that \(\pi\circ \gamma\) is a closed path. The analytic continuation of \(\widetilde{f}_0\) along \(\gamma\) is given by \(F\). On the one hand, the analytic continuation of \(f\) along \(\pi\circ\gamma\) is given by \(f\) itself, and, on the other, by \(\pi\circ F\). Thus, there exists \(h\in\Gamma\) such that \(gh=Fg\), this is,  \(F\) normalizes \(\Gamma\) within \(\mathrm{Aff}(\CC^2)\). This implies that \(F\) induces a well-defined holomorphic map of \(\Gamma\backslash\CC^2\) that coincides with \(f\) in some open subset, and that must agree everywhere with~it.\end{proof}

\subsection{Affine biholomorphisms of the quotients}\label{affquot} We  now describe the groups of   affine maps that act on the quotients of \(E\times E\) of interest.

\begin{proposition}\label{id:groups} The group of biholomorphisms induced by the group of affine transformations of \(\CC^2\)  on 
	\begin{itemize}
		\item  \((E\times E)/\langle \rho\rangle\), is isomorphic to \(\mathrm{PGL}(2,\ZZ[\omega])\); 
		\item  \((E\times E)/\langle \rho^2\rangle\),   is isomorphic to  \((\mathrm{GL}(2,\ZZ[\omega])/\langle\rho^2\rangle)\ltimes (\ZZ/3\ZZ)^2\),   with the   second factor generated, for the third-of-a-period \(\tau_0=\frac{1}{3}+\frac{2}{3}\omega\), by the translations  of \(\CC^2\) by \((0,\tau_0)\) and \((\tau_0,0)\);
		\item  \((E\times E)/\langle \rho^3\rangle\),   is isomorphic to  \((\mathrm{GL}(2,\ZZ[\omega])/\langle-\mathbf{I}\rangle)\ltimes (\ZZ/2\ZZ)^4\),   with the   second factor given by the translations  by half-periods, generated by the  translations by \((\frac{1}{2},0)\), \((\frac{1}{2}\omega,0)\), \((0,\frac{1}{2})\), and \((0,\frac{1}{2}\omega)\). 
	\end{itemize}
\end{proposition}

\begin{proof} Consider  the group \(\mathrm{Aff}(\CC^2)\) of affine transformations of \(\CC^2\), in its semidirect product decomposition \(\mathrm{GL}(2,\CC)\ltimes \CC^2\).  For \(i\in\{1,2,3\}\),	let \(G_i\subset \mathrm{Aff}(\CC^2)\)  be the group \(\langle \rho^i \rangle \ltimes (\ZZ[\omega]\times\ZZ[\omega])\), and let \(N_{G_i}\) be its normalizer. By definition, \((E\times E)/\langle \rho^i \rangle=G_i\backslash \CC^2\). The  group of biholomorphisms of \((E\times E)/\langle \rho^i \rangle\) that come from  \(\mathrm{Aff}(\CC^2)\) is isomorphic to \(N_{G_i}/G_i\).

For \(A\in\mathrm{GL}(2,\CC)\) and \(a\in\CC^2\),  let \((A,a)\) denote the affine transformation \(\zeta \mapsto A\zeta +a\). In the semidirect product factorization of \(\mathrm{Aff}(\CC^2)\), the product reads
\[(A,a)(B,b)=(AB,Ab+a).\]
For \(B\) in the center of \(\mathrm{GL}(2,\CC)\),	
\begin{equation}\label{conjugacy}
	(A,a)(B,b)(A,a)^{-1}=(B,(\mathbf{I}-B)a+Ab).
\end{equation}

If \((A,a)\in N_{G_i}\), and \(b\in \ZZ[\omega]\times \ZZ[\omega]\),   the conjugate of \((\mathbf{I},b)\) by \((A,a)\) must belong to~\(G_i\), and,  from (\ref{conjugacy}), we have that \(Ab\in \ZZ[\omega]\times \ZZ[\omega]\), this is, \(A\in\mathrm{GL}(2,\ZZ[\omega])\). Once this has been established, if \(B\) is a power of \(\rho\), Eq.~(\ref{conjugacy}) gives the second necessary condition \begin{equation}\label{norm.cond}
	(\mathbf{I}-B)a\in \ZZ[\omega]\times \ZZ[\omega].\end{equation} 

For \(G_1\), for \(B=\rho\) and \(a=(a_1,a_2)\), condition (\ref{norm.cond}) is equivalent to
\((1+\omega)a_i=n+m\omega\). By multiplying both sides of this equality by \(-\omega\), we obtain the condition  \(a_i\in\ZZ[\omega]\).  Thus, \(N_{G_1}\) is \(\mathrm{GL}(2,\ZZ[\omega])\ltimes  (\ZZ[\omega]\times \ZZ[\omega])\). This establishes the first case.

For \(G_2\), for \(B=\rho^2\), condition (\ref{norm.cond}) is equivalent to \((1-\omega^2)a_i=n+m\omega\) for some \(n,m\in\ZZ\). From this,
\[3a_i=(1-\omega)(1-\omega^2)a_i=(1-\omega)(n+m\omega)=(n+m)+(2m-n)\omega,\]
and the coefficients \(n+m\) and \(2m-n\) of the last expression represent the most general integers whose sum is in \(3\ZZ\). Thus, \(N_{G_2}\) is  the semidirect product of \(\mathrm{GL}(2,\ZZ[\omega])\) and of the the group of translations 
\[\left\{\textstyle \frac{1}{3}a+\frac{1}{3}b\omega\mid a,b\in\ZZ, a+b\equiv 0\mod 3\right\}.\]	
This establishes the second fact.

For \(G_3\), \(B=-\mathbf{I}\), and condition (\ref{norm.cond}) is equivalent to \(2a\in\ZZ[\omega]\times\ZZ[\omega]\), this is, \(a\) is the most general half-period. 
\end{proof}

By Proposition~\ref{prop:aff-bir}, the groups of birational transformations of \(\PP^2\) preserving the families \(\mathcal{F}^2\) and \(\mathcal{F}^3\) will be associated to the first item in this proposition;   the group  of birational transformations preserving \(\mathcal{F}^4\), to the second; and the group  of birational transformations of the Kummer surface \(\mathrm{Kum}(E\times E)\) that preserve the foliations induced by the linear ones, from the third one.  

In the quotients discussed in the above proposition, \(\mathrm{GL}(2,\ZZ[\omega])\) acts on the family of linear foliations by fractional linear transformations. This action is transitive on the pairwise different triples of  points of \(\PP^1 (\QQ(\omega))\)  (cf.~\cite[\S 4]{bianchi-gruppi}). Thus, all the triples of pairwise different foliations in  \((E\times E)/\langle \rho^i\rangle\)  with slope in \(\QQ(\omega)\) are equivalent to one another under the action of an element of \(\mathrm{GL}(2,\ZZ[\omega])\).  In consequence, we have:

\begin{corollary}\label{coro:trans_rational} For \(j\in\{2,3,4\}\), the group of birational transformations of \(\PP^2\) that preserves the family \(\mathcal{F}^j\) acts transitively upon the triples of pairwise different foliations in \(\{\mathcal{F}^j_\alpha \mid  \alpha\in\QQ(\omega)\}\).
\end{corollary}

In particular,  from this, and from the description of Section~\ref{tomados}, for every \(\alpha\in\QQ(\omega)\), the foliation \(\mathcal{F}_\alpha^4\) is birationally equivalent to \(\mathcal{F}_\infty^4\), an elliptic fibration with three fibers of Kodaira type~\(\mathrm{IV}\) (cf. \cite[Cor.~1.4]{linglu}).

\subsection{Generators for \texorpdfstring{\(\mathrm{GL}(2,\ZZ[\omega])\)}{GL(2,Z[w])} } We begin by recalling   some facts about  the ring \(\ZZ[\omega]\) of Eisenstein integers. It is  \emph{Euclidean}: it has the \emph{norm}  \(N(a+\omega b)=a^2+b^2-ab\) (the square of its complex norm), for which \(N(ab)\geq N(b)\) for nonzero \(a\) and \(b\). It has a division: for \(a\) and \(b\) in \(\ZZ[\omega]\), \(a\neq 0\), there exist \(q\) and \(r\) in  \(\ZZ[\omega]\) with \(N(r) < N(a)\), such that \(b=qa+r\) (\(q\) and \(r\) are not unique, but there are only finitely many choices for them). There is a Euclidean algorithm for calculating the greatest common divisor of two of its elements.

The following result, with its many variations, is classical (see~\cite[\S 1, par.~II]{bianchi-zahlen}; see also \cite[Section~4.4.2]{fine}):

\begin{proposition}\label{gl2gen} The group \(\mathrm{GL}(2,\ZZ[\omega])\) is generated by its elements
\begin{align} \label{alleged}
	B_0 & =\left(\begin{array}{rr} 0 & 1 \\ -1 & 0\end{array}\right), & S_0 & =\left(\begin{array}{rr} -\omega & 0 \\ 0 & 1\end{array}\right), &  U_0 & =\left(\begin{array}{cc} 1 & 1 \\ 0 & 1\end{array}\right).\end{align}
\end{proposition}

\begin{proof} The idea is that any matrix in \(\mathrm{GL}(2,\ZZ[\omega])\) can be transformed into the identity one through elementary column operations (exchanging columns, adding one column to another, multiplying a column by a unit of \(\ZZ[\omega]\)), and that, thanks to the Euclidean algorithm in \(\ZZ[\omega]\), these operations can be realized by successively multiplying a given matrix to the right by products of the above matrices.

We begin by observing that the matrices
\begin{equation} \label{unipotent} S_0 U_0^{-n}S_0^{-1}U_0^m=\left(\begin{array}{cc} 1 & m+n\omega \\ 0 & 1 \end{array}\right)\end{equation}
are all those of the form  \(\left(\begin{array}{cc} 1 & s \\ 0 & 1 \end{array}\right)\) with \(s\in\ZZ[\omega]\), and that
\begin{equation}\label{unip_low}
		B_0\left(\begin{array}{cc} 1 & s \\ 0 & 1 \end{array}\right)B_0^{-1}=\left(\begin{array}{rr} 1 & 0 \\ -s & 1 \end{array}\right);\end{equation}
we also observe that
\(B_0S_0B_0^{-1}=\left(\begin{array}{rr} 1 & 0 \\ 0 & -\omega \end{array}\right)\),
and that the products of powers of this matrix (of order six)  with powers of \(S_0\) (which is also of order six) give  the most general diagonal matrix in \(\mathrm{GL}(2,\ZZ[\omega])\).
	
Let \(A=\left(\begin{array}{cc} a & b \\ c & d \end{array}\right)\in\mathrm{GL}(2,\ZZ[\omega])\). We will show that, by successively multiplying to the right by the matrices of Eq.~(\ref{alleged}), we may transform \(A\) into the identity matrix.
\begin{enumerate}
	\item \label{case:b0}  If \(b=0\),
		\(A=\left(\begin{array}{cc} (-\omega)^{i} & 0 \\ c & (-\omega)^{j} \end{array}\right)\) for some \(i\) and \(j\). Up to multiplication to the right by a suitable diagonal matrix, we get a matrix of the form (\ref{unip_low}), which, as we have seen, may be written as a product of the matrices~(\ref{alleged}).
	\item  If \(b\neq 0\) and \(a=0\), then since \(AB_0=\left(\begin{array}{cc} -b & a \\ -d & c \end{array}\right)\), multiplying \(A\) to the right by \(B_0\) takes us to the previous case.
	\item \label{case:other}
		In the remaining case,  \(ab\neq 0\),  up to multiplying to the right by \(B_0\), we may suppose that  \(N(a)\leq N(b)\). By the division algorithm, there exist \(q\) and \(r\) in \(\ZZ[\omega]\), with  \(N(r)<N(a)\) such that \(b=qa+r\). By multiplying to the right by the  matrix   (\ref{unipotent}) associated to \(-q\),
		\[A\left(\begin{array}{rr} 1 & -q \\ 0 & 1 \end{array}\right)=\left(\begin{array}{cc} a & b-aq \\ c  & d-cq  \end{array}\right)=\left(\begin{array}{cc} a & r\\ c & d-cq \end{array}\right).\]
		If \(r=0\), we are back in  case (\ref{case:b0}). Otherwise, we are still in case (\ref{case:other}), but  the sum of the norms of the elements of the first row has decreased.
	\end{enumerate}
	Iterating this procedure gives the result.
\end{proof}	

Bianchi gave a fundamental polyhedron for the action of \(\mathrm{PSL}(2,\ZZ[\omega])\) on the three-dimen\-sion\-al hyperbolic space \(\HH^3\) \cite[\S 13]{bianchi-gruppi}, from which a presentation of this group may be deduced; an explicit one is given in \cite[Sect.~4.4.2]{fine} (see also \cite{Johnson-Ivic}). Presentations for \(\mathrm{GL}(2,\ZZ[\omega])\) and \(\mathrm{SL}(2,\ZZ[\omega])\) appear in~\cite[Sect.~6]{Swan}.

On \(E\times E\), the foliation  of slope \(\lambda\in\PP^1\), the one generated by \(\indel{u}+\lambda \indel{v}\) for \(\lambda\in\CC\), and by \(\indel{v}\) for \(\lambda=\infty\), induces foliations on all the birational models of the quotients of \(E\times E\) here discussed; the same parameter is affected to all of them, and will be in all cases referred to as its \emph{slope}. Every element of \(\mathrm{GL}(2,\ZZ[\omega])\) will act in every such birational model  preserving the family of foliations,  acting upon this parameter. In all quotients and birational models, the action on the slopes of the foliations will be the same. The actions of the generators of Proposition~\ref{gl2gen} on the foliations, in terms of the slope \(\lambda\), are given by
\begin{align*} B(\lambda) & =-\frac{1}{\lambda}, & S(\lambda) & =-\omega^2\lambda, &  U(\lambda) & =\frac{\lambda}{\lambda+1}.\end{align*}
Another parameter for the foliations is Lins Neto's parameter \(\alpha\), related to the slope by  relation~(\ref{lambda_to_alpha}). In terms of this parameter,  the above transformations read
\begin{align} \label{act.lins.par}
B^\dagger(\alpha) & =\frac{\alpha+2}{\alpha-1}, &  S^\dagger(\alpha) & =-\omega^2\alpha -\omega, &  U^\dagger(\alpha) & =-\frac{2\omega^2\alpha+1}{\alpha+2\omega}.\end{align}

These transformations are the ones induced by the natural linear action of \(\mathrm{GL}(2,\ZZ[\omega])\) on constant vector fields on \(\CC^2\), and this action will be consistent throughout the quotients and birational models as long as we consider the action on the foliations induced by the action on the vector fields that generate them. When considering foliations defined by forms, in order to keep this consistence,  the action by  pullback of the \emph{inverses} of these transformations on the forms must be considered.

\subsection{Birational symmetries of \texorpdfstring{\(\mathcal{F}^4\)}{F4}} By Proposition~\ref{id:groups},  \(\mathrm{GL}(2,\ZZ[\omega])\ltimes (\ZZ/3\ZZ)^2\) 
acts  on \((E\times E)/\langle \rho^2\rangle\) by biholomorphisms that preserve the foliations induced by the linear foliations on \(E\times E\). By Proposition~\ref{prop:aff-bir}, these are all the birational transformations that do so. In this way, the  birational map \(\Phi\) of Theorem~\ref{thm:ln-desc} induces a group homomorphism \(\Phi_*\colon \mathrm{GL}(2,\ZZ[\omega])\times (\ZZ/3\ZZ)^2 \to \mathrm{Bir}(\PP^2)\), whose image is given by the birational transformations of \(\PP^2\) that preserve the family of foliations~\(\mathcal{F}_4\).  The  proof of Theorem~\ref{thm:bir-f4}   will take place in the setting of Section~\ref{sec:deg4}, and is  a higher-dimensional generalization of the procedure from which explicit formulas for Lattès maps of \(\PP^1\) are obtained (cf. \cite[Ch.~7]{milnor-dynamics}). Taking into account the structure of  the group of affine biholomorphisms of \((E\times E)/\langle\rho^2\rangle\) described in Proposition~\ref{id:groups}, and the generating set  of \(\mathrm{GL}(2,\ZZ[\omega])\) of Proposition~\ref{gl2gen},  we will  study the action of  generators of the group of Proposition~\ref{id:groups} on the generators of the field of rational functions of \(\PP^2\) of Lemma~\ref{lemma:ratfun} via the birational map \(\Phi\) of Theorem~\ref{thm:ln-desc}.

For  the linear map \(B_0\) of Proposition~\ref{gl2gen}, let \(B_4=\Phi_*(B_0)\) be the corresponding birational transformation of \(\PP^2\). By the identification in Section~\ref{sec:deg4} between the rational functions of \(\PP^2\) considered in Lemma~\ref{lemma:ratfun}  (as functions on \(\Sigma_4\)) and  functions on \(E\times E\), and using the properties of the Weierstrass functions recalled in Section~\ref{sec:weiers},  we have
\begin{align*}	
B_4^*\left(\frac{p_0}{p_1}\right) & =B_0^*\left(\frac{\wp(u)}{\wp(v)}\right)=\frac{\wp(v)}{\wp(-u)}=\frac{\wp(v)}{\wp(u)}=\frac{p_1}{p_0},\\
B_4^*\left(\frac{q_0}{H}\right) & =B_0^* (\wp'(u) ) = \wp'(v) =\frac{q_1}{H},\\
B_4^*\left(\frac{q_1}{H}\right) & =B_0^* ( \wp'(v)  )= \wp'(-u) =- \wp'(u) =-\frac{q_0}{H}.
\end{align*}
By Lemma \ref{lemma:ratfun}, these conditions determine \(B_4\). If \(B_4(x:y:z)=(P:Q:R)\), with \(P\), \(Q\) and \(R\) homogeneous polynomials in \(x\), \(y\) and \(z\) of the same degree, these three equations give, with the expressions of Lemma~\ref{lemma:ratfun}, explicit equations for \((P:Q:R)\) in terms of \((x:y:z)\). These are, respectively,
\begin{align*}
& \frac{(P-R)(Q-R)}{P^2+PQ+Q^2} 
		=\frac{x^2+xy+y^2}{(x-z)(y-z)},\\
& \frac{P^3+Q^3-2R^3}{P^3-Q^3}
		=-(1+2\omega)\frac{(x^2y+x^2z+xy^2+xz^2+y^2z+yz^2)}{(x-y)(y-z)(z-x)}, \\ 
& (1+2\omega)\frac{(P^2Q+P^2R+PQ^2+PR^2+Q^2R+QR^2)}{(P-Q)(Q-R)(R-P)}
		=\frac{x^3+y^3-2z^3}{x^3-y^3}.
\end{align*}
They can be considered as   a system of polynomial equations in \(P/R\) and \(Q/R\) (one quadratic, two cubic). Furthermore,  by Lemma~\ref{lemma:ratfun}, there is one and only one solution to it. We have solved this system with the help of the computer, obtaining the stated expression for~\(B_4\).  Observe that, by the same Lemma, in order to establish the result, it is sufficient to verify that the rational function \(B_4\) in the statement   acts upon the functions of the Lemma in the above way (although this does not say how to obtain the expression for \(B_4\), it is a simpler calculation, that can be done by hand).

The other transformations are obtained in a similar way. The birational transformation  \(S_4=\Phi_*(S_0)\)  is determined by the conditions 
\begin{align*}	S_4^*\left(\frac{p_0}{p_1}\right) & =S_0^*\left(\frac{\wp(u)}{\wp(v)}\right)=\frac{\wp(-\omega u)}{\wp(v)}=\omega\frac{\wp( u)}{\wp(v)}=\omega\frac{p_0}{p_1},\\
S_4^*\left(\frac{q_0}{H}\right) & =S_0^*(\wp'(u))=\wp'(-\omega u)=-\wp'(u)=-\frac{q_0}{H},\\
S_4^*\left(\frac{q_1}{H}\right)& =S_0^*( \wp'(v) )=\wp'(v)=\frac{q_1}{H},
\end{align*}
while the birational transformation  \(U_4\) associated to \(U_0\) is determined by 
\begin{align*}
U_4^*\left(\frac{p_0}{p_1}\right) & =U_0^*\left(\frac{\wp(u)}{\wp(v)}\right)=\frac{\wp(u+v)}{\wp(v)}=\frac{1}{4\wp(v)}\left(\frac{\wp'(u)-\wp'(v)}{\wp(u)-\wp(v)}\right)^2-\frac{\wp(u)}{\wp(v)}-1 \\& =\frac{1}{4p_1}\left(\frac{q_0-q_1}{p_0-p_1}\right)^2-\frac{p_0}{p_1}-1, \\
U_4^*\left(\frac{q_0}{H}\right) & =U_0^* (\wp'(u)) = \wp'(u+v) \\  & =\frac{\wp'(u)\wp'(v)(\wp'(u)-\wp'(v))}{(\wp(u)-\wp(v))^3}-3\frac{\wp^2(u)\wp'(v)+\wp^2(v)\wp'(u)}{(\wp(u)-\wp(v))^2} \\ & =\frac{q_0q_1(q_0-q_1)}{H(p_0-p_1)^3}-3\frac{p_0^2q_1+p_1^2q_0}{H(p_0-p_1)^2}, \\ 
U_4^*\left(\frac{q_1}{H}\right) & =U_0^* (\wp'(v) )= \wp'(v) =\frac{q_1}{H}.	
\end{align*}
Through these explicit conditions, the transformations \(S_4\) and \(U_4\) can be obtained in the same way as before. These three transformations generate the action of the factor \(\mathrm{GL}(2,\ZZ[\omega])\) in the semidirect product \(\mathrm{GL}(2,\ZZ[\omega])\ltimes (\ZZ/3\ZZ)^2\).

Let us now discuss generators for the second  factor. For \(\zeta_0\in\CC\) such that
\(\wp(\zeta_0)=0\)  and \(\wp'(\zeta_0)=H\),  from  the relations (\ref{weies-sum}--\ref{weiesp-sum})  and the condition \( g=-H^2\),  the transformation \(L_0(u,v)=(u+\zeta_0,v)\) of \(\CC^2\) gives one of the order-three translations in the second item of Proposition~\ref{id:groups}. For \(L_4=\Phi_*(L_0)\),  we have 
\begin{align*}
L_4^*\left(\frac{p_0}{p_1}\right) & =L_0^*\left(\frac{\wp(u)}{\wp(v)}\right)=\frac{\wp(u+\zeta_0)}{\wp(v)}=-\frac{ g+\wp'(u)\wp'(\zeta_0)}{2\wp^2(u)\wp(v)}=\frac{H(H-q_0)}{2p_0^2p_1},\\
L_4^*\left(\frac{q_0}{H}\right)& = L_0^* (\wp'(u))  = \wp'(u+\zeta_0) = \wp'(\zeta_0) +\frac{ g(\wp'(u)-\wp'(\zeta_0))}{\wp^3(u)}  
 =1-\frac{H(q_0-H)}{p_0^3},\\
L_4^*\left(\frac{q_1}{H}\right) & =L_0^*  (\wp'(v))  = \wp'(v) = \frac{q_1}{H}.
\end{align*}
These explicit conditions determine \(L_4\), and the expression in the statement may be obtained from them in the same way as before. Since the factor \((\ZZ/3\ZZ)^2\) in the semidirect product is the normal one, the conjugate \(L_4'\) of \(L_4\) of Eq.~(\ref{bir:e_conj}) belongs to it. Since \(L_4'\) does not belong to the cyclic group generated by \(L_4\), it generates, with the latter, the factor  \((\ZZ/3\ZZ)^2\). This finishes the proof of Theorem~\ref{thm:bir-f4}.

\begin{proposition}
\label{prop:non-effective} Within the group of birational transformations of the plane preserving the family~\(\mathcal{F}^4\), the subgroup of those that act by mapping each foliation  to itself is the group of linear transformations \(\langle J \rangle\ltimes(\ZZ/3\ZZ)^2\), of order~18, generated by the involution \(J\) of Eq.~(\ref{sym2}), and  by the group generated by \(L_4\) and \(L_4'\). 
\end{proposition}
\begin{proof} The translations of \(\CC^2\) act on the foliations by parallel lines by preserving each foliation individually, and thus the translational part of the group of the second item of Proposition~\ref{id:groups}, the normal factor, isomorphic to \((\ZZ/3\ZZ)^2\), acts on \((E\times E)/\langle\rho\rangle\) by preserving the linear foliations. Thus, the group that preserves every foliation contains the factor \((\ZZ/3\ZZ)^2\) and is necessarily a semidirect product with it.  If the  transformation of \((E\times E)/\langle\rho\rangle\) coming from the affine map \((A,b)\) preserves all the linear foliations individually, then \(b\) is one of these translations, and, on its own, the linear part \((A,0)\) preserves all the foliations as well. If a linear transformation induces a nontrivial  transformation of \((E\times E)/\langle\rho\rangle\) that preserves every foliation, it is a diagonal element of \(\mathrm{GL}(2,\ZZ[\omega])\), and is thus equal to \(\rho\), up to a power of~\(\rho^2\). From the proof of Proposition~\ref{prop-rel34}, the birational transformation of \(\PP^2\) associated to \(\rho\) is \(J\), and, from the previous theorem, \(L_4\) and \(L_4'\) correspond to the translations generating the factor  \((\ZZ/3\ZZ)^2\). All the transformations in the group are linear.  
\end{proof}

By the semidirect product structure of the group of this proposition, \(J\) acts by conjugation on the factor \((\ZZ/3\ZZ)^2\). Since \(JL_4J=L_4^2\) and \(JL_4'J=(L_4')^2\), this involutive automorphism of \((\ZZ/3\ZZ)^2\) maps each element to its inverse (the group of the previous proposition is a generalized dihedral group). Thus, the nine elements which are not in \((\ZZ/3\ZZ)^2\) are all involutions, and, by Sylow's second theorem, are all conjugate to one of them, say \(J\), by elements of the group generated by \(L_4\) and~\(L_4'\).

\begin{remark} The birational automorphism of \(\PP^2\) associated to a hyperbolic element of \(\mathrm{GL}(2,\ZZ[\omega])\) preserves two of the foliations in the family of degree four, images of the stable and unstable foliations of the associated Anosov transformation on \(E\times E\). Such birational transformations take a special place among the  birational groups of transformations preserving a foliation~\cite[Thm.~3.1]{cantat-favre}. Explicit examples of such birational automorphisms of \(\PP^2\) may be constructed via Theorem~\ref{thm:bir-f4}.\end{remark}

\subsection{Linear symmetries of the dual Hesse configuration}

The group of linear symmetries of the Hesse configuration  is isomorphic to    \(\mathrm{SL}(2,\ZZ/3\ZZ)\ltimes (\ZZ/3\ZZ)^2\), and has  order~\(216\) (see~\cite[Sect.~7.3, Props.~7 and~8]{PAC}). This is also the group of linear symmetries  of the dual configuration, the one of interest here. The group of linear symmetries of the dual Hesse configuration is exactly the group of linear transformations of \(\PP^2\) that preserve  the family \(\mathcal{F}^4\). With respect to the structure of the birational group of Theorem~\ref{thm:bir-f4}, the subgroup formed by the linear transformations has naturally the structure of a semidirect product, as the normal factor \((\ZZ/3\ZZ)^2\) is generated by the linear transformations \(L_4\) and~\(L_4'\). In the other factor, image of \(\mathrm{GL}(2,\ZZ[\omega])\), we have the linear transformation \(B_4\), whose square is the transformation \(J\) of Eq.~(\ref{sym2}), and the transformation of order three
\begin{equation}
\label{T4}
T_4(x:y:z)  = (x:y: \omega z).\end{equation}
It is equal to \(\Phi_*(T_0)\) for 
\begin{equation}\label{T0} T_0=\left(\begin{array}{cc} -\omega & 0 \\  1 & -\omega^2 \end{array}\right)  = B_0 S_0 U_0 S_0^{-1}U_0^{-1} S_0^{-1}B_0^{-1}S_0,\end{equation}
and \(T_4\) belongs thus to the  conjugate of the quotient of \(\mathrm{GL}(2,\ZZ[\omega])\) generated by \(B_4\), \(S_4\) and~\(U_4\). It acts upon the foliations by
\begin{equation}\label{actionT}
		T^\dagger(\alpha) =\omega \alpha.
	\end{equation}

In agreement with the known results on the linear symmetries of the dual Hesse arrangement, we have the following result.
\begin{proposition}\label{sl2f3} The group of linear transformations generated by \(T_4\) and \(B_4\) is isomorphic to \(\mathrm{SL}(2,\FF_3)\).
\end{proposition}

\begin{proof} The ring homomorphism   \(\phi\colon \ZZ[\omega]\to \FF_3\), \(\phi(m+n\omega)=n+m\pmod 3\) induces a group homomorphism  \(\phi_*\colon \mathrm{GL}(2,\ZZ[\omega])\to \mathrm{GL}(2,\FF_3)\). The images \(\overline{T}_0\) and \(\overline{B}_0\) of \(T_0\) and \(B_0\) under \(\phi_*\) lie in \(\mathrm{SL}(2,\FF_3)\), and are 
\begin{align*} \overline{T}_0 & = \left(\begin{array}{cc} 2 & 0 \\  1 & 2 \end{array}\right), & \overline{B}_0 & = \left(\begin{array}{cc} 0 & 1\\  2 & 0 \end{array}\right).\end{align*}
The  actions of these on \(\PP^1(\FF_3)\) by fractional linear transformations are as follows: 	\(\overline{B}_0\) exchanges \(0\) and \(\infty\) on the one hand, and \(1\) and \(2\) on the other; on its turn, \(\overline{T}_0\) fixes~\(0\), and permutes cyclically \(1\),  \(\infty\),  and \(2\).  Within the group of permutations of \(\PP^1(\FF_3)\), the image of \(\langle \overline{B}_0, \overline{T}_0\rangle\) is thus isomorphic to the alternating group in four symbols \(A_4\), of order \(12\). It has the associated presentation
\[\langle B,T \mid  B^2=1, T^3=1, (BT)^3=1 \rangle.\]
The kernel of the resulting homomorphism \(\langle B_0, T_0\rangle \to A_4\) is thus normally generated by \(B_0^2=-\mathbf{I}\), \(T_0^3=-\mathbf{I}\) and \((B_0T_0)^3=-\mathbf{I}\); Since \(-\mathbf{I}\) is central, this kernel reduces to \(\{\mathbf{I}, \mathbf{-I}\}\). Since \(-\mathbf{I}\) has a nontrivial image under \(\phi_*\),  the group \(\langle B_0, T_0\rangle\) is isomorphic to a subgroup of \(\mathrm{SL}(2,\FF_3)\) of order \(2|A_4|=24\), and is thus all of \(\mathrm{SL}(2,\FF_3)\).

Consider now the image \(\langle B_4, T_4\rangle\) of \(\langle B_0, T_0\rangle\) in \(\mathrm{PGL}(2,\CC)\). Its action on the family of foliations \(\mathcal{F}^4\) is given as follows:  \(T_4\) acts  upon Lins Neto's parameter \(\alpha\) by~(\ref{actionT}); thus, \(T_4\) fixes
\(\mathcal{F}^4_\infty\), and cyclically permutes  \(\mathcal{F}^4_1\), \(\mathcal{F}^4_\omega\) and \(\mathcal{F}^4_{\omega^2}\). The action of \(B_4\), given by (\ref{act.lins.par}), exchanges \(\mathcal{F}^4_\infty\) and \(\mathcal{F}^4_1\) on one hand, and \(\mathcal{F}^4_\omega\) and \(\mathcal{F}^4_{\omega^2}\) on the other. In this  way, \(\langle B_4, T_4\rangle\) acts upon the foliations \(\mathcal{F}_{\alpha}^{4}\) for \(\alpha\in\{1,\omega,\omega^2,\infty\}\) as the full alternating group in four symbols. Since \(B_4^2\) is the nontrivial involution \(J\), which preserves every one of these foliations, the mapping \(\langle B_0, T_0\rangle\to\mathrm{PGL}(2,\CC)\) is injective. \end{proof}

\begin{remark} In Fujiki's studies on finite groups acting on complex abelian surfaces by Lie group automorphisms  \cite{fujiki}, the group \(\mathrm{SL}(2,\FF_3)\)  appears under the name of the \emph{binary tetrahedral group}. The group of automorphisms of \(E\times E\) generated by \(\langle B_0, T_0\rangle\approx  \mathrm{SL}(2,\FF_3)\) and \(\rho^2\) appears in Fujiki's article as  item~8 in Table~9, where it is  shown to be maximal among finite groups acting by Lie group automorphisms on \(E\times E\). Its extension by \(\ZZ/3\ZZ\times\ZZ/3\ZZ\) gives a group   of order~\(648\) acting by biholomorphisms on \(E\times E\), containing the group generated by \(\rho^2\) in its center. The birational map of Theorem~\ref{thm:ln-desc} gives a homomorphism, of geometric origin, from this group   to the group of linear automorphisms of the dual Hesse configuration. \end{remark}

\subsection{An alternative generating set}

The quadratic Cremona involutions of the plane
\begin{align*}   
	Q_1(x:y:z) & = (y^2 -x z : x^2 - y z: z^2- x y  ),\nonumber \\ 
	Q_\omega (x:y:z) & = (\omega y^2 - x z : \omega x^2 - y z: z^2- \omega^2 x y), \nonumber \\ 
	Q_{\omega^2} (x:y:z) & = (\omega^2 y^2 - x z : \omega^2 x^2 - yz : z^2- \omega x y), \nonumber \\ 
	Q_\infty(x:y:z) & =  ( y z: x z: x y ),   
\end{align*}  
were introduced in   \cite{mendes-puchuri-effective} as birational maps that  preserve the family \(\mathcal{F}^4\), and that are specially adapted to the degenerate foliations of the family (for instance, the base points of  \(Q_\kappa\) are the points of \(\mathcal{P}_\kappa\)).

With the transformation \(T_4\) of Eq.~(\ref{T4}), we have the  factorizations
\begin{align*}  
	S_4  & = Q_{\omega} \circ T_4^{-1}, \\ 
	U_4 & = T_4^{-1}\circ  B_4 \circ T_4\circ Q_\infty \circ   Q_{\omega}.\end{align*}
We also have that 
\begin{align*} 
	Q_{1} & =B_4 \circ Q_\infty\circ B_4^{-1},\\
	Q_{\omega^2} & =T_4^{-1} \circ Q_1\circ T_4, \\
	Q_{\omega} & =T_4^{-1}  \circ Q_{\omega^2}\circ T_4, 
\end{align*}
and that \(Q_{\infty}\) is the image under \(\Phi_*\) of the element of \(\mathrm{GL(2,\ZZ[\omega])}\)
\begin{equation}\label{R0}
	R_0=\left(\begin{array}{cc} 1 & 2\omega+1 \\ 0 & -1\end{array}\right) =  S_0U_0^2S_0^{-1}U_0^{-1}B_0S_0^3B_0^{-1}.
\end{equation}	

As a consequence, we have following result.
\begin{proposition}\label{prop:altgen} The  action of  \(\mathrm{GL}(2,\ZZ[\omega])\) by birational transformations on \(\PP^2\) of Theorem~\ref{thm:bir-f4} is generated by the linear transformations \(B_4\) and \(T_4\), and by the standard quadratic Cremona involution \(Q_\infty\). The group of birational transformations of \(\PP^2\) that preserve the family \(\mathcal{F}_4\) is generated by these and by the linear map~\(L_4\).
\end{proposition}

\subsection{Birational symmetries of the other families}
We may give analogues of Theorem~\ref{thm:bir-f4} for the other families. By considering the birational maps induced by the first three transformations of Theorem~\ref{thm:bir-f4} on the quotient~(\ref{quot_propj}), we obtain:

\begin{theorem}\label{bir:3fam} The group of birational automorphisms of \(\PP^2\) that preserve the family  \(\mathcal{F}^3\)  is isomorphic to \(\mathrm{PGL}(2,\ZZ[\omega])\). It is generated by 
\[B_3(x:y:z) =  (  (2z-x)(x+z): x^2-3yz-xz+z^2:(x+z)^2),\] 
\begin{multline*} S_3(x:y:z) =   (
\omega^2(\omega zx+2yz-x^2)(y-\omega z)(y-\omega^2 z) : \\    :
\omega (3\omega xyz+\omega^2yz^2 -\omega x^3+y^2z)(y-\omega z): 
  z(y-\omega^2 z)(y^2+yz+z^2)),\end{multline*}
\begin{multline*} U_3(x:y:z)  =   (
    (2 x^2- \omega x y-\omega^2x z-4y z) (x+\omega y) (  z+\omega x)   : \\ 
   \;\;  : (x^2+\omega^2 y^2-3 y z-\omega x y)(z+\omega x)^2  
 : (x^2+\omega z^2-3yz-\omega^2 xz)(x+\omega y)^2), \end{multline*}
corresponding to the generators of \(\mathrm{GL}(2,\ZZ[\omega])\) of Proposition~\ref{gl2gen} through the birational  identification of \(\PP^2\) and \((E\times E)/\langle\rho\rangle\) given by Theorem~\ref{thm:ln-desc}
 and Proposition~\ref{prop-rel34}. \end{theorem}

Observe that, by the arguments in the proof of Proposition~\ref{prop:non-effective}, the action of the group of birational automorphisms of \(\PP^2\) that preserve the family  \(\mathcal{F}^3\) is  effective, this is, the only birational transformation  that preserves every foliation in \(\mathcal{F}^3\) is the identity.

Among the birational transformations of \(\mathcal{F}^3\) given by Theorem~\ref{bir:3fam}, we have the linear ones 
\begin{align}
	T_3(x:y:z) & =  (\omega x:y:\omega^2 z), \label{T3}\\
	R_3(x:y:z) & =(x:z:y),   \label{R3}
\end{align}
induced, respectively, by the transformations  \(T_0\) and \(R_0\) of Eqs.~(\ref{T0}) and (\ref{R0}). These act 
upon the irreducible components of the configuration \(\mathcal{C}\) of Eq.~(\ref{conf_f3}), given in Eqs.~(\ref{triangle_f3}) and~(\ref{conics_f3}),  by preserving  \(C_1\) and \(C_2\), and by  acting as generators of the symmetric group \(S_3\) on the lines of the triangle \(l_0l_1l_2\) of Eq.~(\ref{triangle_f3}):
\begin{align*}
	T_3^*(l_0)&=l_1, &  T_3^*(l_1)&=l_2,  & T_3^*(l_2)& =l_0, \\
	R_3^*(l_0)&=l_0, & R_3^*(l_1) &=l_2,  & R_3^*(l_2)&=l_1. 
\end{align*}
The transformation \(T_3\) acts upon the foliations by (\ref{actionT}), and for \(R_3\) we have
\begin{equation}\label{actionR}
	R^\dagger(\alpha) =\frac{1}{\alpha}. 
\end{equation}

The order-four linear symmetry \(B_4\) of \(\mathcal{F}^4\) produces the quadratic involution \(B_3\) for~\(\mathcal{F}^3\). In a way similar to Proposition~\ref{prop:altgen}, we have:

\begin{proposition} The group of birational transformations of \(\PP^2\) that preserve the family of foliations \(\mathcal{F}_3\) is  generated by the linear transformations \(T_3\) and \(R_3\), and by the  quadratic   involution~\(B_3\).
\end{proposition}

For the family \(\mathcal{F}^2\), we may establish  a result  analogous to Theorem~\ref{bir:3fam} by conjugating  the birational maps of Theorem~\ref{bir:3fam} by the birational map induced by~(\ref{3to2}).

\section{First integrals and the singularities of their level curves} 

In this section we establish  Theorem~\ref{thm:nodes},  as well as the related Corollary~\ref{coro:ln_degree}. We also give, in Theorem~\ref{node:f3}, an analogous result  for \(\mathcal{F}^3\). 

\subsection{The case of \texorpdfstring{\(\mathcal{F}^4}{F4}\) }\label{sec:thm_mult_f4} The  dual Hesse arrangement (\ref{arr:hesse})  contains a set of twelve points, which we will denote  by \(\mathcal{P}\), at which the lines of the arrangement intersect by triples, and at which a generic foliation of \(\mathcal{F}^4\) has nodes (dicritical, with ratio of eigenvalues \(1:1\), and local first integral \(x/y\)). The only foliations in \(\mathcal{F}^4\) which have degenerate singular points are the foliations \(\mathcal{F}_{\alpha}^{4}\) for \(\alpha\in \{1,\omega,\omega^2,\infty\}\). The degenerate points of such an \(\mathcal{F}_\alpha^4\) are those in \(\mathcal{P}_\alpha\), for the subsets of \(\mathcal{P}\)  defined in connection with Theorem~\ref{thm:nodes}.

Let us describe \(\mathcal{F}_{\infty}^{4}\). It has the cubic first integral  \(\phi_4=(x^3+y^3-2z^3)/(x^3-y^3)\), whose integral curves form a pencil of cubics with nine base points, those in \(\mathcal{P}\)  which are not in~\(\mathcal{P}_\infty\). At the points of \(\mathcal{P}_\infty\), the first integal \(\phi_4\) takes the values:  \(\infty\) at \((0:0:1)\), \(-1\)  at \((0:1:0)\), and \(1\) at \((1:0:0)\). After blowing up the twelve points of \(\mathcal{P}\), \(\phi_4\) becomes an elliptic fibration with three singular fibers above these three values. The nine lines of the arrangement, which pass by triples through the three points of~\(\mathcal{P}_\infty\),  become curves of self-intersection \(-3\), and the points in \(\mathcal{P}_\infty\),   curves of self-intersection \(-1\). Thus, each one of the three singular fibers of \(\phi_4\) is a blown-up fiber of type \(\mathrm{IV}\) within Kodaira's models of singular fibers of elliptic fibrations~\cite[Ch.~5]{BPHV},  as discussed in Section~\ref{tomados}. This is consistent with the description of Section~\ref{sec:deg4}: through the identification with \((E\times E)/\langle\rho^2\rangle\), \(\phi_4\) equals \(\wp'(u)\) (which is a first integral of \(\indel{v}\), and is invariant under the action of \(\rho^2\)). The singular fibers of \(\wp'(u)\) correspond to the three values taken by \(\wp'(u)\) at the  fixed points of the action of \(\rho_0\) on \(E\); these are: \(0\), the pole of \(\wp(u)\),  plus the two zeros of \(\wp(u)\). Thus, for  \(\alpha\in\QQ(\omega)\), the number of times that a leaf of \(\mathcal{F}_{\alpha}^{4}\) passes through a node in \(\mathcal{P}_\infty\) may be calculated by counting the number of points that one of the irreducible  components of its preimage in \(E\times E\) intersects a fiber of  the projection onto the first factor.

In \(\CC^2\), the line \(L_\lambda\) with slope \(\lambda\), \(v=\lambda u\),  intersects the line \(u=u_0 \) at the point \((u_0,\lambda u_0)\). By considering all the points where \(u_0\) belongs to \(\ZZ[\omega]\), we have that, in \(E\times  E\), the image of \(L_\lambda\) intersects the elliptic curve \(\{0\}\times E\) at the points \( \lambda \ZZ[\omega]/ \ZZ[\omega]\subset \CC/\ZZ[\omega]\). Let \(\lambda\in\QQ(\omega)\),  \(\lambda=p/q\), with \(p\) and \(q\) relatively prime in \(\ZZ[\omega]\). Observe that \((p/q)\ZZ[\omega]/\ZZ[\omega] =p\ZZ[\omega]/q\ZZ[\omega]\). We claim that \(p\ZZ[\omega]/q\ZZ[\omega]= \ZZ[\omega]/q\ZZ[\omega]\). Since \(p\) and \(q\) are relatively prime, there exist \(n\) and \(m\)  in \(\ZZ[\omega]\) such that \(mp+nq=1\). Taking this equality modulo \(q\) shows that \(1\in  p\ZZ[\omega]/q\ZZ[\omega]\). From the invariance of \(\ZZ[\omega]\) by multiplication by \(\omega\), we also have that \(\omega\in  p\ZZ[\omega]/q\ZZ[\omega]\), and we conclude that \(p\ZZ[\omega]/q\ZZ[\omega]=\ZZ[\omega]/q\ZZ[\omega]\). The kernel of the group homomorphism \(\CC/q\ZZ[\omega]\to \CC/\ZZ[\omega]\) is \(\ZZ[\omega]/q\ZZ[\omega]\), and, thus,  the cardinality of the latter equals the degree of the former as a covering map. In consequence, the number of points in \(\ZZ[\omega]/q\ZZ[\omega]\) equals   \[\frac{\mathrm{area}(\CC/q\ZZ[\omega])}{\mathrm{area}(\CC/\ZZ[\omega])}=N(q),\]
for \(N(q)\) the square of the complex norm of \(q\).

We have established that the image of \(L_{p/q}\) in \(E\times E\) intersects \(\{0\}\times E\) in \(N(q)\) points. The image of any translate \(L_{p/q}+c\) of \(L_{p/q}\) will intersect \(\{0\}\times E\) in a set with the same number of points, one obtained from the original one by a translation.

The curve \(\{0\}\times E\) is preserved by the action of \(\rho^2\), which acts upon it by the order-three automorphism \(\rho_0^2\colon v\mapsto \omega^2 v\). In the quotient \((\{0\}\times E)/\langle \rho^2\rangle \), with finitely many exceptions, the intersection of the image of a translation of \(L_{p/q}\) will still intersect \(\{0\}\times E\) at \(N(q)\) points, as the intersection of \(L_{p/q}+c\) with \(\{0\}\times E\)  will not have two points in the same orbit of \(\rho_0^2\).

Thus, \emph{in \((E\times E)/\langle\rho^2\rangle\), the image of a generic translate of \(L_{p/q}\)  intersects the image of \(\{0\}\times E\) in \(N(q)\) points; a generic leaf of the foliation \(\mathcal{F}_{\alpha}^{4}\) with slope \(p/q\) will pass through \((0:0:1)\) with multiplicity \(N(q)\)}. 

With finitely many  exceptions,  the image of a  translate of \(L_{p/q}\) will intersect the  other singular fibers \(\phi_4^{-1}(1)\) and  \(\phi_4^{-1}(-1)\) of \(\phi_4\) (that come from  translates of \(\{0\}\times E\)) along the same number of points. This establishes the first item in Theorem~\ref{thm:nodes}.  For the other ones, we will use the linear symmetries  \(T_4\) and \(B_4\), which, as we saw in the proof of Proposition~\ref{sl2f3}, act on \(\PP^2\) by  permuting  the foliations \(\mathcal{F}_{\alpha}^{4}\) for \(\alpha\in\{1,\omega,\omega^2,\infty\}\), the ones with degenerate singular points. From Eq.~(\ref{eq:slope}), the slopes  of \(\mathcal{F}_{\alpha}^{4}\) when \(\alpha\) takes the values \(\infty\), \(1\), \(\omega\) and \(\omega^2\) are,  respectively,   \(\infty\), \(0\), \(\omega+1\) and~\(\omega\). The transformation \(B_0\) exchanges the line with slope \(p/q\) and the one with slope \(-q/p\), while the associated map in \(E\times E\) exchanges \(\{0\} \times E\) (which is a leaf of the foliation with slope \(\infty\)) with \(E\times \{0\}\) (a leaf  of the foliation with slope~\(0\)), so that \(L_{p/q}\) intersects \(E\times \{0\}\) in as many points as \(L_{-q/p}\) intersects \(\{0\}\times E\). Thus, but for finitely many exceptions,  a  leaf of a foliation with slope \(p/q\) will pass through each one of the points of \(\mathcal{P}_1\) in \(N(p)\) points. This establishes the second item of Theorem~\ref{thm:nodes}. The remaining items follow from this last one. The action of  \(T_0\) on the slope of the foliations is given by \(T(\lambda)= \omega\lambda-\omega^2\). It 
maps the  slope \(\omega\) to the slope \(0\), while mapping the slope \(p/q\) to the slope \((p-\omega q)/\omega^2q\). Thus, in \(E\times E\), the line \(L_{p/q}\)  intersects a fiber of the foliation (fibration) of slope \(\omega\) in as many points as \(L_{(p-\omega q)/\omega^2q}\) intersects \(E\times\{0\}\), which is, by the previous reasoning, in \(N(p-\omega q)\) points. A repetition of the argument establishes the third item of Theorem~\ref{thm:nodes}. For the fourth, we have that \(T_0^2\) maps the slope \(\omega+1\) to the slope~\(0\), while mapping  the slope \(p/q\) to the slope \((p+\omega^2 q)/\omega q\); by the same reasons, in \(E\times E\), the line \(L_{p/q}\) intersects each leaf of the foliation with slope \(\omega+1\) in \(N(p+\omega^2 q)\) points. The fourth item of Theorem~\ref{thm:nodes} follows in the same way as before. This finishes the proof of Theorem~\ref{thm:nodes}.

\begin{remark}\label{rem:bru}
There is an inaccuracy in Brunella's geometric description of Lins Neto's family of degree four \cite[Ex.~8.4]{brunella-birational}, as he claims that, for the four blocks of three parallel elliptic curves each in \(E\times E\) that produce the twelve nodes of the dual Hesse configuration in \(\PP^2\), the slopes are in harmonic position and not, as we have  shown here, in an equianharmonic one.
\end{remark}
 
\begin{proof}[Proof of Corollary~\ref{coro:ln_degree}] Let \(C\) be a generic integral curve  of \(\mathcal{F}_\alpha^4\), and let \(d\) be its degree. It is an irreducible elliptic curve, passing through all the points of \(\mathcal{P}\), at which it is either regular or has an ordinary multiple point, and has no other singularity. For \(\kappa\in\{1,\omega,\omega^2,\infty\}\),  let $m_\kappa$ be the multiplicity of \(C\) at each one of the three points in $\mathcal{P}_\kappa$. In general, if \(m_p\) denotes the  multiplicity of   \(C\) at   \(p\), then, for the strict transform \(\widetilde{C}\) of \(C\) after blowing up~\(p\),  \(\widetilde{C}\cdot \widetilde{C}=C\cdot C-m_p^2\) (see~\cite[Lemma 8.1.6]{ctcwall}).
After blowing up the points  of \(\mathcal{P}\), the  strict transform of  the generic integral curve of the foliation  is smooth, and has vanishing self-intersection. Since, by B\'ezout's theorem, the self-intersection of \(C\) is \(d^2\),  
\[ d^2 = 3   \sum_{\kappa\in\{1,\omega, \omega^2,\infty\}} m_\kappa^2.\]	
On the other hand, by the genus formula, since \(C\) is elliptic  and  the singularities of \(C\) are ordinary multiple points,
	\[1 = \frac{(d-1)(d-2)}{2} - 3   \sum_{\kappa\in\{1,\omega,\omega^2,\infty\}} \frac{m_\kappa(m_\kappa-1)}{2}.\]
	From these  two equalities,  
	\[ d = \sum_{\kappa\in\{1,\omega,\omega^2,\infty\}} m_\kappa,\]
and from Theorem 2 we obtain the sought result. \end{proof}	 

The foliations associated to the values of \(\alpha\) that are excluded from Theorem~\ref{thm:nodes} are all linearly equivalent to \(\mathcal{F}_\infty^4\), as the linear map \(T_4\) permutes  cyclically the foliations \(\mathcal{F}_1^4\), \(\mathcal{F}_\omega^4\) and \(\mathcal{F}_{\omega^2}^4\), and the linear map \(B_4\) exchanges  \(\mathcal{F}_\infty^4\) and  \(\mathcal{F}_1^4\). In particular, all these foliations have a first integral of degree three.

\begin{example}\label{exo:degree} For \(\alpha_0=-1-2\omega\), from (\ref{eq:slope}), \(\mathcal{F}_{\alpha_0}^4\) has slope \(\lambda_0=-\frac{2}{3}+\frac{2}{3}\omega\), which equals \(p/q\) for the relatively prime Eisenstein integers \(p=-2-2\omega\) and \(q=1+2\omega\). We have  \(N(p)=4\), \(N(q)=3\), \(N(p+\omega^2 q)=7\) and \(N(p-\omega q)=1\). The first integral of \(\mathcal{F}_{\alpha_0}^4\) has, by our result, degree~\(4+3+7+1=15\).  In Example~\ref{ex:algo:f4} we will calculate this first integral with the algorithm presented in Section~\ref{sec:algo}.
\end{example}

\subsection{The case of \texorpdfstring{\(\mathcal{F}^3\)}{F3} }  
As a consequence of Theorem~\ref{thm:nodes} and Corollary~\ref{coro:ln_degree}, we will establish  their following analogue for the foliations of the family \(\mathcal{F}^3\):
\begin{theorem} \label{node:f3} Let \(\alpha\in\QQ(\omega)\), \(\alpha\notin\{0,1,\omega,\omega^2\}\), let \(\lambda\) be its slope, as in (\ref{eq:slope}). If \(\lambda=p/q\), with \(p\) and \(q\) relatively prime elements of \(\ZZ[\omega]\), the generic integral curve of  \(\mathcal{F}_{\alpha}^{3}\)  is an immersed curve, and its singularities are:
\begin{itemize}
	\item ordinary multiple points with multiplicity
	\begin{itemize}
	\item \(2N(p)\) at \((-1:1:1)\in l_1 \cap l_2\),		
	\item \(2N(p+\omega^2q)\)	at \((-\omega^2:\omega:1)\in l_0 \cap l_1\),	
	\item \(2N(p-\omega q)\)	at \((-\omega:\omega^2:1)\in l_0 \cap l_2\); and
\end{itemize}
	\item points with  multiplicity
	\begin{itemize}
		\item \(N(p)\) at \((2:1:1)\in C_1\cap l_0\),	
		 \item \(N(p+\omega^2q)\)	at \((2\omega:1:\omega^2)\in C_1\cap l_2\),
	 \item \(N(p-\omega q)\)	at \((2\omega^2:1:\omega)\in C_1\cap l_1\),	
	\item \(N(q)\)	at \((0:0:1)\in C_1\cap C_2\), and
		\item at \((0:1:0)\in C_1\cap C_2\), \begin{equation}\label{this_number}
			N(p)+N(p-\omega q)+N(p+\omega^2q)-N(q),	
		\end{equation}	
\end{itemize}
which consist of   smooth branches in the same number as their multiplicity, with a pairwise contact of order two.	
\end{itemize}
The degree of its first integral is 
\[ 2 N(p) + 2 N(p+\omega^2 q) + 2 N(p-\omega q).\]
\end{theorem}

\begin{proof} As  discussed in Section~\ref{sec:deg3}, the foliation \(\mathcal{F}_{\alpha}^{3}\) is the image of \(\mathcal{F}_{\alpha}^{4}\) under the degree-two map \(\tau\) of Eq.~(\ref{quot_propj}), which realizes birationally the quotient under the linear involution \(J\) of Eq.~(\ref{sym2}). The map \(\tau\) is a local diffeomorphism in the complement of the ramification divisor \(z^2(x-y)=0\). It ramifies doubly along the invariant line \(\Delta:x=y\), whose image is the invariant conic \(C_1\) of  Eq.~(\ref{conics_f3}). The points \((0:1:0)\) and \((1:0:0)\) are indeterminacy points, and the curve \(z=0\) that joins them is collapsed onto the point \((0:1:0)\).  Since \(\alpha\notin\{1,\omega,\omega^2,\infty\}\), all the singularities, and, in particular, the nodes of~\(\mathcal{F}_{\alpha}^{4}\), are non-degenerate.  Let \(C\) be a generic integral curve of \(\mathcal{F}_\alpha^3\). Let \(\widetilde{C}\) be an irreducible integral  curve of \(\mathcal{F}_\alpha^4\) that maps to \(C\).

The nodes of \(\mathcal{F}_\alpha^4\) away from the ramification divisor have as images nodes of \(\mathcal{F}_\alpha^3\) with eigenvalues \(1:1\):
\begin{itemize} 
	\item from \(\mathcal{P}_{1}\), \((1:\omega:\omega^2)\) and \((1:\omega^2:\omega)\) map to \((-1:1:1)\); 
	\item from \(\mathcal{P}_{\omega}\), \((1:\omega:1)\) and \((\omega:1:1)\) map to \((-\omega^2:\omega:1)\);  and
	\item from \(\mathcal{P}_{\omega^2}\), \((\omega^2:1:1)\) and \((1:\omega^2:1)\) map to \((-\omega:\omega^2:1)\).
\end{itemize}
The ordinary multiple points  of \(\widetilde{C}\) away from the ramification divisor, placed at these nodes, have as images ordinary multiple points of~\(C\); at each one of these points, the multiplicity of \(C\) is the sum of the multiplicities of \(\widetilde{C}\) at the two preimages of the point. We have already   calculated these in Theorem~\ref{thm:nodes}.

The  points of \(\mathcal{F}_{\alpha}^{4}\) on \(\Delta\), but away from \(z=0\), produce nodes with eigenvalues \(2:1\):
\begin{itemize} 
\item from \(\mathcal{P}_{1}\),  \((1:1:1)\)   maps to \((2:1:1)\); 
\item from \(\mathcal{P}_{\omega}\),  \((1:1:\omega)\)  maps to \((2\omega:1:\omega^2)\); 
\item from \(\mathcal{P}_{\omega^2}\),  \((1:1:\omega^2)\)   maps to \((2\omega^2:1:\omega)\); and
\item from \(\mathcal{P}_{\infty}\),   \((0:0:1)\)   maps to \((0:0:1)\).
\end{itemize}
At each one of these points, the multiplicity of \(C\) is the  multiplicity of \(\widetilde{C}\) at the preimage  of the point. We have have also calculated these in Theorem~\ref{thm:nodes}.

The eighth nodal point of \(\mathcal{F}_{\alpha}^{3}\), placed at \((0:1:0)\), has eigenvalues \(2:1\), and is the image of the line \(z=0\); let us give a more  detailed explanation of this process. The map \(\tau\) can be factored as the composition of:
\begin{itemize} 
	\item the birational map \(\PP^2\dashrightarrow\PP^1\times\PP^1\), \((x:y:z)\to(x:z,y:z)\),
which amounts to  blowing up \((1:0:0)\) and \((0:1:0)\) and collapsing the strict  transform of the line \(z=0\) that joins them, with 
\item the holomorphic map \(\PP^1\times\PP^1\to \PP^2\), \((x_0:x_1,y_0:y_1)\to (x_0y_1+x_1y_0:x_0y_0:x_1y_1)\), which realizes holomorphically the quotient under the permutation of the factors of \(\PP^1\times\PP^1\). 
\end{itemize}
Let us analyze the way in which the first map transforms \(\mathcal{F}_{\alpha}^{4}\). Since \(\alpha\neq 0\), the line \(z=0\) is not invariant by \(\mathcal{F}_{\alpha}^{4}\), and since \(\alpha\neq\infty\),  the singular points of \(\mathcal{F}_{\alpha}^{4}\) on this line, \((1:0:0)\) and \((0:1:0)\),  are dicritical nodes (everywhere transverse to the exceptional divisor after one blowup). Since  \(\mathcal{F}_{\alpha}^{4}\) has degree \(4\), it has no further singularities along \(z=0\). Thus, after blowing up the singular points on \(z=0\), \(\mathcal{F}_{\alpha}^{4}\) is everywhere transverse to the strict transform of \(z=0\). Since, as per Corollary~\ref{coro:ln_degree}, the first integral of \(\mathcal{F}_{\alpha}^{4}\) on \(\PP^2\) has degree \(N(p)+N(q)+N(p-\omega q)+N(p+\omega^2q)\), then, by B\'ezout's theorem, \(\widetilde{C}\) intersects \(z=0\) on this number of points. By Theorem~\ref{thm:nodes}, \(N(q)\) of these correspond to the node at \((1:0:0)\), and as many for the node at \((0:1:0)\). Thus, \(\widetilde{C}\) intersects the strict transform of \(z=0\) at the number of points \(\mu\) given in Eq.~(\ref{this_number}). Since this strict transform  has self-intersection \(-1\), we may collapse it to a point, at which the transformed of \(\mathcal{F}_\alpha^4\) will have a dicritical node of eigenvalues \(1:1\). This is the the point \((\infty,\infty)\) of \(\PP^1\times \PP^1\). At it, the image of \(\widetilde{C}\)  will have an ordinary multiple point with multiplicity \(\mu\). After quotient by the permutation of the factors, the node for the transform of \(\mathcal{F}_{\alpha}^{4}\) will produce a node with eigenvalues \(2:1\) for \(\mathcal{F}_{\alpha}^{3}\), and the image of \(\widetilde{C}\) will produce a point of \(C\) with multiplicity~\(\mu\), one having \(\mu\)  smooth branches with a pairwise contact of order two.  This proves the part of the statement concerning the multiplicities at the singular points. 
 
We now calculate the degrees of the first integrals. Let \(\mathcal{A}\) be the set of singular points of \(\mathcal{F}_\alpha^3\) with eigenvalues \(1:1\), and \(\mathcal{B}\) be the set of those with eigenvalues \(2:1\). For \(p\in\mathcal{A}\cup\mathcal{B}\), let \(m_p\) denote the multiplicity of \(C\) at \(p\). As in the proof of Corollary~\ref{coro:ln_degree}, for \(p\in \mathcal{A}\), \(p\) is either a regular or an ordinary multiple point of \(C\), and, when blowing it up, the self-intersection of \(C\) decreases by~\(m_p^2\). For \(p\in\mathcal{B}\), \(C\) is either regular at \(p\), or has \(m_p\) smooth branches at \(p\) with a contact of order two among any two of them. A first blow up of \(p\) reduces the self-intersection of \(C\) by \(m_p^2\), and, at the resulting exceptional divisor, \(C\) has an ordinary multiple point of the same multiplicity \(m_p^2\). Blowing up this point  decreases the self intersection of (the first strict transform of) \(C\) by  \(m_p^2\) once again.  Thus, by the same reasoning as before,
\[d^2=\sum_{p\in \mathcal{A}}m_p^2+2\sum_{p\in \mathcal{B}}m_p^2.\] 
The genus formula for $C$ reads:
\[1 = \frac{(d-1)(d-2)}{2} -\sum_{p\in \mathcal{A}}\frac{m_p(m_p-1)}{2} - \sum_{p\in\mathcal{B}} m_{p}(m_{p}-1).\]
From these formulas,
\[ 3 d = \sum_{p\in\mathcal{A}} m_p+2\sum_{p\in \mathcal{B}}m_p,\]
and, with the previous calculation of the multiplicities of \(\widetilde{C}\) at its singular points,
\begin{multline*}3d=(2N(p)+2N(p+\omega^2q)+2N(p-\omega q))+\\+2(N(p)+N(p+\omega^2q)+N(p-\omega q)+N(q))+\\+2(N(p)+N(p+\omega^2q)+N(p-\omega q)-N(q)),\end{multline*}
which establishes the assertion on the degree of the first integral.\end{proof}

\begin{example}\label{ex:f3_nonspecial} For \(\alpha=-\omega^2\), from~(\ref{eq:slope}), \(\lambda=p/q\) for the relatively prime Eisenstein integers \(p=\omega\) and \(q=1+2\omega\). We have  \(N(p)=1\),	  \(N(p+\omega^2q)=1\),   \(N(p-\omega q)=4\),  and \(N(q)=3\). By the theorem, the multiplicities of the generic integral curve at the nodes are of \(\mathcal{F}_\alpha^3\) are:  \(1\) at \((2:1:1)\); \(2\) at \((-1:1:1)\); \(1\)	at \((2\omega:1:\omega^2)\); \(2\)	at \((-\omega^2:\omega:1)\); \(4\)	at \((2\omega^2:1:\omega)\); \(8\)	at \((-\omega:\omega^2:1)\); \(3\)	at \((0:0:1)\); and \(3\) at \((0:1:0)\); the first integral has thus, according to the theorem, degree \(2(1+1+4)=12\). In Example~\ref{ex:algo:f3}, we  will calculate this first integral via the algorithm of Section~\ref{sec:algo}.
\end{example}

\subsection{The special foliations  in \texorpdfstring{\(\mathcal{F}^3\)}{F3}} 
The foliations in \(\mathcal{F}^3\) that have degenerate singularities, which we will call \emph{special}, are those whose parameter \(\alpha\) belongs to \(\{0,1,\omega,\omega^2,\infty\}\), those which are not covered by Theorem~\ref{node:f3}. They correspond to the values of the parameter where there is a coalescence of singular points (these five special values of \(\alpha\) may be characterized as the values where the mobile singular point of \(\mathcal{F}_{\alpha}^{3}\) on \(C_1\) passes through each one of  the five singularities of \(\mathcal{C}\) on \(C_1\)).

Let us begin by describing some of these special elements, one in each linear equivalence class. From (\ref{actionT}), the linear transformation  $T_3$ of Eq.~(\ref{T3})  cyclically  permutes  \(\mathcal{F}^3_1\), \(\mathcal{F}^3_{\omega}\) and \(\mathcal{F}^3_{\omega^2}\), which are thus linearly equivalent;  on its turn, from~(\ref{actionR}),  the linear map \(R_3\) of Eq.~(\ref{R3}) exchanges \(\mathcal{F}^3_\infty\) and \(\mathcal{F}^3_0\) and gives a linear equivalence between them. While the foliations $\mathcal{F}_{\infty}^{4}$ and $\mathcal{F}_{1}^{4}$ are linearly equivalent,  $\mathcal{F}_{\infty}^{3}$ and~$\mathcal{F}_{1}^{3}$, being, as we will see, pencils of curves of different degrees, are not.

\subsubsection{The foliation $\mathcal{F}_{\infty}^{3}$}  It has a first integral \(\phi_3\), of degree six, given, for the first integral \(\phi_4\) of \(\mathcal{F}_\infty^4\)  discussed in Section~\ref{sec:thm_mult_f4},  by the image of \(\phi_4^2\) in the quotient; it reads 
\begin{equation}\label{fiphi3}
\phi_3=\frac{Q_1^2}{C_1C_2^2},
\end{equation}
for the conics \(C_i\) defined in (\ref{conics_f3}) plus   \(Q_1 = x^3-3xyz-2z^3\), which defines a nodal cubic. The generic integral curve of $\mathcal{F}^3_\infty$ is an elliptic sextic having
\begin{itemize}
	\item  a triple point at $(0:1:0)$ with an infinitely near ordinary triple point, plus
	\item  ordinary double points at each one of the three vertices of the triangle $l_0l_1l_2$;
\end{itemize}
it passes through the three tangencies  between $C_1$ and the triangle $l_0l_1l_2$, in a way tangent to \(C_1\), and  does  not pass through the point $(0:0:1)$.   The pencil has a third special element,  \(z^3l_0l_1l_2\), at which the first integral takes the value~\(1\). Upon blowing up the 13 base points of the pencil (the 13 blow-ups needed to transform the configuration \(\mathcal{C}\) of Eq.~(\ref{conf_f3}) into a normal crossings one), we obtain an elliptic fibration on a surface that we will denote by $\mathrm{Bl}_{13}(\PP^2)$,  with singular fibers above the values \(0\), \(1\) and \(\infty\) of the first integral. This is sketched in Figure~\ref{fibration-oo-3}. In agreement with the discussion in Section~\ref{tomados}, from the three special fibers of \(\phi_4\), which are  singular fibers of type \(\mathrm{IV}\)  where \(\phi_4\) takes, respectively, the values \(-1\), \(1\) and~\(\infty\), the first two  come together in \(\mathcal{F}_\infty^3\) into a single fiber of type \(\mathrm{IV}\), and, out of the third one, the ramification of \(\phi_4\mapsto \phi_4^2\) at \(\infty\) produces a fiber of type \(\mathrm{II}\). The ramification  of \(\phi_4\mapsto \phi_4^2\) at \(0\) produces a fiber  of type  \(\mathrm{I_0^*}\) out of a regular fiber.

\begin{figure}
\centering	
\includegraphics[width=0.9\textwidth]{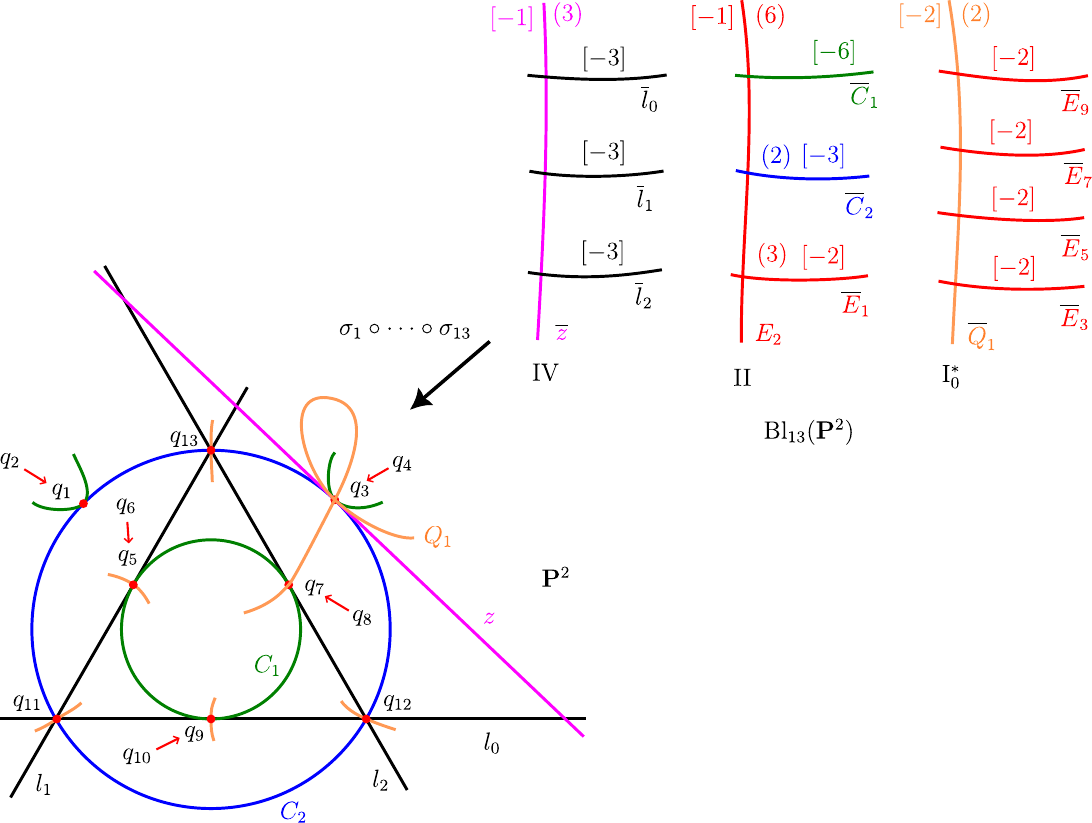}
\caption{Elliptic fibration in \(\mathrm{Bl}_{13}(\PP^2)\) given by the pencil $\mathcal{F}_{\infty}^{3}$.  For each relevant divisor, the multiplicity with which it  appears in the first integral is marked in parenthesis (if it is different from~\(1\)); its self-intersection, in square brackets. The strict transform of the curve \(F\) is denoted by \(\overline{F}\). The exceptional divisor obtained from blowing up the point \(p_i\) is denoted by~\(E_i\).} 
	\label{fibration-oo-3}
\end{figure}

The Kummer surface \(\mathrm{Kum}(E\times E)\) discussed in Section~\ref{tomados} may be recovered as an elliptic fibration  by taking the ramified triple cover of the rational curve where \(\phi_3\) takes its values, following the ramified triple cover of the orbifold \(O(3,3,3)\) by the equianharmonic orbifold \(O^*(2,2,2,2)\)  discussed in that same section.

\subsubsection{The foliation  $\mathcal{F}_{1}^{3}$}  It is given by the pencil of quartics \(l_1 l_2 C_2 : l_0^2 C_1 \), and its third special element is given by twice the conic \(C_3: x^2-2yz+xz+xy\). This pencil has the same base points as the pencil \(\mathcal{F}_\infty^3\), and becomes a fibration in the same blown-up surface \(\mathrm{Bl}_{13}(\PP^2)\). The self-intersections of the curves in $\mathrm{Bl}_{13}(\PP^2)$ are, of course, the same, but their position, their multiplicities, and the role they play in the singular fibers  change. All this is sketched in  Figure~\ref{fibration-1-3} (compare it with Figure~\ref{fibration-oo-3}).

\begin{figure}
\centering	
\includegraphics[width=0.9\textwidth]{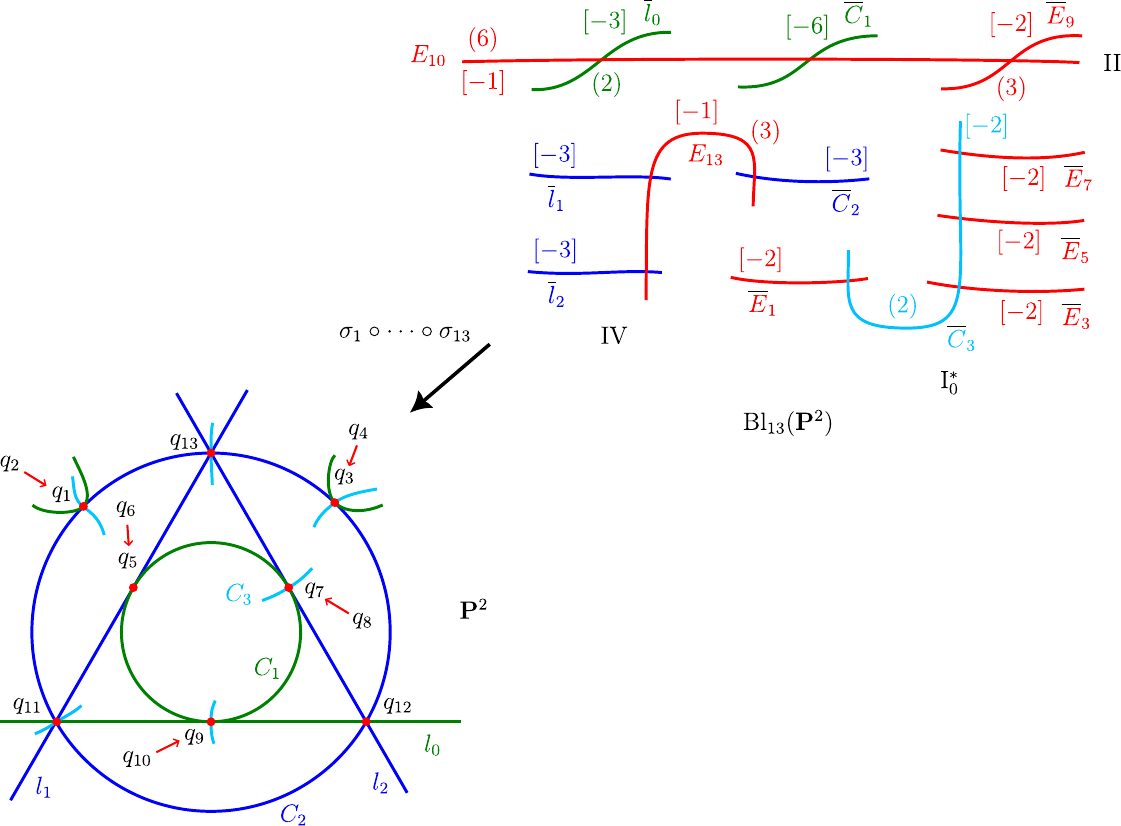}
	\caption{The foliation  $\mathcal{F}_{1}^{3}$ as an elliptic fibration on \(\mathrm{Bl}_{13}(\PP^2)\).}
	\label{fibration-1-3}
\end{figure}

The foliations $\mathcal{F}_{\infty}^{4}$ and $\mathcal{F}_{1}^{4}$ are exchanged by the quadratic  involution \(B_3\) of Theorem~\ref{bir:3fam}. The indeterminacy points of \(B_3\) are given by the point $(-1:1:1)$, the point $(0:1:0)$, and an infinitely near point above this last one in the direction of the line  $z=0$;  \(B_3\)  contracts the  line $x+z=0$ onto the point $(0:1:0)$, and the  line $z=0$ onto the point $(-1:1:1)$; conversely, these points are blown-up by \(B_3\) into the corresponding lines. For the  strict transforms of   the curves of the configurations, \(B_3(Q_1) = (C_3)\), \(B_3(l_0)=  C_2\), \(B_3(l_1)= l_2\), and \(B_3(C_1) = (C_1)\).

\section{An algorithm for the first integrals} \label{sec:algo}

As discussed  in the Introduction, we here  present an algorithm that, for every \(\alpha\in\QQ(\omega)\), gives a  birational transformation of \(\PP^2\) that  maps \(\mathcal{F}_{\alpha}^{4}\) into the foliation \(\mathcal{F}_{\infty}^{4}\).  This rational map is produced by suitably composing some of the maps of Theorem~\ref{thm:bir-f4}. In particular, by pulling-back the first integral of \(\mathcal{F}_{\infty}^{4}\), we may obtain an explicit first integral of \(\mathcal{F}_{\alpha}^{4}\).

Let \(\lambda\in\QQ(\omega)\). Let  \(a\) and \(b\) be relatively prime integers in \(\ZZ[\omega]\) such that \(\lambda=-a/b\).  Since \(a\) and \(b\) are relatively prime, there exist \(c\) and \(d\) in \(\ZZ[\omega]\) such that \(ad-bc=1\), and thus
\(\left(\begin{array}{cc} a & b \\ c & d\end{array}\right)\left(\begin{array}{r}-b \\ a \end{array}\right)=\left(\begin{array}{c} 0 \\ 1 \end{array}\right)\). This element of \(\mathrm{GL}(2,\ZZ[\omega])\) maps the foliation with slope \(\lambda\) on \(E\times E\) (the one generated by the vector field \(\indel{u}+\lambda \indel{v}\), or, equivalently, the one given by the line \(v=\lambda u\) and the ones parallel to it) into the foliation with slope \(\infty\) (the one generated by \(\indel{v}\)). For \(\alpha\) given by (\ref{lambda_to_alpha}), the birational transformation of \(\PP^2\) associated to the above linear one will transform the foliation \(\mathcal{F}_{\alpha}^{4}\) into the foliation~\(\mathcal{F}_{\infty}^{4}\). In order to implement this in an algorithmic way, we will construct the matrix and its factorization into the generators of Proposition~\ref{gl2gen} simultaneously.

Consider the slightly modified version of Euclid's algorithm:
\begin{align*}
	a & =   q_1 b -r_1 \\
	b & =   q_2 r_1 -r_2 \\ 
	r_1 & =   q_3 r_2 -r_3 \\
	&  \vdots   \\
	r_{n-1} & =   q_{n+1} r_{n} -r_{n+1} \\
	r_n & =   q_{n+2} r_{n+1},   
\end{align*}
so that
\[
\frac{a}{b}=q_1-\cfrac{1}{q_2-\cfrac{1}{\ddots-\cfrac{1}{q_{n+2}}}}.
\]
For \(q\in\ZZ[\omega]\), let
\(W_q=\left(\begin{array}{cc} q & 1 \\ -1 & 0 \end{array}\right)\). We have that
\(W_{q_1}\left(\begin{array}{r} -b \\ a \end{array}\right)=\left(\begin{array}{r} -r_1 \\ b \end{array}\right)\),
and thus \[W_{q_{n+2}}W_{q_{n+1}}\cdots  W_{q_1}\left(\begin{array}{r} -b \\ a \end{array}\right)=\left(\begin{array}{c} 0 \\ r_{n+1} \end{array}\right).\]
From
\(W_q =\left(\begin{array}{cc} 1 & -q \\ 0 & 1 \end{array}\right)B_0\),
we have, with the identity (\ref{unipotent}), a factorization of \(W_q\) into the generators of Proposition~\ref{gl2gen}: for \(m,n\in\ZZ\),
\begin{equation}\label{factor-Wq}
	W_{m+n\omega}=S_0U_0^{n}S_0^{-1}U_0^{-m}B_0.
\end{equation}

We have constructed an element \(A_\lambda=W_{q_{n+2}}\cdots W_{q_1}\) of \(\mathrm{GL}(2,\ZZ[\omega])\) as an explicit product of the generators of Proposition~\ref{gl2gen}, that, through its action on \(\CC^2\), maps the vector field
\(\indel{u}+\lambda\indel{v}\) to a multiple of \(\indel{v}\). (This element actually belongs to \(\mathrm{SL}(2,\ZZ[\omega])\), for so do the \(W_q\).)

Through the mapping \(\Phi_*\colon \mathrm{GL}(2,\ZZ[\omega])\to \mathrm{Bir}(\PP^2)\) induced by the map \(\Phi\) of Theorem~\ref{thm:ln-desc}, the birational transformation \(\Phi_*(A_\lambda)\) maps the foliation in \(\mathcal{F}^4\) of slope \(\lambda\) into~\(\mathcal{F}_{\infty}^{4}\). We have \(\Phi_*(A_\lambda)=\Phi_*(W_{q_{n+2}})\circ \cdots \circ\Phi_*(W_{q_1})\), and with~(\ref{factor-Wq}), \(\Phi_*(W_{m+n\omega})=S_4  U_4^{n}  S_4^{-1} U_4^{-m}  B_4\). This allows, in principle, to explicitly write  down \(\Phi_*(A_\lambda)\).

\begin{example}\label{ex:algo:f4} Let us revisit the foliation of  \(\mathcal{F}^4\) of Example \ref{exo:degree}. For \(\alpha_0=-1-2\omega\),   consider the foliation \(\mathcal{F}_{\alpha_0}^{4}\), which, from (\ref{eq:slope}), has slope  \(\lambda_0=-\frac{2}{3}+\frac{2}{3}\omega\). We have  that \(\lambda_0=-a/b\) for the relatively prime Eisenstein integers \(a=2+2\omega\) and  \(b=1+2\omega\). From the slightly modified version Euclid's algorithm previously discussed,  
\begin{align*}
	a & =   (1-\omega) b -(\omega+1), \\
	b & =   (2+\omega)(\omega+1),   
\end{align*}
(since their greatest common divisor is the unit \(\omega+1\), \(a\) and \(b\) are indeed relatively prime) and thus
\[\frac{a}{b}=(1-\omega)-\frac{1}{2+\omega},\]
so that, with (\ref{factor-Wq}), \(W_{2+\omega}W_{1-\omega}\) reads
\begin{equation}\label{exo:factor}
(S_0U_0S_0^{-1}U_0^{-2}B_0)
(S_0U_0^{-1}S_0^{-1}U_0^{-1}B_0)
=\left(\begin{array}{cc}  2 &  2+\omega \\ -1+\omega & -1 \end{array}\right).\end{equation}
As expected,
\[\left(\begin{array}{cc}  2 &  2+\omega \\ -1+\omega & -1 \end{array}\right) \left(\begin{array}{c} -(1+2\omega) \\ 2+2\omega \end{array}\right)=\left(\begin{array}{c} 0 \\ \omega+1 \end{array}\right).\]

Lins Neto's first integral of \(\mathcal{F}_{\infty}^{4}\) is
\begin{equation} \label{H_infty}
	H_\infty=\frac{y^3-z^3}{x^3-z^3}.
\end{equation}
Its pull-back by the map corresponding to (\ref{exo:factor}),
\begin{equation}\label{exo:pull}B_4^*(U_4^{-1})^*(S_4^{-1})^*(U_4^{-1})^*S_4^*B_4^{*}(U_4^{-2})^*(S_4^{-1})^*U_4^*S_4^*(H_\infty)\end{equation}
(which has, after simplification, degree nine), gives the rational function
\[\frac{(x^3-z^3) (\omega y^4+yx^3+y z^3-(1+2 \omega) x^2 z^2-(1-\omega) x y^2 z )^3}{(y^3-z^3)(\omega x^4+x y^3+ x z^3-(1+2\omega) y^2 z^2-(1-\omega)x^2y z)^3}.\] 
By construction, it is a first integral of \(\mathcal{F}_{\alpha_0}^{4}\). Its degree, \(15\), is the one established in~\cite{puchuri}, in agreement with Corollary~\ref{coro:ln_degree} (cf. Example \ref{exo:degree}).
\end{example}

Remark the reversal in the order of the factorizations in Eqs.~(\ref{exo:factor}) and~(\ref{exo:pull}), consistent with the remarks at the end of Section~\ref{affquot}. 

In the above example, chosen for its simplicity, the parameter \(\alpha_0\) is an Eisenstein integer, and the first integral of \(\mathcal{F}_{\alpha_0}^4\) may also be calculated via the algorithm of~\cite{mendes-puchuri-effective}. The next example does not fall within the scope of the latter:

\begin{example} Consider \(\mathcal{F}_{\alpha_0}^4\) for $\alpha_0 = -\frac{1}{2}(3+\omega)$, the foliation with slope $\lambda= -a/b$, for the relatively prime Eisenstein integers \(a=2-\omega\) and \(b=2+2 \omega\).
According to the degree formula, from either \cite{puchuri} or Corollary~\ref{coro:ln_degree}, the first integral has degree~$15$. The modified version of Euclid's algorithm reads:
\begin{align*}
	a & =   (-2\omega) b -(2+\omega), \\
	b & =   (1+\omega) (2+\omega)-(-1),   \\
	(2+\omega) & = (-2-\omega)(-1),
\end{align*}
and thus
\[\frac{a}{b}=-2\omega-\cfrac{1}{(1+\omega)-\cfrac{1}{-2-\omega}}.\]
As in the previous example, for \(A_\lambda=W_{-2-\omega}W_{1+\omega}W_{-2\omega}\),
\[\Phi_*(A_{\lambda})=(S_4  U_4^{-1}  S_4^{-1} U_4^{2}  B_4)(S_4  U_4  S_4^{-1} U_4^{-1}  B_4)(S_4  U_4^{-2}  S_4^{-1}    B_4).\]
The action of the latter via pull-back on   the rational function \(H_\infty\) of Eq.~(\ref{H_infty}) gives
\[(\Phi_*(A_{\lambda_0}))^*(H_\infty)=\frac{(x-\omega^2y) (x-z) (\omega z-y) }{(x-\omega y) (x-\omega z) (y-z) }\times \frac{f^3}{g^3},\]
for 
\[f=y^2 z^2-y^3 z-\omega^2y z^3-x^3 z+\omega^2x^2 z^2-x z^3+x^2 y z+x z^2 y+\omega^2 y^2 x z-\omega^2x^3 y+y^2 x^2-y^3 x,\]
\[g=y^3 z-\omega^2y^2 z^2+y z^3+x^3 z-x^2 z^2+\omega^2x z^3-\omega^2x^2 y z-x z^2 y-y^2 x z+x^3 y-y^2 x^2+\omega^2y^3 x.\]
By construction, this is a rational first integral of \(\mathcal{F}_{\alpha_0}^4\). Its general level curve has geometrical genus \(1\), having, by Theorem~\ref{thm:nodes}, ordinary quadruple points at the points of $\mathcal{P}_\infty$; ordinary septuple points at those of  $\mathcal{P}_1$; ordinary triple points at those of $\mathcal{P}_{\omega}$; and regular  points at those of $\mathcal{P}_{\omega^2}$.
The quartics which appear in  its numerator and denominator are rational, and have ordinary double points at the three points of~$\mathcal{P}_1$.
\end{example}

The algorithm can also the applied, via Theorem~\ref{bir:3fam}, to foliations in \(\mathcal{F}^3\):
\begin{example}\label{ex:algo:f3} Consider the foliation \(\mathcal{F}_{-\omega^2}^3\) of  the family of degree three appearing in Example~\ref{ex:f3_nonspecial}. For \(\alpha_0=-\omega^2\), we have the slope  \(\lambda=-a/b\) for the relatively prime Eisenstein integers \(a=-\omega\) and \(b=1+2\omega\). The variant of Euclid's algorithm reads
	\begin{align*}
		a & =  (-1) b -(-1-\omega), \\
		b & =   (-2-\omega)(-1-\omega),   
	\end{align*}
and thus
\[\frac{a}{b}=-1-\frac{1}{-2-\omega}.\]
We have
\[W_{-2-\omega}W_{-1} = (S_0U_0^{-1}S_0^{-1}U_0^{2}B_0) (U_0B_0)
		=\left(\begin{array}{cc}  1+\omega &  -2-\omega \\ 1 & 1 \end{array}\right).\]
The rational function  \(\phi_3\) of Eq.~(\ref{fiphi3}) is a first integral of \(\mathcal{F}_{\infty}^{3}\). The pull-back of \(\phi_3-1\) by the map corresponding to this last expression via Theorem~\ref{bir:3fam},
	\[B_3^*U_3^*B_3^*(U_3^2)^*  (S_3^{-1})^*(U_3^{-1})^*S_3^*(\phi_3-1),\]
gives the rational function
\[ -4\frac{l_0 l_1 l_2(x^3-3xyz-\omega^2y^2z-\omega yz^2)^3}{C_1C_2^2(y-\omega^2 z)^6},\]
expressed in terms of the polynomials of Eqs.~(\ref{triangle_f3}) and (\ref{conics_f3}).	
By construction, it is a first integral of~\(\mathcal{F}_{-\omega^2}^3\).
\end{example}

\section{Lins Neto's foliations in positive characteristic}

Lins Neto's foliations may be defined over many algebraically closed of positive characteristic.  
Our main results in this setting establish which foliations of \(\mathcal{F}^4\) are algebraically integrable over fields of positive characteristic (Theorem~\ref{thm:charp}) and, more generally, describe their invariant algebraic curves  (Theorem~\ref{thm:p-div}).

\subsection{The \texorpdfstring{$p$}{p}-divisor of a foliation}

We begin by recalling some standard notions concerning derivations and foliations in positive characteristic; we refer the reader to \cite{mendson2022foliations} for further details.

Let \(k\) be a field of characteristic \(p>0\), let \(R\) be a finitely generated integral \(k\)-algebra, and let \(v\colon R\to R\) 
be a \(k\)-derivation.  Let \(v^{[i]}\) denote the \(i\)-fold composition of \(v\) (\(v^{[1]}=v\), \(v^{[i+1]}=v\circ v^{[i]}\)). We have that  \(v^{[p]}\) is again a \(k\)-derivation:  by the general Leibniz rule, and, since, for \(0<i<p\),  
\(\binom{p}{i}\equiv 0 \pmod p\), 
\[
    v^{[p]}(fg)
    =
    \sum_{i=0}^{p}
    \binom{p}{i}
    v^{[p-i]}(f)v^{[i]}(g)=v^{[p]}(f)g+fv^{[p]}(g).
\]

Every \(p\)-th power belongs to the kernel of~\(v\). More precisely, for every \(f\in R\),  \( v(f^p)=pf^{p-1}v(f)=0\).
Thus, \(v\) is naturally linear over the subring \(R^p\subset R\). Also, for every \(f\in R\),  
\[(fv)^{[p]} = f^pv^{[p]}-fv^{[p-1]}(f^{p-1})v\]
(see \cite[Lemma~1, p.~481]{Hochschild1955};  use \(U = \operatorname{End}_{\FF_p}(R)\), with \(V\) the collection of  operators given by multiplication by functions). In particular, for the rank-one distribution \(Rv\) generated by~\(v\),  
\[(fv)^{[p]}\equiv f^pv^{[p]} \mod Rv.\] 
Therefore, if one considers the class of \(v^{[p]}\) modulo \(Rv\), replacing \(v\) by another generator \(fv\) multiplies this class by \(f^p\). This will be the basic fact underlying the upcoming definition of the \emph{\(p\)-divisor} of a foliation in characteristic \(p\).

The key notion in the study of foliations in positive characteristic is the following:
\begin{definition}
Let \(X\) be a smooth algebraic variety over an algebraically closed field \(k\) of characteristic \(p>0\), and let \(\mathcal F\) be a foliation by curves on \(X\). We say that \(\mathcal F\) is \emph{\(p\)-closed} if, for every local generator \(v\) of \(T_{\mathcal F}\), the derivation \(v^{[p]}\) is again tangent to \(\mathcal F\), or, equivalently, if $v\wedge v^{[p]}=0$.
\end{definition}

The notion is equivalent to integrability:  \emph{If \(\mathcal F\) is a foliation on
\(X\), then \(\mathcal F\) is \(p\)-closed if and only if there exists a variety \(Y\) and a dominant rational map \(f\colon X\dashrightarrow Y\) such that, on a dense open subset where \(f\) is regular, $\mathcal F=\ker(Df)$} \cite[Prop.~1.9]{MR1468476}. In particular, \emph{a foliation on a smooth algebraic variety over an algebraically closed field of positive characteristic \(p\) is \(p\)-closed if and only if it is algebraically integrable.}

Let \(X\) be a smooth algebraic surface,   \(\mathcal{F}\) a foliation by curves on \(X\). Let \(\{(U_i,\omega_i,v_i)\}_{i\in I}\) be a local description of \(\mathcal F\), where, on \(U_i\), \(v_i\) generates \(T_{\mathcal F}\) and \(\omega_i\) generates \(N_{\mathcal F}^{*}\), so that, on \(U_i\cap U_j\), we have \(\omega_i=f_{ij}\omega_j\)  and   \(v_i=g_{ij}v_j\)  for \(f_{ij},g_{ij}\in\mathcal O_X^{*}(U_i\cap U_j)\).  By denoting by \(v\lrcorner \omega \) the contraction of the form \(\omega\) by the derivation \(v\), for each $i,j \in I$ we have, on \(U_i\cap U_j\), 
\[   v_{i}^{[p]}\lrcorner\omega_{i} = (g_{ij}v_{j})^{[p]}\lrcorner (f_{ij}\omega_{j}) = (g_{ij}^{p}v_{j}^{[p]}-g_{ij}v_{j}^{[p-1]}(g_{ij}^{p-1})v_{j}) \lrcorner (f_{ij}\omega_{j}) = (g_{ij}^{p}f_{ij})v_{j}^{[p]} \lrcorner \omega_{j}.\]
Thus, the collection $\{v_{i}^{[p]} \lrcorner \omega_{i}\}_{i\in I}$  determines a global section  \[s_{\mathcal{F}} \in H^{0}(X,K_{\mathcal{F}}^{\otimes p}\otimes N_{\mathcal{F}}),\] which vanishes identically if and only if \(\mathcal{F}\) is \(p\)-closed. If \(\mathcal{F}\) is not \(p\)-closed, we define its  \emph{$p$-divisor} \(\Delta_{\mathcal{F}}\) as the zero divisor of $s_{\mathcal{F}}$: \(\Delta_{\mathcal{F}}=(s_{\mathcal{F}})_0\in\mathrm{Div}(X)\). 

\begin{remark}\label{pdivP2}
If  $\mathcal{F}$ is a foliation of degree $d$ on $\PP_{k}^{2}$ which is not \(p\)-closed, $\deg(\Delta_{\mathcal F})=p(d-1)+d+2$, since,  for a foliation of degree $d$ on \(\PP^2_k\),  we have that $K_{\mathcal F}=\mathcal{O}(d-1)$ and $ N_{\mathcal F}=\mathcal{O}(d+2)$.\end{remark}

The following is a  fundamental property of the \(p\)-divisor of a foliation. Let \(X\) be a smooth algebraic surface over \(k\), let \(\mathcal F\) be a non-\(p\)-closed foliation on \(X\), and let \(C\subset X\) be an irreducible algebraic curve. \emph{If \(C\) is \(\mathcal F\)-invariant, then \(\operatorname{ord}_C(\Delta_{\mathcal F})>0\); conversely, if $p    \nmid \operatorname{ord}_C(\Delta_{\mathcal F})$, then \(C\) is \(\mathcal F\)-invariant}  \cite[Prop.~3.8]{mendson2022foliations}.

Recall that a morphism $f\colon X\to Y$ between smooth algebraic surfaces is \emph{étale} if for every \(x\in X\), the differential \(Df_x\colon T_xX\to T_{f(x)}Y\) is an isomorphism (see \cite[Sect.~4.3.2 or Prop.~3.2.6]{MR1917232} for some equivalent conditions). The \(p\)-divisor is well-behaved under étale morphisms:

\begin{proposition}\label{sec:etalepdivisor}
Let $f\colon X\to Y$ be an étale morphism between smooth algebraic surfaces over a field of characteristic \(p>0\). Let \(\mathcal{F}\) be a foliation on \(Y\), and assume that \(\mathcal{F}\) is not \(p\)-closed. Then \(f^*\mathcal F\) is not
\(p\)-closed and $\Delta_{f^*\mathcal F}=f^*\Delta_{\mathcal F}.$
\end{proposition}

\begin{proof}
Set $\mathcal G=f^*\mathcal F$ and let \(\omega\) be a local generator of \(N_{\mathcal F}^*\) on an open set
\(U\subset Y\). Since \(f\) is étale,
\(f^*\omega\) is a local generator of \(N_{\mathcal G}^*\) on
\(f^{-1}(U)\). Hence,
$N_{\mathcal G}=f^*N_{\mathcal F}$. Moreover, since \(f\) is étale, $K_X=f^*K_Y$, and by the adjunction formula for \(\mathcal F\) and \(\mathcal G\), we obtain that
$
K_{\mathcal G}+N_{\mathcal G}^*
=
K_X
=
f^*(K_{\mathcal F}+N_{\mathcal F}^*).
$
Since $N_{\mathcal G}^*=f^*N_{\mathcal F}^*$, it follows that $K_{\mathcal G}=f^*K_{\mathcal F}$. Let \(v\) be a local vector field defining \(\mathcal F\) on \(U\). Since
\(f\) is étale, there exists a unique vector field \(\widetilde v\) on
\(f^{-1}(U)\) that is \(f\)-related to \(v\): \(Df(\widetilde{v}) = v\). Consequently,
\(\widetilde v^{[p]}\) is \(f\)-related to \(v^{[p]}\). Therefore,  $\widetilde v^{[p]}\lrcorner(f^*\omega)
=
f^*( v^{[p]}\lrcorner\omega).$ We thus have  \(s_{\mathcal G}=f^*s_{\mathcal F}\). Since \(s_{\mathcal F}\neq0\), it follows that \(s_{\mathcal G}\neq0\); hence, \(\mathcal G\) is not \(p\)-closed. Taking zero divisors yields  \(\Delta_{\mathcal G}
    =
    (s_{\mathcal G})_0
    =
    (f^*s_{\mathcal F})_0
    =
    f^*(\Delta_{\mathcal F})\). 
\end{proof}

\begin{corollary}\label{cor:allequal} If \(\mathcal F\) is a non-\(p\)-closed foliation on \(X\), and $f: X\to X$ is an automorphism preserving \(\mathcal{F}\), the \(p\)-divisor of \(\mathcal{F}\) is \(f\)-invariant, $f^*\Delta_{\mathcal F}=\Delta_{\mathcal F}$. 
\end{corollary}

\subsection{The elliptic curve and its supersingularness} Over the algebraically closed field \(k\) of characteristic $p> 3$, we consider the complete elliptic curve \(E\) defined by \(\zeta^2 = 4\xi^3+1\), along with the derivation \(D_E\) on \(E\) induced by  \(\zeta\indel{\xi}+6\xi^2\indel{\zeta}\). The latter is nowhere zero on \(E\), as it is dual to the nowhere vanishing regular differential $\dd \xi/\zeta$. Since \(H^0(E,T_E)\) is one-dimensional, and generated by \(D_E\), there exists \(a_E\in k\) such that \(D_E^{[p]}=a_E D_E\).

\begin{proposition}\label{supersingularFp} With the previous notations, \(a_E=0\) if and only if $p\equiv 2\pmod 3.$
\end{proposition}

\begin{proof} Both conditions are  equivalent to \(E\) being \emph{supersingular}, which is also equivalent to the  curve having \emph{Hasse invariant}  (also called \emph{\(p\)-rank}) equal to~\(0\)  (see~\cite[Ch.~V]{silverman}, \cite[\S  15]{MumfordAbelian74} for definitions and equivalences).

We have that \(H^0(E,T_E)\) is one-dimensional, and generated by \(D_E\), which is, moreover, invariant by left translations. In this  case, the \(p\)-rank (or {Hasse invariant}) of \(E\) equals the rank of the \(p\)-th power map on \(H^0(E,T_E)\), and  thus the \(p\)-rank of \(E\) is \(0\) if \(a_E=0\), and \(1\) if \(a_E\neq 0\) (see~\cite[\S  14--15]{MumfordAbelian74}).

We have the following criterion \cite[Thm~4.1,  Ch.~V]{silverman}:  \emph{The elliptic curve
\(\zeta^2=f(\xi)\) with \(f\) a cubic polynomial in \(\FF_p [\xi]\) with different roots in \(k\) has Hasse invariant equal to \(0\) if and only if the coefficient of \(\xi^{p-1}\) in \(f(\xi)^{(p-1)/2}\) is zero}.  Let us calculate this in our case (cf.~\cite[Ex.~4.4, Ch.~V]{silverman}). Set \(m=(p-1)/2\). By the binomial theorem,
	\[
	(4\xi^3+1)^m
	=
	\sum_{j=0}^{m}\binom{m}{j}4^j \xi^{3j}.
	\]
The monomial \(\xi^{p-1}\) occurs in this expression if and only if $3j=p-1$ for some~$j$, and thus the coefficient of $\xi^{p-1}$  is zero if \(p\equiv 2 \pmod{3}\); if \(p\equiv 1 \pmod{3}\), the   coefficient of  $\xi^{p-1}$ in the above expression,
	\[\binom{\frac{p-1}{2}} {\frac{p-1}{3}}\,4^{\frac{p-1}{3}},\]
does not vanish. \end{proof}

The condition in Proposition~\ref{supersingularFp} is also the one behind the irrationality of the third roots of unity in \(k\):

\begin{lemma}\label{winfp}
The Galois field \(\FF_p\) has a primitive cubic root of unity if and only if $p\equiv 1 \mod 3$. \end{lemma}

\begin{proof} A primitive cubic root of unity is an element of order three in the multiplicative group of the field; since  \(\FF_p^*\) has order \(p-1\), by Cauchy's theorem, it has an element of order three if and only if  \(p-1\) is divisible by~\(3\). \end{proof}

\subsection{The family $\mathcal{F}^4$} In this section we prove Theorem~\ref{thm:charp}. We work over an algebraically closed field \(k\) of characteristic \(p>3\).  Since \(p\neq 3\), \(k\) has primitive cubic roots of unity. Let $\mathcal{H} \subset \PP_{k}^2$ be the union of lines of the dual Hesse configuration~(\ref{arr:hesse}). As in the complex case, the most general foliation of degree four of \(\PP^2_k\) that is tangent to \(\mathcal{H}\) is, for \(\alpha\in \PP^1_k\), the foliation  \(\mathcal{F}^4_\alpha\)  defined by the form of Eq.~(\ref{eq:f4linsneto}). The description of   \(\mathcal{F}^4\) of Theorem~\ref{thm:ln-desc}  carries over to this setting:

\begin{theorem}\label{sec:main}
Let \(k\) be an algebraically closed field of characteristic \(p>3\), and let \(E\) be the elliptic curve defined locally by  \(\zeta^2=4\xi^3+1\), with the derivation \(D_E\)   for which \(D_E(\xi)=\zeta\). Let \(D_0\) and \(D_1\) be the derivations on \(E\times E\) induced by \(D_E\) on the first and second factors.  For \(\lambda\in k\), let \(\mathcal G_\lambda\) be the foliation on \(E\times E\) generated by \(D_0+\lambda D_1\).

For \(\alpha\in k\), let \(\lambda\in k\) be given by (\ref{eq:slope}) for a primitive cubic root of unity \(\omega\in k\). Then, \(\mathcal F^4_\alpha\) is birationally equivalent to the  quotient  of \(\mathcal G_\lambda\)  under  the action of the transformation of order three given by the square of (\ref{rho-alg}). In particular, \(\mathcal F^4_\alpha\) is algebraically integrable if and only if so is \(\mathcal G_\lambda\).
\end{theorem}

This result can be established through a careful reading of the proof of Theorem~\ref{thm:ln-desc}, which applies  essentially word by word to this setting, as all the objects in that proof, being  defined over \(\ZZ[\omega,1/3]\), are well-defined over \(k\), and all the  algebraic and differential relations among them are still valid. For instance, from (\ref{z0z1linind}), since \(1+2\omega\neq 0\) (for \((1+2\omega)^2=-3\)), the derivations \(Z_0\) and \(Z_1\) are linearly independent at every point of~\(\Sigma_4'\); the quotient under the order-nine map \(\eta\) of Eq.~(\ref{quot:ordernine}) is étale, as its order and  \(\mathrm{char}(k)\) are relatively prime; the map \(F\) is étale, as it maps   \(Z_0\) and \(Z_1\) to the everywhere linearly independent derivations \(D_0\) and~\(D_1\); and so on. The formula of Eq.~(\ref{bir_inverse_of_F}) for the birational inverse of \(F\) is still valid.

Observe that, in this proof, \(F\) is étale, as it maps the  derivations \(Z_0\) and \(Z_1\) to the everywhere linearly independent derivations \(D_0\) and \(D_1\). The formula of Eq.~(\ref{bir_inverse_of_F}) for the birational inverse of \(F\) is still valid.

The transformations of Theorem~\ref{thm:bir-f4} give birational transformations of \(\PP^2_k\) that act upon the foliations as described by formula~(\ref{act.lins.par}). Theorem~\ref{sec:main} implies that these   are induced by symmetries of \(E\times E\).

\begin{proof}[Proof of Theorem~\ref{thm:charp}] The foliation $\mathcal{F}_{\infty}^4$ is algebraically integrable since it has the rational first integral $(y^3-z^3)/(x^3-z^3)$, and this case may be left aside.
By Theorem~\ref{sec:main}, \(\mathcal{F}^4_\alpha\) is algebraically integrable if and only if so is the corresponding  \(\mathcal{G}_\lambda\). Since  algebraic integrability is equivalent to \(p\)-closedness, it suffices to determine the conditions under which \(\mathcal{G}_\lambda\) is \(p\)-closed for \(\lambda\in k\). We have that \(D_E^{[p]}=a_E D_E\) for some \(a_E\in k\), and thus $D_0^{[p]}=a_E D_0$, and $D_1^{[p]}=a_E D_1$.
Since \([D_0,D_1]=0\), and since, for \(0<i<p\),  
\(\binom{p}{i}\equiv 0 \pmod p\),  we have
\begin{equation}\label{dlp}D_\lambda^{[p]}=
(D_0+\lambda D_1)^{[p]}
=
\sum_{i=0}^{p}
\binom{p}{i}
\lambda^iD_0^{[p-i]}D_1^{[i]}=D_0^{[p]}+\lambda^pD_1^{[p]}= a_E(D_0+\lambda^pD_1).
\end{equation}
Consequently,
\[
    \begin{aligned}
    D_\lambda\wedge D_\lambda^{[p]}
    &=
    a_E(D_0+\lambda D_1)
    \wedge
    (D_0+\lambda^pD_1)  =
    a_E(\lambda^p-\lambda)D_0\wedge D_1.
    \end{aligned}
\]
Since \(D_0\wedge D_1\) does not vanish, \(\mathcal{G}_\lambda\) is \(p\)-closed if and only if
$a_E=0$ or $\lambda\in\FF_p$. Thus, with Proposition~\ref{supersingularFp}:
\begin{itemize}
\item if \(p\equiv2\pmod{3}\), \(a_E=0\), and  \(\mathcal{F}^4_\alpha\) is \(p\)-closed (independently of the value of \(\alpha\)); and
\item if $p\equiv1\pmod{3}$, \(a_E\neq 0\); since \(\omega\in\FF_p\) (Lemma~\ref{winfp}), from relation (\ref{eq:slope}), \(\lambda\in \FF_p\) if and only if \(\alpha\in\FF_p\), and \(\mathcal{F}^4_\alpha\) is \(p\)-closed if and only if \(\alpha\in\FF_p\). 
\end{itemize}
This establishes the result.
 \end{proof}

In the second case, explicit first integrals for the corresponding foliations may be constructed  in the spirit of Section~\ref{sec:algo}: the transformation \(U_4\) will act transitively on the slopes in \(\FF_p^*\cup\{\infty\}\), and the first integral of \(\mathcal{F}^4_\infty\) may be pulled back by a convenient power of \(U_4\) in order to give the first integral of a given foliation with such a slope.
 
Even in the case where there is no rational first integral, we can describe all the invariant algebraic curves of a foliation in \(\mathcal{F}^4\) (compare with Corollary~\ref{int:cond}):
\begin{theorem}\label{thm:p-div} If $p\equiv 1 \pmod 3$ and $\alpha \notin \FF_p$,  the $p$-divisor of $\mathcal{F}_\alpha^4$ is
\begin{equation}\label{p-divisor-f4}
    \Delta_{\mathcal{F}_\alpha^4} = \frac{p+2}{3}\mathcal{H}.
\end{equation}
In particular, the only irreducible invariant curves of $\mathcal{F}_\alpha^4$ are the components of $\mathcal{H}$.  
\end{theorem}
\begin{proof} By Theorem~\ref{sec:main}, the vector fields \(Z_i\) and \(D_i\) are \(F\)-related.
Since \(D_0\) and \(D_1\) are linearly independent at every point of
\(E\times E\), the vector fields \(Z_0\) and \(Z_1\) are linearly independent
at every point of \(\Sigma_4\). In particular,
$Z_0\wedge Z_1$ does not vanish on \(\Sigma_4\).

Since \(p\equiv1\pmod 3\), \(\omega\in \FF_p\) (Lemma~\ref{winfp}). Since  \(\alpha\notin\FF_p\) and is \(\lambda\) given by (\ref{eq:slope}), \(\lambda\notin\FF_p\). From (\ref{dlp}), since \(F\) is étale and \(Z_i\) is \(F\)-related to
\(D_i\), it follows that  \(Z_\lambda^{[p]}
    =
    a_E (Z_0+\lambda^pZ_1 )\), 
and  \(Z_\lambda\wedge Z_\lambda^{[p]} = a_E(\lambda^p-\lambda)Z_0\wedge Z_1\). Since \(a_E\neq0\), \(\lambda^p-\lambda\neq0\), and \(Z_0\wedge Z_1\) is nowhere vanishing,   we conclude that \(Z_\lambda\wedge Z_\lambda^{[p]}\) does not vanish on~\(\Sigma_4\). Thus the, \(p\)-divisor of the pull-back
foliation \(\pi^*\mathcal F^4_\alpha\) is empty.

Let \(U=\PP^2_k\setminus\{H=0\}\). Since  \(\pi\colon\Sigma_4\to U\)
is finite étale, Proposition~\ref{sec:etalepdivisor} gives $s_{\pi^*\mathcal F^4_\alpha}
=
\pi^*(s_{\mathcal F^4_\alpha}|_U)$. Hence, $s_{\mathcal F^4_\alpha}|_U$ has no zeros, and
    \(\operatorname{Supp}(\Delta_{\mathcal F^4_\alpha})
    \subseteq \{H=0\}\).

The group of order nine of linear transformations of \(\PP^2_k\) generated by the commuting maps of order three $L_4$ and $L_4'$ (as given in the statement of Theorem~\ref{thm:bir-f4})  preserves each
foliation \(\mathcal F^4_\alpha\) individually. Moreover, it acts
transitively on the nine lines of the dual Hesse configuration:  \(L_4\) permutes the three blocks of three lines each given by the cubic factors of~(\ref{arr:hesse}), whereas \(L_4'\)  preserves each block while permuting the three lines contained in it. Hence, by Corollary~\ref{cor:allequal}, \(\Delta_{\mathcal F^4_\alpha}=m\mathcal H\) for some  integer \(m\geq1\). Since $\deg(\mathcal H)=9$, and, by Remark~\ref{pdivP2}, $\deg(\Delta_{\mathcal F_\alpha^4})=3p+6$, we have that $m= \frac{1}{3}(p+2)$,   establishing (\ref{p-divisor-f4}).

It remains to identify the invariant curves. By the basic property of the \(p\)-divisor, every irreducible algebraic curve invariant by \(\mathcal F_\alpha^4\) must occur in the support of \(\Delta_{\mathcal F_\alpha^4}\). Hence, any such curve is one of the irreducible components of \(\mathcal H\). Conversely, the definition, \(\mathcal F_\alpha^4\) is tangent to the dual Hesse configuration \(\mathcal H\). Therefore, the irreducible invariant algebraic  curves of \(\mathcal F_\alpha^4\) are precisely the nine components of~\(\mathcal H\).
\end{proof}

\begin{remark}[In characteristics $2$ and $3$] From Eq.~(\ref{eq:f4linsneto}) we have that, in the affine chart $z=1$, for \(\alpha\neq \infty\), the foliation \(\mathcal{F}^4_\alpha\) is generated by
\[v_\alpha=(x^3-1)(x-\alpha y^2)\del{x}+(y^3-1)(y-\alpha x^2)\del{y} ,
\]
while \(\mathcal{F}^4_\infty\) is given by \[v_\infty=(x^3-1)y^2\del{x}+(y^3-1)x^2\del{y}.\]
	
In characteristic $2$, every one of the foliations is $2$-closed, since
\begin{align*}
	v_\alpha^{[2]} & =(1+\alpha x^2y^2)v_\alpha, 
 & 
	v_\infty^{[2]} & =x^2y^2\,v_\infty.
\end{align*}

In characteristic $3$, only one member of the family is $3$-closed. Setting $Q=xy+xz+yz$ and $L=x+y+z$, the homogenized coefficient of $v_\alpha\wedge v_\alpha^{[3]}$ is
\[(\alpha-1)^3(x-y)^3(x-z)^3(y-z)^3[\alpha Q^3+xyz L^3],\]
with the term in the square bracket becoming $Q^3$ for $\alpha=\infty$. The cubes reflect the degeneration $x^3-z^3=(x-z)^3$ of the dual Hesse arrangement in characteristic $3$. Since $\alpha^3-1=(\alpha-1)^3$, the foliation $\mathcal F^4_\alpha$ is $3$-closed if and only if $\alpha=1$. In particular, $\mathcal F^4_\infty$ is not $3$-closed and
\[\Delta_{\mathcal F^4_\infty}=3\{x=y\}+3\{x=z\}+3\{y=z\}+3\{Q=0\}.\]
This is the only characteristic in which $\mathcal F^4_\infty$ fails to be $p$-closed.
\end{remark}

\subsection{The \texorpdfstring{\(p\)}{p}-divisor for foliations in the other families}
The families of foliations of degrees two and three may be defined in the same way as in the complex case, through the rational maps described in Sections~\ref{sec:deg3} and~\ref{sec:deg2}. After restricting to suitable open subsets, the family of degree four  is
obtained from the family of degree three  by an étale pull-back, whereas the  family of degree two is birationally identified with the degree-three family
through a Cremona transformation, and the criterion established for the family of degree four also detects the \(p\)-closedness of these families.
The ramification and exceptional divisors of the corresponding maps account
for the different multiplicities that appear in their \(p\)-divisors.

Assume that \(p>3\), \(p\equiv1\pmod{3}\), and that \(\alpha\notin \FF_p\). Then the foliations \(\mathcal F^3_\alpha\) and \(\mathcal F^2_\alpha\) are not \(p\)-closed, and their \(p\)-divisors are given by
\[
\Delta_{\mathcal F^3_\alpha}
=
\frac{p+5}{6}C_1
+
\frac{p+2}{3}
 (C_2+l_0+l_1+l_2),
\]
for the invariant curves given by the Eqs.~(\ref{triangle_f3}) and~(\ref{conics_f3}), and
\[
\Delta_{\mathcal F^2_\alpha}
=
\frac{p+5}{6}Q
+
\frac{p+2}{3}L,
\]
where  \(Q\) is the quartic of Eq.~(\ref{tricuspid}), and \(L\) is  \(z=0\), the only bitangent of the latter.

\subsection{Reduction modulo \texorpdfstring{\(p\)}{p}}

Let \(\mathcal F\) be a foliation on \(\PP^2_{\CC}\) of degree \(d\), defined by a homogeneous projective \(1\)-form 
    \(\omega=A\,\dd x+B\,\dd y+C\,\dd z\), 
    \(A,B,C\in\CC[x,y,z]_{d+1}\).
Choose a finitely generated \(\ZZ\)-subalgebra $R\subset\CC$ containing the coefficients of \(A\), \(B\), and \(C\). After replacing \(R\) by a localization, we may assume that \(\omega\) extends to a relative projective
\(1\)-form,
\[
   \omega_R
    \in
    H^0\!\left(
        \PP^2_R,
        \Omega^1_{\PP^2_R/R}(d+2)
    \right),
\]
whose reduction at every closed point of \(\operatorname{Spec}(R)\) defines a
foliation of degree \(d\).

For a closed point \(\mathfrak p\in\operatorname{Spec}(R)\), let $ \kappa(\mathfrak p)=R/\mathfrak p$
be its finite residue field, and let $\ell=\operatorname{char}\kappa(\mathfrak p)$ be its residue characteristic. The reduction of \(\omega_R\) modulo
\(\mathfrak p\) defines a foliation $\mathcal F_{\mathfrak p}$ on \(\PP^2_{\kappa(\mathfrak p)}\subset \PP^2_{\overline{\kappa(\mathfrak p)}}\). We   say that
\(\mathcal F_{\mathfrak p}\) is \emph{\(\ell\)-closed} if its base change to an
algebraic closure of \(\kappa(\mathfrak p)\) is \(\ell\)-closed.

The following result shows that \(p\)-closedness for infinitely many reductions is substantially weaker than algebraic integrability in characteristic zero.

\begin{theorem}\label{thm:infinitely-many-p-closed-reductions}
Let \(\alpha\in \overline{\QQ}\setminus \QQ(\omega)\). Consider the complex foliation \(\mathcal{F}_\alpha^4\). For
every  closed point \(\mathfrak p\)   over 
\(R=\ZZ[\alpha]\) whose residue characteristic \(\ell\) is such that  
    \(\ell>3\) and
    \(\ell\equiv2\pmod 3\), 
the reduction modulo $\mathfrak p$ of the complex foliation \(\mathcal{F}_\alpha^4\) is \(\ell\)-closed. Moreover, the set of such closed points \(\mathfrak p\) is Zariski-dense in \(\operatorname{Spec}(R)\).
\end{theorem}

\begin{proof}
Since \(\alpha\notin \QQ(\omega)\),  by Corollary~\ref{int:cond}, the complex foliation \(\mathcal F^4_\alpha\) is not algebraically integrable. Let $R=\ZZ[\alpha]$, and let \(\mathfrak p\in\operatorname{Spec}(R)\) be a closed point. After base change to \(\overline{\kappa(\mathfrak p)}\), the reduction of \(\mathcal F^4_\alpha\) is a member of the degree-four family in
characteristic \(\ell\). Hence, by Theorem~\ref{thm:charp} and the hypothesis on \(\ell\), its reduction $\mathcal F^4_{\alpha, \mathfrak p }$ is \(\ell\)-closed.

It remains to prove the existence of infinitely many such closed points. Since \(R=\ZZ[\alpha]\) is a finitely generated integral
\(\ZZ\)-algebra of characteristic zero, the structural morphism \(\pi:\operatorname{Spec}(R)\to\operatorname{Spec}(\ZZ)\),  \(\mathfrak p \mapsto \mathfrak p \cap \ZZ\), is dominant.  By Chevalley's theorem \cite[Thm.~6]{matsumura}, the set
$\pi(\operatorname{Spec}(R))$ is constructible and contains a non-empty Zariski open subset of $\operatorname{Spec}(\ZZ)$. Therefore, there exists an integer \(N>0\) such that $\operatorname{Spec}(\ZZ[1/N])
    \subseteq
    \pi(\operatorname{Spec}(R))$. Consequently, for every prime \(\ell\nmid N\), the fiber $\pi^{-1}(\ell) = \operatorname{Spec}(R\otimes\FF_\ell)$ is nonempty. Since this fiber is of finite type over \(\FF_\ell\), it
has a closed point: this gives a closed point
\(\mathfrak p\in\operatorname{Spec}(R)\) with residue characteristic~\(\ell\). 

Finally, Dirichlet's theorem (see \cite[Ch.~VII]{MR1193029}) yields infinitely many primes \(\ell\equiv2\pmod3\). After discarding the finitely many primes dividing \(N\) and those smaller than \(5\), each such prime occurs as the residue characteristic of a closed point \(\mathfrak p\) for which \(\mathcal F^4_{\alpha,\mathfrak p}\) is \(\ell\)-closed. Since \(\alpha\) is algebraic over \(\QQ\), the ring \(R=\ZZ[\alpha]\) has Krull dimension one; hence, any infinite set of closed points is Zariski-dense in \(\operatorname{Spec}(R)\).
\end{proof}

The previous theorem is related to the following conjecture, due to Ekedahl, Shepherd-Barron, and Taylor~\cite[Conjecture~F]{EST}:

\begin{conjecture}\label{conj:ESBT}
Let $\mathcal F$ be a holomorphic foliation on a complex projective
manifold $X$ with an integral model $(\mathcal X,\mathcal F)$ defined
over a finitely generated $\ZZ$-algebra $R$. The foliation $\mathcal F$ is algebraically integrable if  and only if $\mathcal F_{\mathfrak p}$ is $p$-closed for every closed point  $\mathfrak p$ in a non-empty Zariski-open subset of $\operatorname{Spec}(R)$.
\end{conjecture}

\begin{remark} Theorem~\ref{thm:infinitely-many-p-closed-reductions} shows that the ``non-empty Zariski-open subset'' condition in Conjecture~\ref{conj:ESBT}
cannot be replaced by the weaker requirement that the reduction be
$p$-closed on a Zariski-dense subset of closed points of $\operatorname{Spec}(R)$. Indeed, if \(\alpha\in \overline{\QQ}\setminus \QQ(\omega)\), then $\mathcal F^4_\alpha$ is not algebraically integrable in characteristic zero, whereas its reductions are $p$-closed on a
Zariskid-ense set of closed points, namely at suitable reductions of residue characteristic $p\equiv2\pmod 3$. Thus, $p$-closedness on a Zariski-dense set of reductions is strictly weaker than the arithmetic condition appearing in Conjecture~\ref{conj:ESBT}. Examples of regular foliations on surfaces whose reductions modulo $p$ are $p$-closed for infinitely many primes and non-$p$-closed for infinitely many primes were already constructed in~\cite{EST}. Our examples are of a substantially different nature since, being foliations on \(\PP^2_{\CC}\), they have singularities.
\end{remark}
 
\begin{remark}
The phenomenon described in Theorem~\ref{thm:infinitely-many-p-closed-reductions} also occurs for elementary	logarithmic foliations. For simplicity, let \( F_0,F_1,F_2\in \ZZ[x,y,z] \) be generic homogeneous polynomials of the same degree, and let \(\alpha\in\CC\setminus\QQ\). Consider the logarithmic foliation \(\mathcal{G}_{\alpha}\) on \(\PP^2_{\CC}\) defined by
\[
	\omega_\alpha
	=
	F_1F_2\,\dd F_0
	-\alpha F_0F_2\,\dd F_1
	+(\alpha-1)F_0F_1\,\dd F_2.
\]
Let \(D_j=\{F_j=0\}\), and choose an intersection point \(q\in D_0\cap D_1\setminus D_2.\) Taking \(x=F_0\) and \(y=F_1\) as local coordinates at \(q\), and writing \(u=F_2\), which is a unit at \(q\), a holomorphic \(1\)-form defining the foliation locally is
	\[
	\Omega = yu\,\dd x-\alpha xu\, \dd y+(\alpha-1)xy\,\dd u.
	\]
Its linear part at \(q\) is, up to multiplication by a nonzero scalar, \(y\,\dd x-\alpha x\,\dd y.\) Thus, the quotient of the eigenvalues of the singularity at \(q\) is \(\alpha\). A nondegenerate singularity admitting a nonconstant meromorphic first integral must be resonant, and in particular, the quotient of its eigenvalues must be rational. Since \(\alpha\notin\QQ\), the foliation \(\mathcal{G}_{\alpha}\) is not algebraically integrable.
	
Now take \(\alpha=\ii =\sqrt{-1}\). For every prime \(p\equiv 1 \pmod 4\) there exists \(a\in\FF_p\) such that \(a^2=-1\). The reduction modulo \( p\) of \(\omega_\ii\) is
	\[
	\omega_{\ii,p}
	=
	\overline{F_1}\overline{F_2}\,\dd\overline{F_0}
	-a\overline{F_0}\overline{F_2}\,\dd\overline{F_1}
	+(a-1)\overline{F_0}\overline{F_1}\,\dd\overline{F_2}.
	\]
Choosing an integer representative of \(a\), we obtain
	\[ H
	=\frac{\overline{F_0} \cdot \overline{F_2}^{a-1}}{\overline{F_1}^{a}}
	\]
as a rational first integral for \(\mathcal{G}_{\ii,\mathfrak p}\). Hence, \(\mathcal{G}_{\ii,\mathfrak p}\) admits a rational first integral and is therefore \(p\)-closed. Since, by Dirichlet's theorem, there are infinitely many primes \(p\) such that \(p\equiv 1\pmod 4\), the complex foliation \(\mathcal{G}_{\ii}\), although not algebraically integrable, has a \(p\)-closed reduction for infinitely many primes. Proposition \ref{thm:infinitely-many-p-closed-reductions} provides a non-logarithmic instance of the same arithmetic phenomenon, in which \(p\)-closedness is instead governed by the supersingularness  of the underlying elliptic curve.
\end{remark}

\providecommand{\bysame}{\leavevmode\hbox to3em{\hrulefill}\thinspace}
\providecommand{\noopsort}[1]{}
\providecommand{\mr}[1]{\href{http://www.ams.org/mathscinet-getitem?mr=#1}{MR~#1}}
\providecommand{\zbl}[1]{\href{http://www.zentralblatt-math.org/zmath/en/search/?q=an:#1}{Zbl~#1}}
\providecommand{\jfm}[1]{\href{http://www.emis.de/cgi-bin/JFM-item?#1}{JFM~#1}}
\providecommand{\arxiv}[1]{\href{http://www.arxiv.org/abs/#1}{arXiv~#1}}
\providecommand{\doi}[1]{\url{https://doi.org/#1}}
\providecommand{\MR}{\relax\ifhmode\unskip\space\fi MR }
\providecommand{\MRhref}[2]{%
	\href{http://www.ams.org/mathscinet-getitem?mr=#1}{#2}
}
\providecommand{\href}[2]{#2}


\begin{thebibliography}{BHPV04}


    
	\bibitem[Ahl78]{ahlfors}
	\bgroup\scshape{}L.~V. Ahlfors\egroup{}, \emph{Complex analysis}, third ed.,
	\emph{International Series in Pure and Applied Mathematics}, McGraw-Hill Book
	Co., New York, 1978.
	
	\bibitem[BHPV04]{BPHV}
	\bgroup\scshape{}W.~P. Barth\egroup{}, \bgroup\scshape{}K.~Hulek\egroup{},
	\bgroup\scshape{}C.~A.~M. Peters\egroup{}, and \bgroup\scshape{}A.~Van~de
	Ven\egroup{}, \emph{Compact complex surfaces}, second ed., \emph{Ergebnisse
		der Mathematik und ihrer Grenzgebiete. 3. Folge. A Series of Modern Surveys
		in Mathematics} \textbf{4}, Springer-Verlag, Berlin, 2004. 
	\doi{10.1007/978-3-642-57739-0}.
	
	\bibitem[Bia91]{bianchi-zahlen}
	\bgroup\scshape{}L.~Bianchi\egroup{}, Geometrische {D}arstellung der {G}ruppen
	linearer {S}ubstitutionen mit ganzen complexen {C}oefficienten nebst
	{A}nwendungen auf die {Z}ahlentheorie,  \emph{Math. Ann.} \textbf{38} (1891),
	313--333. \doi{10.1007/BF01199425}.
	
	\bibitem[Bia92]{bianchi-gruppi}
	\bgroup\scshape{}L.~Bianchi\egroup{}, Sui gruppi di sostituzioni lineari con
	coefficienti appartenenti a corpi quadratici immaginarî,  \emph{Math. Ann.}
	\textbf{40} (1892), 332--412. \doi{10.1007/BF01443558}.
	
	\bibitem[BK86]{PAC}
	\bgroup\scshape{}E.~Brieskorn\egroup{} and
	\bgroup\scshape{}H.~Kn\"orrer\egroup{}, \emph{Plane algebraic curves},
	Birkh\"auser Verlag, Basel, 1986.  \doi{10.1007/978-3-0348-5097-1}.
	
	\bibitem[Bru99]{brunella-entire}
	\bgroup\scshape{}M.~Brunella\egroup{}, Courbes enti\`eres et feuilletages
	holomorphes,  \emph{Enseign. Math. (2)} \textbf{45} no.~1-2 (1999), 195--216.  \doi{https://doi.org/10.5169/seals-64446}.
	
	\bibitem[Bru15]{brunella-birational}
	\bgroup\scshape{}M.~Brunella\egroup{}, \emph{Birational geometry of
		foliations}, \emph{IMPA Monographs} \textbf{1}, Springer, Cham, 2015.  \doi{10.1007/978-3-319-14310-1}.
	
	\bibitem[CF03]{cantat-favre}
	\bgroup\scshape{}S.~Cantat\egroup{} and \bgroup\scshape{}C.~Favre\egroup{},
	Sym\'etries birationnelles des surfaces feuillet\'ees,  \emph{J. Reine Angew.
		Math.} \textbf{561} (2003), 199--235.  
	\doi{10.1515/crll.2003.066}.
	
	\bibitem[ESBT99]{EST}
	\bgroup\scshape{}T.~Ekedahl\egroup{}, \bgroup\scshape{}N.~I.
	Shepherd-Barron\egroup{}, and \bgroup\scshape{}R.~L. Taylor\egroup{}, \emph{A
		conjecture on the existence of compact leaves of algebraic foliations},
	preprint, 1999.
	
	\bibitem[Fin89]{fine}
	\bgroup\scshape{}B.~Fine\egroup{}, \emph{Algebraic theory of the {B}ianchi
		groups}, \emph{Monographs and Textbooks in Pure and Applied Mathematics}
	\textbf{129}, Marcel Dekker, Inc., New York, 1989.  
	
	\bibitem[Fuj88]{fujiki}
	\bgroup\scshape{}A.~Fujiki\egroup{}, Finite automorphism groups of complex tori
	of dimension two,  \emph{Publ. Res. Inst. Math. Sci.} \textbf{24} no.~1
	(1988), 1--97.   \doi{10.2977/prims/1195175326}.
	
	\bibitem[Gui02]{guillot-exemplesLN}
	\bgroup\scshape{}A.~Guillot\egroup{}, Sur les exemples de {L}ins {N}eto de
	feuilletages alg\'{e}briques,  \emph{C. R. Math. Acad. Sci. Paris}
	\textbf{334} no.~9 (2002), 747--750.  
	\doi{10.1016/S1631-073X(02)02343-9}.
	
	\bibitem[Gui06]{guillot-quadratic}
	\bgroup\scshape{}A.~Guillot\egroup{}, Semicompleteness of homogeneous quadratic
	vector fields,  \emph{Ann. Inst. Fourier (Grenoble)} \textbf{56} no.~5
	(2006), 1583--1615.    \doi{10.5802/aif.2221}.
	
 
	
	\bibitem[Hoc55]{Hochschild1955}
	\bgroup\scshape{}G.~Hochschild\egroup{}, Simple algebras with purely
	inseparable splitting fields of exponent~{$1$},  \emph{Trans. Amer. Math.
		Soc.} \textbf{79} (1955), 477--489.   \doi{10.2307/1993043}.


	\bibitem[JIW99]{Johnson-Ivic}
	\bgroup\scshape{}N.~W. Johnson\egroup{} and
	\bgroup\scshape{}A.~Ivi\'c~Weiss\egroup{}, Quadratic integers and {C}oxeter
	groups,  \emph{Canad. J. Math.} \textbf{51} no.~6 (1999), 1307--1336.  \doi{10.4153/CJM-1999-060-6}.
	
	\bibitem[Kna92]{MR1193029}
	\bgroup\scshape{}A.~W. Knapp\egroup{}, \emph{Elliptic curves},
	\emph{Mathematical Notes} \textbf{40}, Princeton University Press, Princeton,
	NJ, 1992.  
	
	\bibitem[LLT23]{linglu}
	\bgroup\scshape{}H.~Ling\egroup{}, \bgroup\scshape{}J.~Lu\egroup{}, and
	\bgroup\scshape{}S.-L. Tan\egroup{}, The birational invariants of {L}ins
	{N}eto's foliations,  \emph{Proc. Amer. Math. Soc.} \textbf{151} no.~12
	(2023), 5223--5238.    \doi{10.1090/proc/16401}.
	
 	
	\bibitem[LN02]{linsneto}
	\bgroup\scshape{}A.~Lins~Neto\egroup{}, Some examples for the {P}oincar\'{e}
	and {P}ainlev\'{e} problems,  \emph{Ann. Sci. \'{E}cole Norm. Sup. (4)}
	\textbf{35} no.~2 (2002), 231--266. 
	\doi{10.1016/S0012-9593(02)01089-3}.
	
	\bibitem[LN04]{lins_neto-pencils}
	\bgroup\scshape{}A.~Lins~Neto\egroup{}, Curvature of pencils of foliations,
	\emph{Ast\'erisque} no.~296 (2004), 167--190, Analyse complexe, syst\`emes
	dynamiques, sommabilit\'e{} des s\'eries divergentes et th\'eories
	galoisiennes. I.   Available at
	\url{https://www.numdam.org/item/AST_2004__296__167_0/}.
	
	\bibitem[Liu02]{MR1917232}
	\bgroup\scshape{}Q.~Liu\egroup{}, \emph{Algebraic geometry and arithmetic
		curves}, \emph{Oxford Graduate Texts in Mathematics} \textbf{6}, Oxford
	University Press, Oxford, 2002.

\bibitem[Mat80]{matsumura}
\bgroup\scshape{}H.~Matsumura\egroup{}, \emph{Commutative algebra}, second ed.,
\emph{Mathematics Lecture Note Series} \textbf{56}, Benjamin/Cummings
Publishing Co., Inc., Reading, MA, 1980.

	
	\bibitem[McQ01]{mcquillan}
	\bgroup\scshape{}M.~McQuillan\egroup{}, \emph{Non-commutative {M}ori theory},
	preprint IHES/M/01/07, 2001.
	
	\bibitem[MP24]{MP-negative}
	\bgroup\scshape{}L.~G. Mendes\egroup{} and
	\bgroup\scshape{}L.~Puchuri\egroup{}, Negative curves and elliptic fibrations
	on a special rational surface,  (2024), preprint arXiv:2409.19235. Available
	at \url{https://arxiv.org/abs/2409.19235}.
	
	\bibitem[MP26]{mendes-puchuri-effective}
	\bgroup\scshape{}L.~G. Mendes\egroup{} and
	\bgroup\scshape{}L.~Puchuri\egroup{}, Effective integrability of {L}ins
	{N}eto's family of foliations,  \emph{Hokkaido Math. J.} \textbf{2} (2026),
	229--255. \doi{10.14492/hokmj/2024-887}.
	
	\bibitem[Men23]{mendson2022foliations}
	\bgroup\scshape{}W.~Mendson\egroup{}, Foliations on smooth algebraic surfaces
	in positive characteristic,  \emph{J. Pure Appl. Algebra} \textbf{227} no.~9
	(2023), Paper No. 107379.  \doi{10.1016/j.jpaa.2023.107379}.
	
	\bibitem[Mil06]{milnor-dynamics}
	\bgroup\scshape{}J.~Milnor\egroup{}, \emph{Dynamics in one complex variable},
	third ed., \emph{Annals of Mathematics Studies} \textbf{160}, Princeton
	University Press, Princeton, NJ, 2006.  
	
	\bibitem[MP97]{MR1468476}
	\bgroup\scshape{}Y.~Miyaoka\egroup{} and
	\bgroup\scshape{}T.~Peternell\egroup{}, \emph{Geometry of higher-dimensional
		algebraic varieties}, \emph{DMV Seminar} \textbf{26}, Birkh\"{a}user Verlag,
	Basel, 1997.  \doi{10.1007/978-3-0348-8893-6}.
	
	\bibitem[Mum08]{MumfordAbelian74}
	\bgroup\scshape{}D.~Mumford\egroup{}, \emph{Abelian varieties},  Tata
		Institute of Fundamental Research Studies in Mathematics~\textbf{5}, Tata
	Institute of Fundamental Research, Bombay, 2008,  Corrected
	reprint of the second (1974) edition.  
	
	\bibitem[Poi91]{poincare-palermo1}
	\bgroup\scshape{}H.~Poincar\'e\egroup{}, {Sur l'int\'egration alg\'ebrique des
		\'equations diff\'erentielles du premier ordre et du premier degr\'e},
	\emph{Rend. Circ. Mat. Palermo} \textbf{5} (1891), 161--191.
	
	\bibitem[Poi97]{poincare-palermo2}
	\bgroup\scshape{}H.~Poincar\'e\egroup{}, {Sur l'int\'egration alg\'ebrique des
		\'equations diff\'erentielles du premier ordre et du premier degr\'e},
	\emph{Rend. Circ. Mat. Palermo} \textbf{11} (1897), 193--239.
	
	\bibitem[PM13]{puchuri}
	\bgroup\scshape{}L.~Puchuri~Medina\egroup{}, Degree of the first integral of a
	pencil in {$\mathbb{P}^2$} defined by {L}ins {N}eto,  \emph{Publ. Mat.}
	\textbf{57} no.~1 (2013), 123--137. 
	\doi{10.5565/PUBLMAT\_57113\_05}.
	
	\bibitem[Sil09]{silverman}
	\bgroup\scshape{}J.~H. Silverman\egroup{}, \emph{The arithmetic of elliptic
		curves}, second ed., \emph{Graduate Texts in Mathematics} \textbf{106},
	Springer, Dordrecht, 2009.  \doi{10.1007/978-0-387-09494-6}.
	
	\bibitem[Swa71]{Swan}
	\bgroup\scshape{}R.~G. Swan\egroup{}, Generators and relations for certain
	special linear groups,  \emph{Advances in Math.} \textbf{6} (1971), 1--77.  \doi{10.1016/0001-8708(71)90027-2}.
	
	\bibitem[Wal04]{ctcwall}
	\bgroup\scshape{}C.~T.~C. Wall\egroup{}, \emph{Singular points of plane
		curves}, \emph{London Mathematical Society Student Texts} \textbf{63},
	Cambridge University Press, Cambridge, 2004.  
	\doi{10.1017/CBO9780511617560}.
	
\end{thebibliography}
\end{document}